\documentclass[twoside, leqno]{article}

\usepackage{amssymb}
\usepackage{amsmath}
\usepackage{amsthm}
\usepackage{color}
\usepackage{mathrsfs}

\allowdisplaybreaks

\def\rr{{\mathbb R}}
\def\rn{{\mathbb{R}^n}}
\def\urn{\mathbb{R}_+^{n+1}}
\def\zz{{\mathbb Z}}
\def\cc{{\mathbb C}}
\def\aa{{\mathbb A}}
\def\nn{{\mathbb N}}

\def\ca{{\mathcal A}}

\def\cd{{\mathcal D}}

\def\cg{{\mathcal G}}

\def\cm{{\mathcal M}}

\def\cp{{\mathcal P}}
\def\cq{{\mathcal Q}}

\def\cs{{\mathcal S}}

\def\fz{\infty }

\def\lz{\lambda}

\def\vz{\varphi}
\def\vez{\varepsilon}

\def\lf{\left}
\def\r{\right}

\def\la{\langle}
\def\ra{\rangle}
\def\hs{\hspace{0.25cm}}
\def\ls{\lesssim}
\def\gs{\gtrsim}

\def\ov{\overline}
\def\noz{\nonumber}
\def\wz{\widetilde}
\def\wh{\widehat}

\def\gfz{\genfrac{}{}{0pt}{}}

\def\loc{{\mathop\mathrm{\,loc\,}}}
\def\supp{\mathop\mathrm{\,supp\,}}

\def\vlp{{L^{p(\cdot)}(\rn)}}

\def\btlve{F_{p(\cdot),q(\cdot)}^{s(\cdot),\phi}(\rn)}
\def\beve{b_{p(\cdot),q(\cdot)}^{s(\cdot),\phi}(\rn)}
\def\bbeve{B_{p(\cdot),q(\cdot)}^{s(\cdot),\phi}(\rn)}

\newtheorem{theorem}{Theorem}[section]
\newtheorem{lemma}[theorem]{Lemma}
\newtheorem{corollary}[theorem]{Corollary}
\newtheorem{proposition}[theorem]{Proposition}

\theoremstyle{definition}
\newtheorem{remark}[theorem]{Remark}
\newtheorem{definition}[theorem]{Definition}
\renewcommand{\appendix}{\par
   \setcounter{section}{0}%
   \setcounter{subsection}{0}%
   \setcounter{subsubsection}{0}%
   \gdef\thesection{\@Alph\c@section}%
   \gdef\thesubsection{\@Alph\c@section.\@arabic\c@subsection}%
   \gdef\theHsection{\@Alph\c@section.}%
   \gdef\theHsubsection{\@Alph\c@section.\@arabic\c@subsection}%
   \csname appendixmore\endcsname
 }

\numberwithin{equation}{section}

\begin{document}

\arraycolsep=1pt

\title{\bf\Large Besov-Type Spaces with Variable Smoothness and Integrability
\footnotetext{\hspace{-0.35cm} 2010 {\it
Mathematics Subject Classification}. Primary 46E35;
Secondary 42B25, 42B35.
\endgraf {\it Key words and phrases.} Besov space, variable exponent,
Peetre maximal function, embedding, atom, trace.
\endgraf This project is supported by the National
Natural Science Foundation of China
(Grant Nos.~11171027, 11361020 \& 11471042),
the Specialized Research Fund for the Doctoral Program of Higher Education
of China (Grant No. 20120003110003) and the Fundamental Research
Funds for Central Universities of China (Grant No.~2012LYB26, 2013YB60 and 2014KJJCA10). }}
\author{Dachun Yang, Ciqiang Zhuo\footnote{Corresponding author}\ \ and Wen Yuan}
\date{}
\maketitle

\vspace{-0.8cm}

\begin{center}
\begin{minipage}{13cm}
{\small {\bf Abstract}\quad
In this article, the authors introduce
Besov-type spaces with variable smoothness
and integrability. The authors then establish their
characterizations, respectively, in terms of $\varphi$-transforms
in the sense of Frazier and Jawerth, smooth atoms or Peetre maximal functions,
as well as a Sobolev-type embedding.  As an application of their atomic
characterization, the authors obtain a trace theorem of these variable Besov-type spaces.
}
\end{minipage}
\end{center}

\vspace{-0.1cm}

\section{Introduction\label{s1}}
\hskip\parindent
Spaces of variable integrability, also known as variable exponent Lebesgue spaces
$L^{p(\cdot)}(\rn)$, can be traced back to Orlicz \cite{or31,or32},
and studied by Musielak \cite{ms83} and Nakano \cite{nak50,nak51}, but the modern
development started with the articles \cite{kr91} of Kov\'a\v{c}ik
and R\'akosn\'{\i}k as well as \cite{cruz03} of Cruz-Uribe and \cite{din04} of Diening.
The variable Lebesgue spaces have already widely used in
the study of harmonic analysis; see, for example,
\cite{cfbook,cdh11,cfn03,dhhm09,dhr11,ns12,ins14}. Apart from theoretical considerations,
such function spaces have interesting applications in fluid dynamics
 \cite{am02,rm00}, image processing \cite{clr06},
partial differential equations and variational calculus \cite{am05,fan07,hhl08,ohno09,su09}.

In recent years, function spaces with variable exponents attract many attentions,
especially based on classical Besov and Triebel-Lizorkin spaces (see Triebel's monographes
\cite{t83,t92,t06} for the history of these two spaces). When Leopold
\cite{leop891,leop892,leop91,leop99} and Leopold and Schrohe \cite{leops96}
studied pseudo-differential operators, they introduced
related Besov spaces with variable smoothness, $B_{p,p}^{s(\cdot)}(\rn)$, which
were further generalized to the case that $q\neq p$, including
$B_{p,q}^{s(\cdot)}(\rn)$ and $F_{p,q}^{s(\cdot)}(\rn)$,
by Besov \cite{besov99,besov03,besov05}.
Along a different line of study, Xu \cite{Xu081,Xu082}
studied Besov spaces $B_{p(\cdot),q}^{s}(\rn)$ and Triebel-Lizorkin spaces
$F_{p(\cdot),q}^{s}(\rn)$ with variable exponent $p(\cdot)$
but fixed $q$ and $s$.
As was well known from the trace theorem (see, for example, \cite[Theorem 11.1]{fj90})
and Sobolev-type embeddings (see, for example, \cite[Theorem 2.7.1]{t83})
of classical function spaces, the smoothness and the integrability
often interact each other. However, the unification of both trace theorems
and Sobolev-type embeddings does not hold true on function spaces
with only one variable index; for example, the trace space of Sobolev space
$W^{k,p(\cdot)}$ is no longer a space of the same type (see \cite{dhr11}).
Thus, function spaces with full ranges of variable smoothness and
variable integrability are needed.

The concept of function spaces with variable smoothness and
variable integrability was firstly mixed up by Diening, H\"ast\"o and Roudenko in
\cite{dhr09}, they introduced Triebel-Lizorkin spaces with variable exponents
$F_{p(\cdot),q(\cdot)}^{s(\cdot)}(\rn)$ and proved a
trace theorem as follows:
$$\mathop\mathrm{Tr} F_{p(\cdot),q(\cdot)}^{s(\cdot)}(\rn)
=F_{p(\cdot,0), p(\cdot,0)}^{s(\cdot,0)-1/{p(\cdot,0)}}(\rr^{n-1}),$$
(see \cite[Theorem 3.13]{dhr09}), which shows that these spaces behaved nicely with
respect to the trace operator.
Subsequently, Vyb\'iral \cite{vj09} established Sobolev-Jawerth embeddings
of these spaces. On the other hand, Almeida and H\"ast\"o \cite{ah10} introduced
the Besov space with variable smoothness and integrability
$B_{p(\cdot),q(\cdot)}^{s(\cdot)}(\rn)$, which makes a further step
in completing the unification process of function
spaces with variable smoothness and integrability.
Later, Drihem \cite{dd12} established the atomic characterization of
$B_{p(\cdot),q(\cdot)}^{s(\cdot)}(\rn)$ and Noi et al. \cite{noi14, noi12,nois12}
also studied the space $B_{p(\cdot),q(\cdot)}^{s(\cdot)}(\rn)$ and
$F_{p(\cdot),q(\cdot)}^{s(\cdot)}(\rn)$ including the boundedness of trace and
extension operators, duality and complex interpolation. Here we point out that vector-valued
convolution inequalities developed in \cite[Lemma 4.7]{ah10}
and \cite[Theorem 3.2]{dhr09} supply well remedy for the absence of the
Fefferman-Stein vector-valued inequality for the mixed Lebesgue sequence spaces
$\ell^{q(\cdot)}(L^{p(\cdot)}(\rn))$ and $L^{p(\cdot)}(\ell^{q(\cdot)}(\rn))$,
respectively, in studying Besov spaces and Triebel-Lizorkin spaces
with variable smoothness and integrability.

More generally, 2-microlocal Besov and Triebel-Lizorkin spaces with variable,
$B_{p(\cdot),q(\cdot)}^{\textbf{\emph{w}}(\cdot)}(\rn)$ and
$F_{p(\cdot),q(\cdot)}^{\textbf{\emph{w}}(\cdot)}(\rn)$, were
introduced by Kempka \cite{kempka09,kempka10} and provided a unified approach that
cover the classical Besov and Triebel-Lizorkin spaces as well as versions of variable
smoothness and integrability. Afterwards, Kempka and Vyb\'iral \cite{kv12}
characterized these spaces by local means and ball means of differences.
The trace spaces of 2-microlocal type spaces were studied very recently
by Moura et al. \cite{mns13} and Gon\c{c}alves et al. \cite{gmn14}.

On the other hand, Besov-type spaces $B_{p,q}^{s,\tau}(\rn)$
and Triebel-Lizorkin spaces $F_{p,q}^{s,\tau}(\rn)$ and their homogeneous counterparts
for all admissible parameters were introduced in \cite{yyjfa,yymz10,ysiy}
in order to clarify the relations among Besov spaces,
Triebel-Lizorkin spaces and
$Q$ space (see \cite{dx04,ejpx}).
Various properties and equivalent characterizations of
Besov-type and Triebel-Lizorkin-type spaces,
including smoothness atomic, molecular or wavelet decompositions, characterizations,
respectively,
via differences, oscillations, Peetre maximal functions, Lusin area functions or
$g_\lz^\ast$ functions, have already been established in
\cite{d13,lsuyy,yyna,yyzaa,yy4,yyzjfsa,yyz12,yhsy}.
Moreover, these function spaces, including some of their special cases
 related to $Q$ spaces,
have been used to study the existence and the regularity of
solutions of some partial differential equations such as (fractional)
Navier-Stokes equations; see, for example, \cite{lxy,ly13,lzjmaa,lzjfa,t13,ysy13,t15}.
Based on $F_{p,q}^{s,\tau}(\rn)$, we introduced the Triebel-Lizorkin-type
space with variable exponent $\btlve$ in \cite{yyz14} with a measurable function
$\phi$ on $\urn$ and obtained a related trace theorem (\cite[Theorem 4.1]{yyz14}).

In this article, based on Besov-type spaces $B_{p,q}^{s,\tau}(\rn)$ and
variable Besov spaces $B_{p(\cdot),q(\cdot)}^{s(\cdot)}(\rn)$,
we are aimed to introduce another more generalized
scale of function spaces with variable smoothness $s(\cdot)$,
variable integrability $p(\cdot)$ and $q(\cdot)$, and a measurable function $\phi$
on $\urn$, denoted
by $\bbeve$, which covers both Besov spaces with variable
smoothness and integrability and Besov-type spaces.
We then establish their $\vz$-transform characterization in the sense
of Frazier and Jawerth. We also characterize these spaces by
smooth atoms or Peetre maximal functions in this article and
give some basic properties and Sobolev-type embeddings. As applications,
we prove a trace theorem of $\bbeve$ and obtain several equivalent norms of these spaces.

This article is organized as follows.

In Section \ref{s2}, we first give some
conventions and notation such as semimodular spaces, variable and
mixed Lebesgue-sequence spaces, and also introduce variable Besov-type spaces $\bbeve$.
We point out that the function spaces studied in this article fit into the
framework of so-called semimodular spaces. At the end of this section,
we point out that, in general, the scale of Besov-type spaces with variable
smoothness and integrability and the scale of Musielak-Orlicz Besov-type spaces
in \cite{yyzrmc} do not cover each other (see Remark \ref{r-com} below).

Section \ref{s-trans} is devoted to the $\vz$-transform characterization of
$\bbeve$ in the sense of Frazier and Jawerth \cite{fj90},
which is then applied to show that $\bbeve$ is well defined.
This is different from \cite[Theorem 5.5]{ah10}, in which the space
$B_{p(\cdot),q(\cdot)}^{s(\cdot)}(\rn)$ was proved to be well defined via
the Calder\'on reproducing formula. We point out that the method used in this article
is originally from Frazier and Jawerth \cite{fj90},
which is smartly modified in this article, via a subtle decomposition
of dyadic cubes, so that it is suitable to the present setting
(see Theorem \ref{t-transform} and Corollary \ref{c-indepen} below).
Observe that the \emph{r-trick lemma} from \cite[Lemma A.6]{dhr09}
(see also Lemma \ref{l-r-trick} below) plays a key role in establishing
a convolutional estimate so that we can use the convolutional inequality
from \cite[Lemma 4.7]{ah10} (see also Lemma \ref{l-conv-ineq} below)
to obtain the desired conclusion.

In Section \ref{s4}, by making full use of
the \emph{r-trick lemma} from \cite[Lemma A.6]{dhr09} again, we mainly give out the
Sobolev-type embedding property of $\bbeve$ (see Proposition \ref{p-se-embed}
and Theorem \ref{t-sobolev} below).
Some other basic embeddings and
properties of the spaces $\bbeve$ are also presented.

In Section \ref{s-equi}, we first characterize the space $\bbeve$ via Peetre maximal
functions (see Theorem \ref{t-equivalent} below). A key step to obtain this is to establish
a technical lemma (see Lemma \ref{l-equi} below),
which indicates that the Peetre maximal function can be controlled,
via semimodulars,  by the approximation to the identity in a suitable way.
Applying Theorem \ref{t-equivalent}, we further obtain two equivalent
characterizations of  $\bbeve$ (see Theorem \ref{t-equivalent-x} below).
Finally, in this section, by applying a Hardy-type inequality from \cite[Lemma 3.11]{d13} (see also
Lemma \ref{l-hardy} below) and the Sobolev-type embedding theorem obtained
in Section \ref{s4}, together with some ideas from the proof of Lemma \ref{l-equi},
we establish the smooth atomic characterization of $\bbeve$ (see Theorem \ref{t-atomd} below).

In the last section, Section \ref{s6}, as an application of the smoothness atomic
characterization obtained in Theorem \ref{t-atomd},
we prove a trace theorem for $\bbeve$ (see Theorem \ref{t-trace} below), which partly extends
the corresponding one obtained in \cite[Theorem 3.4]{mns13} and also
\cite[Theorem 5.1(1)]{noi14}.
The key point for this is to prove that
the trace space of $\bbeve$ is independent of the $n$-th coordinate of
variable exponents $p(\cdot)$, $q(\cdot)$ and $s(\cdot)$ (see Corollary \ref{c-trace}
and Lemma \ref{l-trace4} below).

\section{Preliminary\label{s2}}

\hskip\parindent
Throughout the article, we denote by
$C$ a \emph{positive constant} which is independent of the main
parameters, but may vary from line to line. The \emph{symbols}
$A\ls B$ means $A\le CB$. If $A\ls B$ and $B\ls A$, then we write $A\sim B$.
For all $a,\,b\in\rr$, let $a\vee b:=\max\{a,\,b\}$.
For all $k:=(k_1,\dots,k_n)\in\zz^n$, let $|k|:=|k_1|+\cdots+|k_n|$.
Let $\zz_+:=\{0,1,\dots\}$, $\nn:=\{1,2,\dots\}$ and $\mathbb{K}:=\rr$ or $\cc$.
Let $\urn:=\rn\times[0,\fz)$.
If $E$ is a subset of $\rn$, we denote by $\chi_E$ its
\emph{characteristic function} and $\wz \chi_E:=|E|^{-1/2}\chi_E$.
For all $x\in\rn$ and $r\in(0,\fz)$,
denote by $Q(x,r)$ the cube centered at $x$ with side length $r$,
whose sides parallel axes of coordinate.
For all cube $Q\subset \rn$, we denote its \emph{center} by $c_Q$ and
its \emph{side length} by $\ell(Q)$ and, for $a\in(0,\fz)$, we denote by
$aQ$ the \emph{cube} concentric with $Q$ having the side length with $a\ell(Q)$.

\subsection{Modular spaces}
\hskip\parindent
In this subsection, we recall some conventions and notions about
(semi)modular spaces, and state some basic results.
For an exposition of these concepts, we refer to the monograph
\cite[Chapters 1-3]{dhr11}.
 The function spaces studied in this article
fit into the framework of so-called semimodular spaces.
In what follows, let $X$ be a vector space over $\mathbb{K}$.

\begin{definition}
A function
$\varrho:\ X\to[0,\fz]$ is called a \emph{semimodular} on $X$ if
it satisfies:

(i) $\varrho(0)=0$ and, for all $f\in X$ and $\lz\in \mathbb K$ with $|\lz|=1$,
$\varrho(\lz f)=\rho(f)$;

(ii) If $\varrho(\lz f)=0$ for all $\lz\in(0,\fz)$, then $f=0$;

(iii) $\rho$ is \emph{quasiconvex}, namely, there exists $A\in[1,\fz)$ such that,
for all $f,\ g\in X$,
$$\varrho(\theta f+(1-\theta)g)\le A\lf[\theta\varrho(f)+(1-\theta)\varrho(g)\r];$$

(iv) $\lz\mapsto\varrho(\lz f)$ is left continuous on $[0,\fz)$ for every $f\in X$,
namely, $\lim_{\lz<1,\lz\to1}\varrho(\lz f)=\varrho(f)$.

A semimodular $\varrho$ is called a \emph{modular} if it satisfies that
$\varrho(f)=0$ implies $f=0$, and is called \emph{continuous} if, for every
$f\in X$, the mapping $\lz\mapsto\varrho(\lz f)$ is continuous on $[0,\fz)$, namely,
$\lim_{\lz\to1}\varrho(\lz f)=\varrho(f)$.
\end{definition}

\begin{definition}
Let $\varrho$ be a (semi)modular on $X$. Then
$$X_{\varrho}:=\{f\in X:\ \exists\ \lz\in(0,\fz)\
{\rm such\ that}\ \varrho(\lz f)<\fz\}$$
is called a \emph{(semi)modular space} with the norm
$$\|f\|_\varrho:=\inf\lf\{\lz\in(0,\fz):\ \varrho(f/\lz)\le1\r\}.$$
\end{definition}

The following Lemma \ref{l-modular} is just \cite[Lemma 2.1.14]{dhr11}.

\begin{lemma}\label{l-modular}
Let $\varrho$ be a semimodular on $X$.
Then $\|f\|_{\varrho}\le1$ if and only if $\varrho(f)\le1$;
moreover, if $\varrho$ is continuous, then $\|f\|_{\varrho}<1$ if and only if
$\varrho(f)<1$, as well as $\|f\|_{\varrho}=1$ if and only if $\varrho(f)=1$.
\end{lemma}
\begin{remark}
When dealing with some complicated quasi-norms defined via variable exponents,
we are often converted to dealing with the corresponding semimodulars
by Lemma \ref{l-modular}; see Remarks \ref{r-modular} and \ref{re-mixed}(i) below.
\end{remark}

\subsection{Spaces of variable integrability}

\hskip\parindent
Here, we recall some definitions and notation for the space with
variable integrability.
For a measurable function
$p(\cdot):\ \rn\to(0,\fz]$, let
\begin{equation*}
p_-:=\mathop{\rm ess\,inf}\limits_{x\in \rn}p(x)\quad {\rm and}\quad
p_+:=\mathop{\rm ess\,sup}\limits_{x\in \rn}p(x).
\end{equation*}
The set of \emph{variable exponents} in this article, denoted by $\cp(\rn)$, is
the \emph{set of all measurable functions
$p(\cdot):\ \rn\to(0,\fz]$ satisfying $p_-\in(0,\fz]$}.
For $p(\cdot)\in\cp(\rn)$ and $x\in\rn$, define the function
$\rho_{p(x)}$ by setting,
for all $t\in[0,\fz)$,
\begin{eqnarray*}
\rho_{p(x)}(t):=\lf\{
\begin{array}{cl}
t^{p(x)},\quad&{\rm if}\ p(x)\in(0,\fz),\\
0,\quad&{\rm if}\ p(x)=\fz\ {\rm and}\ t\in[0,1],\\
\fz,\quad&{\rm if}\ p(x)=\fz\ {\rm and}\ t\in(1,\fz).
\end{array}\r.
\end{eqnarray*}
The \emph{variable exponent modular} of a measurable function $f$
on $\rn$ is defined by
\begin{equation*}
\varrho_{p(\cdot)}(f):=\int_\rn\rho_{p(x)}(|f(x)|)\,dx.
\end{equation*}

\begin{remark}\label{r-modular}
Let $p\in\cp(\rn)$ satisfy $p_-\in[1,\fz]$.
Then $\varrho_{p(\cdot)}$ is a semimodular (see \cite[Definition 3.2.1]{dhr11}),
which, together with Lemma \ref{l-modular}, implies that
$\|f\|_{\vlp}\le1$ if and only if $\varrho_{p(\cdot)}(f)\le1$.
Moreover, for all $p\in\cp(\rn)$, $\|f\|_{\vlp}\le1$ if and only if $\varrho_{p(\cdot)}(f)\le1$.
\end{remark}

\begin{definition}\label{def1}
Let $p(\cdot)\in\cp(\rn)$ and $E$ be a measurable subset of $\rn$. Then the
\emph{variable exponent Lebesgue space} $L^{p(\cdot)}(E)$ is defined to be the
set of all measurable functions $f$ such that
\begin{equation*}
\|f\|_{L^{p(\cdot)}(E)}
:=\inf\lf\{\lz\in(0,\fz):\ \varrho_{p(\cdot)}\lf(f\chi_E/\lz\r)\le1\r\}<\fz.
\end{equation*}
\end{definition}

\begin{remark}\label{re-vlp}
Let $p\in\cp(\rn)$.

(i) If $p_-\in[1,\fz]$, then $\vlp$ is a Banach space
(see \cite[Theorem 3.2.7]{dhr11}). In particular,
for all $\lz\in\cc$,
$\|\lz f\|_{\vlp}=|\lz|\|f\|_{\vlp}$
 and, for all $f,\ g\in\vlp$,
$$\|f+g\|_{\vlp}\le \|f\|_{\vlp}+\|g\|_{\vlp}.$$

(ii) (\textbf{The H\"older inequality}) Assume that
$1<p_-\le p_+<\fz$.
 It was proved in \cite[Lemma 3.2.20]{dhr11} that,
if $f\in \vlp$ and $g\in L^{p^\ast(\cdot)}(\rn)$, then $fg\in L^1(\rn)$
and
$$\int_\rn|f(x)g(x)|\,dx\le C\|f\|_{\vlp}\|g\|_{L^{p^\ast(\cdot)}(\rn)},$$
where $p^\ast(x):=\frac{p(x)}{p(x)-1}$ for all $x\in\rn$, $C$ is a positive constant
depending on $p_-$ or $p_+$, but independent of $f$ and $g$.

(iii) If $p_+\in(0,1]$, then it is easy to see that,
for all nonnegative functions $f,\ g\in\vlp$,
the following reverse Minkowsiki inequality holds true:
$$\|f\|_{\vlp}+\|g\|_{\vlp}\le\|f+g\|_{\vlp}.$$
\end{remark}

\begin{definition}\label{def2}
Let $p,\ q\in\cp(\rn)$ and $E$ be a measurable subset of $\rn$. Then the
\emph{mixed Lebesgue-sequence space} $\ell^{q(\cdot)}(L^{p(\cdot)}(E))$
is defined to be the set of all sequences $\{f_v\}_{v\in\nn}$ of functions
in $L^{p(\cdot)}(E)$ such that
\begin{eqnarray*}
\lf\|\{f_v\}_{v\in\nn}\r\|_{\ell^{q(\cdot)}(L^{p(\cdot)}(E))}
:=\inf\lf\{\lz\in(0,\fz):\ \varrho_{\ell^{q(\cdot)}(L^{p(\cdot)})}
\lf(\{f_v\chi_E/\lz\}_{v\in\nn}\r)\le1\r\}<\fz,
\end{eqnarray*}
where, for all sequences $\{g_v\}_{v\in\nn}$ of measurable functions,
\begin{equation}\label{mixed-x}
\varrho_{\ell^{q(\cdot)}(L^{p(\cdot)})}(\{g_v\}_{v\in\nn})
:=\sum_{v\in\nn}\inf\lf\{\mu_v\in(0,\fz):\ \varrho_{p(\cdot)}
\lf(g_v/\mu_v^{1/q(\cdot)}\r)\le1\r\}
\end{equation}
with the convention $\lz^{1/\fz}=1$ for all $\lz\in(0,\fz)$.
\end{definition}

\begin{remark}\label{re-mixed}
Let $p,\ q\in\cp(\rn)$.

(i) The mixed Lebesgue-sequence space $\ell^{q(\cdot)}(L^{p(\cdot)}(\rn))$
was introduced by Almeida and H\"ast\"o \cite{ah10}. Moreover,
$\varrho_{\ell^{q(\cdot)}(L^{p(\cdot)})}$ is a semimodular (see
\cite[Proposition 3.5]{ah10}), which, together with Lemma \ref{l-modular}, implies that
$\|f\|_{\ell^{q(\cdot)}(L^{p(\cdot)}(\rn))}\le1$ if and only if
$\varrho_{\ell^{q(\cdot)}(L^{p(\cdot)})}(f)\le1$.

(ii) If $q_+\in(0,\fz)$, then, for all measurable functions $g$ on $\rn$,
it holds true that
$$\inf\lf\{\lz\in(0,\fz):\ \varrho_{p(\cdot)}
\lf(\frac g{\lz^{1/q(\cdot)}}\r)\le1\r\}
=\lf\||g|^{q(\cdot)}\r\|_{L^{\frac{p(\cdot)}{q(\cdot)}}(\rn)}.$$

(iii) Let $\{g_v\}_{v\in\nn}$ be a sequence of functions in $\vlp$.
If, for all $v\in\{2,3,\dots\}$, $g_v\equiv0$, then
$$\lf\|\{g_v\}_{v\in\nn}\r\|_{\ell^{q(\cdot)}(L^{p(\cdot)}(\rn))}=\|g_1\|_{\vlp}$$
(see \cite[Example 3.4]{ah10}).

(iv) If $p,\ q\in\cp(\rn)$, then $\|\cdot\|_{\ell^{q(\cdot)}(L^{p(\cdot)}(\rn))}$
is a quasi-norm on $\ell^{q(\cdot)}(L^{p(\cdot)}(\rn))$
(see \cite[Theorem 3.8]{ah10}); if either $\frac 1{p(x)}+\frac1{q(x)}\le1$ or
$q$ is a constant, then $\|\cdot\|_{\ell^{q(\cdot)}(L^{p(\cdot)}(\rn))}$ is a norm
(see \cite[Theorem 3.6]{ah10}); if either $p(x)\ge1$ and $q\in[1,\fz)$ is a constant
almost everywhere or $1\le q(x)\le p(x)\le\fz$ for almost every $x\in\rn$, then
$\|\cdot\|_{\ell^{q(\cdot)}(L^{p(\cdot)}(\rn))}$ is also a norm
(see \cite[Theorem 1]{kv13}).

(v) By \cite[Proposition 3.3]{ah10}, we know that, if $q\in(0,\fz]$ is constant,
 then
$$\lf\|\{g_v\}_{v\in\nn}\r\|_{\ell^q(L^{p(\cdot)}(\rn))}
=\lf\|\lf\{\|g_v\|_{L^{p(\cdot)}(\rn)}\r\}_{v\in\nn}\r\|_{\ell^q}.$$
\end{remark}

A measurable function $g\in\cp(\rn)$ is said to satisfy the
\emph{locally {\rm log}-H\"older continuous condition},
denoted by $g\in  C_{\rm loc}^{\log}(\rn)$,
if there exists a positive constant $C_{\log}(g)$ such that, for all $x,\ y\in\rn$,
\begin{equation}\label{ve1}
|g(x)-g(y)|\le \frac{C_{\log}(g)}{\log(e+1/|x-y|)},
\end{equation}
and $g$ is said to satisfy the
\emph{globally {\rm log}-H\"older continuous condition},
denoted by $g\in  C^{\log}(\rn)$,
if $g\in  C_{\rm loc}^{\log}(\rn)$ and there exist positive constants
$C_\fz$ and $g_\fz$
such that, for all $x\in\rn$,
\begin{equation}\label{ve2}
|g(x)-g_\fz|\le \frac{C_\fz}{\log(e+|x|)}.
\end{equation}
\begin{remark}\label{re-conv}
(i) Let $p\in  C^{\log}(\rn)$. Then, it was proved in
\cite[Lemma 4.6.3]{dhr11} that, for every $f\in\vlp$ and every nonnegative,
radially decreasing function $g\in L^1(\rn)$,
$$\|f\ast g\|_{\vlp}\le C\|f\|_{\vlp}\|g\|_{L^1(\rn)},$$
where $C$ is a positive constant independent of $f$ and $g$.

(ii) Let $p\in  \cp(\rn)$. If $p_+\in(0,\fz)$, then $p\in C^{\log}(\rn)$
if and only if $1/p\in C^{\log}(\rn)$. If $p$ satisfies \eqref{ve2}, then
$p_\fz=\lim_{|x|\to\fz}p(x)$.

(iii) If $q\in C_{\rm loc}^{\log}(\rn)$ and
$q_+=\fz$, then, by \eqref{ve1}, it is easy to see that
$q(x)=\fz$ for all $x\in\rn$. From this and Remark \ref{re-mixed}(v),
we deduce that, in the case that $q_+=\fz$, the mixed norm
$\|\cdot\|_{\ell^{q(\cdot)}(\vlp)}$ becomes the norm
$\|\cdot\|_{\ell^{\fz}(\vlp)}$.
\end{remark}

\subsection{The Besov-type space $\bbeve$}
\hskip\parindent
Let $\mathcal{G}(\urn)$ be the \emph{set} of all measurable functions
 $\phi:\ \urn\to(0,\fz)$ having the following properties: there exist
 positive constants $c_1$ and $c_2$ such that,
for all $x\in\rn$ and $r\in(0,\fz)$,
\begin{equation}\label{phi-1}
c_1^{-1}\phi(x,2r)\le\phi(x,r)\le c_1\phi(x,2r)
\end{equation}
and, for all $x,\,y\in\rn$ and $r\in(0,\fz)$ with $|x-y|\le r$,
\begin{equation}\label{phi-2}
c_2^{-1}\phi(y,r)\le\phi(x,r)\le c_2\phi(y,r).
\end{equation}
\begin{remark}
(i) We point out that \eqref{phi-1} and \eqref{phi-2} are
 called the \emph{doubling condition} and the
\emph{compatibility condition}, respectively, which have been used by Nakai
\cite{nakai93,nakai06} and Nakai and Sawano \cite{ns12} when they studied generalized
Campanato spaces.

(ii) There are several examples of $\phi$ that satisfy \eqref{phi-1} and
\eqref{phi-2}; see \cite[Remark 1.3]{yyz14}.
\end{remark}

In what follows, for $\phi\in\cg(\urn)$ and all cubes $Q:=Q(x,r)\subset\rn$ with
center $x\in\rn$ and radius $r\in(0,\fz)$, define $\phi(Q):=\phi(Q(x,r)):=\phi(x,r)$.
Let $\cs(\rn)$ be the \emph{space of all Schwartz functions} on
 $\rn$ and $\cs'(\rn)$
its \emph{topological dual space}.
A pair of functions, $(\vz, \Phi)$, is said to be \emph{admissible} if
$\vz,\ \Phi\in\cs(\rn)$ satisfy
\begin{equation}\label{x.1}
 {\rm supp}\, \wh \vz\subset\{\xi\in\rn:\ 1/2\le|\xi|\le2\}\
 {\rm and}\ |\wh \vz(\xi)|\ge
c>0\ {\rm when}\  3/5\le|\xi|\le5/3
\end{equation}
and
\begin{equation}\label{x.2}
{\rm supp}\, \wh \Phi\subset\{\xi\in\rn:\ |\xi|\le2\}\ {\rm and}\
|\wh\Phi(\xi)|\ge c>0\ {\rm when}\
|\xi|\le5/3,
\end{equation}
where $\wh f(\xi):=\int_{\rn}f(x)e^{-ix\cdot\xi}\,dx$ for all $\xi\in\rn$
and $c$ is a positive constant independent of $\xi\in\rn$.
For all $j\in\nn$,
$\vz\in\cs(\rn)$ and $x\in\rn$, we put
$\vz_j(x):=2^{jn}\vz(2^jx)$ and $\wz \vz(x):=\overline{\vz(-x)}$.
For $j\in\zz$ and $k\in\zz^n$,
denote by $Q_{jk}$ the \emph{dyadic cube} $2^{-j}([0,1)^n+k)$,
$x_{Q_{jk}}:=2^{-j}k$
its \emph{lower left corner} and $\ell(Q_{jk})$ its \emph{side length}.
Let $\cq:=\{Q_{jk}:\ j\in\zz,\ k\in\zz^n\}$,
$\cq^\ast:=\{Q\in\cq:\ \ell(Q)\le1\}$
 and $j_Q:=-\log_2\ell(Q)$ for all $Q\in\cq$.

Now we introduce the Besov-type space with variable smoothness
and integrability.

\begin{definition}\label{def-b}
Let $(\vz,\Phi)$ be a pair of admissible functions on $\rn$.
Let $p$, $q\in C^{\log}(\rn)$, $s\in  C_{\loc}^{\log}(\rn)\cap L^\fz(\rn)$
 and $\phi\in\mathcal G(\urn)$. Then the \emph{Besov-type space with variable
smoothness and integrability}, $\bbeve$, is defined to be the
set of all $f\in\cs'(\rn)$ such that
\begin{equation*}
\|f\|_{\bbeve}:=\sup_{P\in\cq}\frac1{\phi(P)}
\lf\|\lf\{2^{js(\cdot)}|\vz_j\ast f|\r\}_{j\ge(j_P\vee 0)}\r\|
_{\ell^{q(\cdot)}(L^{p(\cdot)}(P))}<\fz,
\end{equation*}
where the supremum is taken over all dyadic cubes $P$ in $\rn$.
\end{definition}
\begin{remark}\label{r-defi}
Let $p,\ q,\ s$ and $\phi$ be as in Definition \ref{def-b}.

(i) If $\phi(Q)=1$ for all cubes $Q$ of $\rn$, then
$\bbeve=B_{p(\cdot),q(\cdot)}^{s(\cdot)}(\rn),$
where $B_{p(\cdot),q(\cdot)}^{s(\cdot)}(\rn)$ denotes the
\emph{Besov space with variable smoothness and integrability}
introduced in \cite{ah10}.

(ii) If $p,\ q,\ s$ are constant exponents and $\phi(Q):=|Q|^{\tau}$ with
$\tau\in[0,\fz)$ for all cubes $Q$ of $\rn$,
then $\bbeve=B_{p,q}^{s,\tau}(\rn)$, where
$B_{p,q}^{s,\tau}(\rn)$ denotes the \emph{Besov-type space}
introduced in \cite{ysiy}.

(iii) By Remark \ref{re-conv}(iii), we see that, when $q_+=\fz$,
$$\|f\|_{\bbeve}=\|f\|_{B_{p(\cdot),\fz}^{s(\cdot),\phi}(\rn)}:=
\sup_{P\in\cq}\frac1{\phi(P)}\sup_{j\ge(j_P\vee 0)}
\lf\|2^{js(\cdot)}|\vz_j\ast f|\r\|_{L^{p(\cdot)}(P)}.$$

(iv) If $q,\ s$ are constants and $\phi$ is as in (ii), then
$\bbeve=B_{p(\cdot),q}^{s,\tau}(\rn)$, which was investigated in \cite{lsuyy1}.
\end{remark}

We end this section by comparing Besov-type spaces with variable smoothness and
integrability in this article with Musielak-Orlicz Besov-type spaces in \cite{yyzrmc}
and show that, in general, these two scales of Besov-type spaces do not cover each
other.

To recall the definition of Musielak-Orlicz Besov-type spaces, we need some notions
on Musielak-Orlicz functions. A function $\vz:\ \rn\times[0,\fz)\to[0,\fz)$
is called a \emph{Musielak-Orlicz function} if the function $\vz(x,\cdot):\
[0,\fz)\to[0,\fz)$ is an Orlicz function for all $x\in\rn$, namely, for any given
$x\in\rn$, $\vz(x,\cdot)$ is nondecreasing, $\vz(x,0)=0$, $\vz(x,t)\in(0,\fz)$ for
all $t\in(0,\fz)$ and $\lim_{t\to\fz}\vz(x,t)=\fz$, and $\vz(\cdot,t)$ is a
Lebesgue measurable function for all $t\in[0,\fz)$. A Musielak-Orlicz function
$\vz$ is said to be of \emph{uniformly upper} (resp. \emph{lower})
\emph{type $p$} for some $p\in[0,\fz)$ if there exists a positive constant $C$ such that,
for all $x\in\rn$, $t\in[0,\fz)$ and $s\in[1,\fz)$ (resp. $s\in[0,1]$),
$\vz(x,st)\le Cs^p\vz(x,t)$ (see \cite{ky14}).
Let
$$i(\vz):=\sup\{p\in(0,\fz):\ \vz\ {\rm is\ of\ uniformly\ lower\ type}\ p\}$$
and
$$I(\vz):=\inf\{p\in(0,\fz):\ \vz\ {\rm is\ of\ uniformly\ upper\ type}\ p\}.$$
The function $\vz(\cdot,t)$ is said to satisfy the
\emph{uniformly Muckenhoupt condition for some $r\in[1,\fz)$}, denoted by
$\vz\in \aa_r(\rn)$, if, when $r\in(1,\fz)$,
$$\sup_{t\in(0,\fz)}\sup_{{\rm balls}\ B\subset\rn}
\frac1{|B|^r}\int_B\vz(x,t)\,dx\lf\{\int_B[\vz(y,t)]^{-r'/r}\,dy\r\}^{r/r'}<\fz,$$
where $1/r+1/r'=1$, or, when $r=1$,
$$\sup_{t\in(0,\fz)}\sup_{{\rm balls}\ B\subset\rn}
\frac1{|B|}\int_B\vz(x,t)\,dx\lf\{\mathop{\rm ess\,sup}\limits_{y\in B}
[\vz(y,t)]^{-1}\r\}<\fz.$$
Let $\aa_\fz(\rn):=\cup_{r\in[1,\fz)}\aa_r(\rn)$.

The \emph{Musielak-Orlicz space} $L^{\vz}(\rn)$ is defined as the set of all
measurable functions $f$ on $\rn$ such that
$$\|f\|_{L^\vz(\rn)}:=\inf\lf\{\lz\in(0,\fz):\
\int_\rn\vz(x,|f(x)|/\lz)\,dx\le1\r\}<\fz.$$

Let $\cs_\fz(\rn)$ be the \emph{space} of all Schwartz functions $h$ satisfying
that, for all multi-indices $\gamma:=(\gamma_1,\dots,\gamma_n)\in\zz_+^n$,
$\int_\rn h(x)x^\gamma\,dx=0$ and let $\cs'_\fz(\rn)$ be its \emph{topological dual space}.
Now we recall the definition of Musielak-Orlicz Besov-type spaces from \cite{yyzrmc} as follows.

\begin{definition}\label{d-mob}
Let $s\in\rr$, $\tau\in[0,\fz)$, $q\in(0,\fz]$ and $\psi$ be a Schwartz function
satisfying supp\,$\wh\psi\subset\{\xi\in\rn:\ 1/2\le|\xi|\le2\}$ and
$|\wh\psi(\xi)|\ge C>0$ if $3/5\le|\xi|\le5/3$ for some positive constant $C$
independent of $\xi\in\rn$. For all $j\in\zz$ and $x\in\rn$, let
$\psi_j(x):=2^{jn}\psi(2^jx)$. Assume that, for $j\in\{1,2\}$, $\vz_j$ is a
Musielak-Orlicz function with $0<i(\vz_j)\le I(\vz_j)<\fz$ and $\vz_j\in\aa_\fz(\rn)$.
Then the \emph{Musielak-Orlicz Besov-type space}
$\dot B_{\vz_1,\vz_2,q}^{s,\tau}(\rn)$ is defined to be the space of all
$f\in\cs_\fz'(\rn)$ such that
$$\|f\|_{\dot B_{\vz_1,\vz_2,q}^{s,\tau}(\rn)}
:=\sup_{P\in\cq}\frac1{\|\chi_P\|_{L^{\vz_1}(\rn)}}
\lf\|\lf\{\sum_{j=j_P}^\fz(2^{js}|\psi_j\ast f|)^q\r\}^{1/q}
\r\|_{L^{\vz_2}(\rn)}<\fz$$
with suitable modification made when $q=\fz$, where the supremum is taken over
all dyadic cubes $P$ of $\rn$.

\begin{remark}\label{r-com}
(i) Observe that, if $\vz(x,t):=t^{p(x)}$ for all $x\in\rn$ and $t\in[0,\fz)$,
then $L^\vz(\rn)=L^{p(\cdot)}(\rn)$.

(ii) Let $\vz_1$ be as in Definition \ref{d-mob}.
If $\phi(P):=\|\chi_P\|_{L^{\vz_1}(\rn)}$ for all cubes $P\subset\rn$, then, by \cite[Lemma 2.6]{zyl14}
and \cite[Remark 1.3(iv)]{yyz14}, we see that $\phi$ satisfies
\eqref{phi-1} and \eqref{phi-2}.

(iii) The scale of Besov-type spaces with variable smoothness and integrability
can not be covered by the scale of
Musielak-Orlicz Besov-type spaces. Indeed, by \cite[Remark 2.23(iii)]{yyzrmc},
we find that there exists some function $p(\cdot)$ satisfying conditions in Definition
\ref{def-b}, but $t^{p(\cdot)}$ is not a Musielak-Orlicz function as in Definition
\ref{d-mob}.

(iv) Also, the scale of Besov-type spaces  with variable smoothness and integrability
can not
cover the scale of Musielak-Orlicz Besov-type spaces, since a Musielak-Orlicz function
$\vz(x,t)$ may not be written as $\vz(x,t):=t^{p(x)}$ for all
$x\in\rn$ and $t\in[0,\fz)$ with some variable exponent $p(\cdot)$ as in Definition
\ref{def-b} (see, for example, the Musielak-Orlicz function $\vz$ as in
\cite[(1.5)]{yyzrmc}).
\end{remark}

\end{definition}

\section{The $\vz$-transform characterization\label{s-trans}}

\hskip\parindent
The purpose of this section is to show that
$\bbeve$ is independent of the choice of admissible function pairs
$(\vz,\Phi)$. To this end, we first introduce
the sequence space $\beve$ with respect to
 $\bbeve$ and then establish its $\vz$-transform characterization
in the sense of Frazier and Jawerth \cite{fj90}.

\begin{definition}\label{d-be}
Let $p$, $q$, $s$ and $\phi$ be as in Definition
\ref{def-b}. Then the \emph{sequence space} $\beve$ is defined to be the set of all
sequences $t:=\{t_Q\}_{Q\in\cq^\ast}\subset \cc$ such that
\begin{eqnarray*}
\|t\|_{\beve}:=\sup_{P\in\cq}\frac1{\phi(P)}
\lf\|\lf\{\sum_{\gfz{Q\in\cq^\ast,\,Q\subset P}{\ell(Q)=2^{-j}}}
|Q|^{-\frac{s(\cdot)}n}|t_Q|\wz\chi_Q\r\}_{j\ge(j_P\vee 0)}\r\|
_{\ell^{q(\cdot)}(L^{p(\cdot)}(P))}<\fz,
\end{eqnarray*}
where the supremum is taken over all dyadic cubes $P$ in $\rn$.
\end{definition}
\begin{remark}\label{r-lattice}
Let
$\mathcal D_0(\rn):=\{Q\subset\rn:\ Q\ {\rm is\ a\ cube\ and}\
\ell(Q)=2^{-j_0}\ {\rm for\ some}\
j_0\in\zz\}.$
Then the supremum in Definitions \ref{def-b} and \ref{d-be}
can be equivalently taken over all cubes in $\mathcal D_0(\rn)$,
the details being omitted.
\end{remark}

Let $(\vz,\Phi)$ be a pair of admissible functions. Then $(\wz \vz,\wz \Phi)$
is also a pair of admissible functions, where $\wz\vz(\cdot):=\ov{\vz(-\cdot)}$
and $\wz\Phi(\cdot):=\ov{\Phi(-\cdot)}$.
Moreover, by \cite[pp.\,130-131]{fj90} or \cite[Lemma (6.9)]{fjw91},
there exist
Schwartz functions $\psi$ and $\Psi$ satisfying \eqref{x.1} and \eqref{x.2},
respectively, such that, for all $\xi\in\rn$,
\begin{equation}\label{cz1}
\widehat{\Phi}(\xi)\widehat{\Psi}(\xi)+
\sum_{j=1}^\fz\widehat{\vz}(2^{-j}\xi)\widehat{\psi}(2^{-j}\xi)=1.
\end{equation}
Recall that the \emph{$\vz$-transform}
$S_\vz$ is defined to be the mapping taking each $f\in\cs'(\rn)$ to the
sequence $S_\vz(f):=\{(S_\vz f)_Q\}_{Q\in\cq^\ast}$,
where $(S_\vz f)_Q:=|Q|^{1/2}\vz_{j_Q}\ast f(x_Q)$ with $\vz_0$ replaced by $\Phi$;
the \emph{inverse $\vz$-transform} $T_\psi$ is defined to
be the mapping taking a sequence $t:=\{t_Q\}_{Q\in\cq^\ast}\subset\cc$ to
\begin{equation}\label{psi}
T_\psi t:=\sum_{Q\in\cq^\ast,\,\ell(Q)=1}t_Q\Psi_Q
+\sum_{Q\in\cq^\ast,\,\ell(Q)<1}t_Q\psi_Q;
\end{equation}
see, for example, \cite[p.\,31]{ysiy}.

Now we state the following $\vz$-transform characterization for $\bbeve$, which
is the main result of this section.
\begin{theorem}\label{t-transform}
Let $p,\ q$, $s$ and $\phi$ be as in Definition \ref{def-b}
and $\vz,\ \psi,\ \Phi$ and $\Psi$ as in \eqref{cz1}.
Then operators $S_\vz:\ B_{p(\cdot),q(\cdot)}^{s(\cdot),\phi}(\rn)
\to b_{p(\cdot),q(\cdot)}^{s(\cdot),\phi}(\rn)$ and
$T_\psi:\ b_{p(\cdot),q(\cdot)}^{s(\cdot),\phi}(\rn)
\to B_{p(\cdot),q(\cdot)}^{s(\cdot),\phi}(\rn)$ are bounded. Furthermore,
$T_\psi\circ S_\vz$ is the identity on
$B_{p(\cdot),q(\cdot)}^{s(\cdot),\phi}(\rn)$.
\end{theorem}

\begin{remark}
(i) The conclusion of Theorem \ref{t-transform} is new even when $\phi\equiv1$.

(ii) If $p,\ q$, $s$ and $\phi$ are as in Remark \ref{r-defi}(ii),
then Theorem \ref{t-transform} goes back to \cite[Theorem 2.1]{ysiy}.

(iii) $T_\psi$ is well defined for all
$t\in b_{p(\cdot),q(\cdot)}^{s(\cdot),\phi}(\rn)$;
see Lemma \ref{l-welld} below.
\end{remark}

From Theorem \ref{t-transform} and an argument similar to that used in
\cite[Remark 2.6]{fj90}, we immediately deduce the following conclusion,
the details being omitted.
 \begin{corollary}\label{c-indepen}
With all notation as in Definition \ref{def-b}, the space $\bbeve$ is independent
of the choice of the admissible function pairs $(\vz,\Phi)$.
\end{corollary}

The remainder of this section is to prove Theorem \ref{t-transform}.
We begin with the following Lemmas \ref{l-esti-cube} and \ref{l-cube-1}, which
 are just
\cite[Lemma 2.5]{yyz14} and \cite[Lemma 2.6]{yyz14}, respectively.

\begin{lemma}\label{l-esti-cube}
Let $\phi\in\cg(\urn)$. Then
there exist positive constants $C$ and $\wz C$ such that,
for all $j\in\zz_+$ and $k\in\zz^n$,
$\phi(Q_{jk})\le C 2^{j\log_2c_1}(|k|+1)^{2\log_2c_1}$
and, for all $Q\in\cq$ and $l\in\zz^n$,
$$\frac{\phi(Q+l\ell(Q))}{\phi(Q)}\le \wz C(1+|l|)^{2\log_2c_1},$$
where $c_1$ is as in \eqref{phi-1}.
\end{lemma}

\begin{lemma}\label{l-cube-1}
Let $p\in  C^{\log}(\rn)$. Then there exists a positive
constant $C$ such that, for all dyadic cubes $Q_{jk}$ with $j\in\zz_+$ and $k\in\zz^n$,
\begin{equation*}
C^{-1}2^{-\frac n{p_-}j}(1+|k|)^{n(\frac1{p_+}-\frac 1{p_-})}
\le \|\chi_{Q_{jk}}\|_{L^{p(\cdot)}(\rn)}
\le C2^{-\frac n{p_+}j}(1+|k|)^{n(\frac1{p_-}-\frac 1{p_+})}.
\end{equation*}
\end{lemma}

In what follows, for all $h\in\cs(\rn)$ and $M\in\zz_+$, let
$$\|h\|_{\cs_M(\rn)}
:=\sup_{|\gamma|\le M}\sup_{x\in\rn}|\partial^\gamma h(x)|(1+|x|)^{n+M+\gamma}.$$

\begin{lemma}\label{l-welld}
Let $p,\ q$, $s$ and $\phi$ be as in Definition \ref{def-b}.
Then, for all $t\in b_{p(\cdot),q(\cdot)}^{s(\cdot),\phi}(\rn)$,
$T_\psi t$ in \eqref{psi} converges in $\cs'(\rn)$; moreover,
$T_\psi:\ b_{p(\cdot),q(\cdot)}^{s(\cdot),\phi}(\rn)\to \cs'(\rn)$
is continuous.
\end{lemma}
\begin{proof}
Observe that, by Remark \ref{re-mixed}(iii), we find that,
for any $Q\in\cq^\ast$,
\begin{eqnarray*}
|t_Q|&&\le\lf\||Q|^{-\frac{s(\cdot)}n}|t_Q|\wz\chi_Q\r\|_{L^{p(\cdot)}(Q)}
\|\chi_Q\|_{L^{p(\cdot)}(Q)}^{-1}|Q|^{\frac{s_-}n+\frac12}\\
&&\le\lf\|\lf\{\sum_{\gfz{\wz Q\subset Q,\,\wz Q\in\cq^\ast}{\ell(\wz Q)=2^{-j}}}
|\wz Q|^{-\frac{s(\cdot)}n}|t_{\wz Q}|\wz\chi_{\wz Q}\r\}_{j\ge(j_Q\vee0)}
\r\|_{\ell^{q(\cdot)}(L^{p(\cdot)}(Q))}
\frac{|Q|^{\frac{s_-}n+\frac12}}{\|\chi_Q\|_{L^{p(\cdot)}(Q)}}\\
&&\le \|t\|_{\beve}\frac{\phi(Q)}{\|\chi_Q\|_{L^{p(\cdot)}(Q)}}
|Q|^{\frac{s_-}n+\frac12}.
\end{eqnarray*}
Then, by this and an argument similar to that used in the proof of \cite[Lemma 2.4]{yyz14},
we conclude that, for all $h\in\cs(\rn)$,
$|\la T_\psi t,h\ra|\ls\|t\|_{\beve}\|h\|_{\cs_M(\rn)}$
with some large $M\in(0,\fz)$, which completes the proof of
Lemma \ref{l-welld}.
\end{proof}

In what follows, for any $m\in(0,\fz)$ and $j\in\zz$, let, for all $x\in\rn$,
$\eta_{j,m}(x):=2^{jn}(1+2^j|x|)^{-m}$.
The following lemma is the so-called \emph{r-trick lemma},
which is \cite[Lemma A.6]{dhr09}.

\begin{lemma}\label{l-r-trick}
Let $r\in(0,\fz)$, $v\in\zz_+$ and $m\in(n,\fz)$. Then there exists a positive
constant $C$, only depending on $r,\ m$ and $n$, such that, for all $x\in\rn$ and
$g\in\cs'(\rn)$ with ${\rm supp}\,\wh g\subset\{\xi:\ |\xi|\le2^{v+1}\}$,
$\sup_{z\in Q}|g(z)|\le C\lf[\eta_{v,m}\ast(|g|^r)(x)\r]^\frac1r$,
where $Q\in\cq$ contains $x$ and $\ell(Q)=2^{-v}$.
\end{lemma}

The following Lemma \ref{l-eta} is just \cite[Lemma 19]{kv12}, which is a variant
of \cite[Lemma 6.1]{dhr09}.
\begin{lemma}\label{l-eta}
Let $s\in  C_{\loc}^{\log}(\rn)$ and $d\in[C_{\log}(s),\fz)$, where
$C_{\log}(s)$ denotes the constant as in \eqref{ve1} with $g$ replaced by $s$.
Then, for all $x,\, y\in\rn$ and $v\in\nn$,
$2^{vs(x)}\eta_{v,m+d}(x-y)\le C2^{vs(y)}\eta_{v,m}(x-y)$
with $C$ being a positive constant independent of $x,\ y$ and $v$;
moreover, for all nonnegative measurable functions $f$,
it holds true that
$$2^{vs(x)}\eta_{v,m+d}\ast f(x)\le C\eta_{v,m}\ast(2^{vs(\cdot)}f)(x),\quad x\in\rn.$$
\end{lemma}
\begin{remark}\label{r-3.10x}
Using the same notion as in Lemma \ref{l-eta}, if $\lz\in[2^{-v},2^{-v}+\theta]$
with $\theta\in[0,\fz)$, then, by an argument similar to that used in
the proof of Lemma \ref{l-eta}, we conclude that there exists a positive constant
$C$ such that, for all $x\in\rn$,
$$\lz^{-s(x)}\eta_{v,m+d}\ast f(x)\le C\eta_{v,m}\ast(\lz^{-s(\cdot)}f)(x).$$
\end{remark}

It is well known that the boundedness of the Hardy-Littlewood maximal operator
plays a key role in the study of the classical theory of function spaces.
However, in the case of variable function spaces, such boundedness is usually absence.
For example, the Hardy-Littlewood maximal operator is in general not
bounded on the mixed Lebesgue-sequence space
$\ell^{q(\cdot)}(L^{p(\cdot)}(\rn))$ (see \cite[Example 4.1]{ah10}).
As a suitable substitute, a convolution with radical decreasing functions fits
very well into this scheme. Indeed, we have the following Lemma \ref{l-conv-ineq},
which is just \cite[Lemma 4.7]{ah10} (see also \cite[Lemma 10]{kv12}).

\begin{lemma}\label{l-conv-ineq}
Let $p,\ q\in  C^{\log}(\rn)$ satisfy $p_-,\ q_-\in[1,\fz]$ and
$m\in(n+C_{\log}(1/q),\fz)$, where $C_{\log}(1/q)$
denotes the constant as in \eqref{ve1} with $g$ replaced by $1/q$.
 Then there exists a positive constant $C$ such that,
for all sequences $\{f_v\}_{v\in\nn}$ of measurable functions,
$$\lf\|\lf\{\eta_{v,m}\ast f_v\r\}_{v\in\nn}
\r\|_{\ell^{q(\cdot)}(L^{p(\cdot)}(\rn))}
\le C\lf\|\lf\{f_v\r\}_{v\in\nn}\r\|_{\ell^{q(\cdot)}(L^{p(\cdot)}(\rn))}.$$
\end{lemma}

\begin{remark}\label{re-norm}
In Lemma \ref{l-conv-ineq},
we require that $p_-,\ q_-\ge1$.
However, the following observation that, for all $r\in(0,\fz)$ and
sequences $\{g_v\}_{v\in\nn}\subset \ell^{q(\cdot)}(\vlp)$,
\begin{equation*}
\lf\|\{g_v\}_{v\in\nn}\r\|_{\ell^{q(\cdot)}(L^{p(\cdot)}(\rn))}
=\lf\|\{|g_v|^r\}_{v\in\nn}\r\|_{\ell^{\frac{q(\cdot)}r}
(L^{\frac{p(\cdot)}r})(\rn)}^\frac1r
\end{equation*}
makes it possible to apply Lemma \ref{l-conv-ineq}
even when $p_-,\ q_-\in(0,1)$.
\end{remark}

\begin{lemma}\label{l-inf-norm}
Let $p,\ q\in\cp(\rn)$, $q_+\in(0,\fz)$ and $f$ be a measurable function on $\rn$.

{\rm(i)} If $\|f\|_{\vlp}\le1$, then
$\||f|^{q(\cdot)}\|_{L^{\frac{p(\cdot)}{q(\cdot)}}(\rn)}\le \|f\|_{\vlp}^{q_-}$.

{\rm (ii)} If $\|f\|_{\vlp}>1$, then
$\||f|^{q(\cdot)}\|_{L^{\frac{p(\cdot)}{q(\cdot)}}(\rn)}\le \|f\|_{\vlp}^{q_+}$.

{\rm (iii)} If $\||f|^{q(\cdot)}\|_{L^{\frac{p(\cdot)}{q(\cdot)}}(\rn)}\ge1$, then
$\|f\|_{\vlp}\le \||f|^{q(\cdot)}\|_{L^{\frac{p(\cdot)}{q(\cdot)}}(\rn)}^{1/q_-}$.

{\rm (iv)} If $\||f|^{q(\cdot)}\|_{L^{\frac{p(\cdot)}{q(\cdot)}}(\rn)}<1$, then
$\|f\|_{\vlp}\le\||f|^{q(\cdot)}\|_{L^{\frac{p(\cdot)}{q(\cdot)}}(\rn)}^{1/q_+}$.
\end{lemma}
\begin{proof}
By similarity, we only prove (i) and (iii). Let $f\in\vlp$.
Then, by Remark \ref{r-modular} and the fact that
$\|\frac f{\|f\|_{\vlp}}\|_{\vlp}=1,$
 we see that
$\varrho_{p(\cdot)}(f/\|f\|_{\vlp})\le1$. Thus, if $\|f\|_{\vlp}\le1$, then
\begin{eqnarray*}
\varrho_{p(\cdot)}\lf(\frac{f}{[\|f\|_{\vlp}^{q_-}]^{1/q(\cdot)}}\r)
&&\le \varrho_{p(\cdot)}\lf(\frac{f}{[\|f\|_{\vlp}^{q_-}]^{1/q_-}}\r)
=\varrho_{p(\cdot)}\lf(\frac{f}{\|f\|_{\vlp}}\r)\le1,
\end{eqnarray*}
which implies that
$\||f|^{q(\cdot)}\|_{L^{\frac{p(\cdot)}{q(\cdot)}}(\rn)}\le \|f\|_{\vlp}^{q_-}$
and then completes the proof of (i).

For (iii), if
$\||f|^{q(\cdot)}\|_{L^{\frac{p(\cdot)}{q(\cdot)}}(\rn)}\ge1$, then, for all
$\lz>\||f|^{q(\cdot)}\|_{L^{\frac{p(\cdot)}{q(\cdot)}}(\rn)}$,
$$\varrho_{p(\cdot)}\lf(f/\lz^{1/q_-}\r)
\le\varrho_{p(\cdot)}\lf(f/\lz^{1/q(\cdot)}\r)\le1,$$
which implies that $\|f\|_{\vlp}\le \lz^{\frac1{q_-}}$.
By this and the arbitrariness of
$\lz>\||f|^{q(\cdot)}\|_{L^{\frac{p(\cdot)}{q(\cdot)}}(\rn)},$
we conclude that (iii) holds true, which completes the proof of
Lemma \ref{l-inf-norm}.
\end{proof}

For a sequence $t=\{t_Q\}_{Q\in\cq^\ast}\subset \cc$,
$r\in(0,\fz)$ and $\lz\in(0,\fz)$,
let $t_{r,\lz}^\ast:=\{(t_{r,\lz}^\ast)_Q\}_{Q\in\cq^\ast}$, where,
for all $Q\in\cq^\ast$,
$$(t_{r,\lz}^\ast)_Q:=\lf\{\sum_{R\in \cq^\ast,\, \ell(R)=\ell(Q)}
\frac{|t_R|^r}{[1+\{\ell(R)\}^{-1}|x_R-x_Q|]^\lz}\r\}^{\frac1r}.$$

\begin{lemma}\label{l-estimate1}
Let $p,\ q$, $s$ and $\phi$ be as in Definition \ref{def-b},
$r\in(0,\min\{p_-,q_-\})$
and $$\lz\in(2n+C_{\log}(s)+2r\log_2c_1,\fz),$$
where $C_{\log}(s)$ denotes the constant as in \eqref{ve1}
with $g$ replaced by $s$,
and $c_1$ is as in \eqref{phi-1}.
Then there exists a constant $C\in[1,\fz)$ such that,
for all $t\in b_{p(\cdot),q(\cdot)}^{s(\cdot),\phi}(\rn)$,
\begin{equation}\label{estimate5}
\|t\|_{b_{p(\cdot),q(\cdot)}^{s(\cdot),\phi}(\rn)}
\le \|t_{r,\lz}^\ast\|_{b_{p(\cdot),q(\cdot)}^{s(\cdot),\phi}(\rn)}
\le C\|t\|_{b_{p(\cdot),q(\cdot)}^{s(\cdot),\phi}(\rn)}.
\end{equation}
\end{lemma}
\begin{proof}
To prove this lemma, it suffices to show the second inequality of
\eqref{estimate5} since the first one holds true obviously.
We first claim that, for all
$t\in b_{p(\cdot),q(\cdot)}^{s(\cdot),1}(\rn)$,
$\|t_{r,\lz}^\ast\|_{b_{p(\cdot),q(\cdot)}^{s(\cdot),1}(\rn)}
\ls\|t\|_{b_{p(\cdot),q(\cdot)}^{s(\cdot),1}(\rn)}$.
Indeed, observe that, for all $r\in(0,\min\{p_-,q_-\})$,
$Q\in\cq^\ast$ and $x\in Q$,
\begin{equation*}
(t_{r,\lz}^\ast)_Q\sim\lf[\eta_{j_Q,\lz}\ast
\lf(\sum_{R\in\cq^\ast,\,\ell(R)=2^{-{j_Q}}}|t_R|^r\chi_R\r)(x)\r]^\frac1r.
\end{equation*}
Thus, by Lemma \ref{l-eta}, Remark \ref{r-defi}(iv) and Lemma \ref{l-conv-ineq},
 we see that
\begin{eqnarray*}
\lf\|t_{r,\lz}^\ast\r\|_{b_{p(\cdot),q(\cdot)}^{s(\cdot),1}(\rn)}
&&\ls\lf\|\lf\{\eta_{j,\lz-C_{\log}(s)}\ast\lf(\lf[2^{js(\cdot)}
\sum_{R\in\cq^\ast,\,\ell(R)=2^{-j}}|t_R|\wz\chi_R\r]^r\r)\r\}_{j\in\zz_+}
\r\|_{\ell^{\frac{q(\cdot)}r}(L^{\frac{p(\cdot)}r}(\rn))}^\frac1r\\
&&\ls\lf\|\lf\{2^{js(\cdot)}\sum_{R\in\cq^\ast,\,\ell(R)=2^{-j}}
|t_R|\wz\chi_R\r\}_{j\in\zz_+}\r\|_{\ell^{q(\cdot)}(L^{p(\cdot)}(\rn))}
\sim\|t\|_{b_{p(\cdot),q(\cdot)}^{s(\cdot),1}(\rn)},
\end{eqnarray*}
which proves the above claim.

For all $P\in\cq$ and $Q\in\cq^\ast$, let
$v_{Q}^P:=t_Q$ if $Q\subset 4P$ and $v_{Q}^P:=0$ otherwise, and let
$u_{Q}^P:=t_Q-v_{Q}^P$.
Let $v^P:=\{v_{Q}^P\}_{Q\in\cq^\ast}$ and $u^P:=\{u_{Q}^P\}_{Q\in\cq^\ast}$.
Then, we have
\begin{eqnarray}\label{estimate6}
\qquad\lf\|t_{r,\lz}^\ast\r\|_{\beve}
&&\le\sup_{P\in\cq}\lf\{\frac1{\phi(P)}\lf\|
\lf\{\sum_{\gfz{Q\in\cq^\ast,Q\subset P}{\ell(Q)=2^{-j}}}
|Q|^{-\frac{s(\cdot)}{n}}|((v^P)_{r,\lz}^\ast)_Q|\wz\chi_Q\r\}_{j\ge(j_P\vee0)}
\r\|_{\ell^{q(\cdot)}(L^{p(\cdot)}(P))}\r.\\
&&\hs+\lf.\frac1{\phi(P)}\lf\|
\lf\{\sum_{\gfz{Q\in\cq^\ast,Q\subset P}{\ell(Q)=2^{-j}}}
|Q|^{-\frac{s(\cdot)}{n}}|((u^P)_{r,\lz}^\ast)_Q|\wz\chi_Q\r\}_{j\ge(j_P\vee0)}
\r\|_{\ell^{q(\cdot)}(L^{p(\cdot)}(P))}\r\}\noz\\
&&=:\sup_{P\in\cq}\lf({\rm I}_{P,1}+{\rm I}_{P,2}\r).\noz
\end{eqnarray}
By the above claim, \eqref{phi-1} and Remark \ref{r-lattice}, we find that
\begin{eqnarray}\label{estimate7}
{\rm I}_{P,1}
&&\le\frac1{\phi(P)}\lf\|(v^P)_{r,\lz}^\ast\r\|_{b_{p(\cdot),q(\cdot)}
^{s(\cdot),1}(\rn)}
\ls\frac1{\phi(P)}\lf\|v^P\r\|_{b_{p(\cdot),q(\cdot)}
^{s(\cdot),1}(\rn)}\\
&&\ls\frac1{\phi(4P)}
\lf\|\lf\{\sum_{\gfz{Q\in\cq^\ast,Q\subset 4P}{\ell(Q)=2^{-j}}}
|Q|^{-\frac{s(\cdot)}{n}}|t_Q|\wz\chi_Q\r\}_{j\ge(j_{4P}\vee0)}
\r\|_{\ell{q(\cdot)}(L^{P(\cdot)}(4P))}\ls\|t\|_{\beve}.\noz
\end{eqnarray}

To estimate ${\rm I}_{P,2}$, we only consider the case that $q_+\in(0,\fz)$,
since the proof of the case that $q_+=\fz$
is similar, the details being omitted.
Without loss of generality, we may assume that $\|t\|_{\beve} =1$ and prove that
${\rm I}_{P,2}\ls1$.
To this end, it suffices to show that
\begin{equation*}
\lf\|\lf\{\sum_{\gfz{Q\in\cq^\ast,\,Q\subset P}{\ell(Q)=2^{-j}}}
\frac{\chi_P}{\phi(P)}|Q|^{-\frac{s(\cdot)}n}((u^P)_{r,\lz}^\ast)_Q\wz \chi_Q\r\}
_{j\ge(j_P\vee 0)}\r\|_{\ell^{q(\cdot)}(\vlp)}\ls1.
\end{equation*}
By \eqref{mixed-x}, and (i) and (ii) of Remark \ref{re-mixed},
we see that the above inequality is equivalent to that there exists some large
positive constant $C_0$ such that
\begin{equation*}
\sum_{j=(j_P\vee0)}^\fz
\lf\|\lf[\sum_{\gfz{Q\in\cq^\ast,\,Q\subset P}{\ell(Q)=2^{-j}}}
\frac{\chi_P}{C_0\phi(P)}|Q|^{-\frac{s(\cdot)}n}((u^P)_{r,\lz}^\ast)_Q\wz \chi_Q
\r]^{q(\cdot)}\r\|_{L^{\frac{p(\cdot)}{q(\cdot)}}(\rn)}\le1,
\end{equation*}
which, by Lemma \ref{l-inf-norm}(i), is a consequence of
\begin{equation}\label{estimate4}
{\rm J}_P:=\sum_{j=(j_P\vee0)}^\fz\lf\|\sum_{\gfz{Q\in\cq^\ast,\,Q\subset P}
{\ell(Q)=2^{-j}}}
\frac{\chi_P}{C_0\phi(P)}|Q|^{-\frac{s(\cdot)}n}((u^P)_{r,\lz}^\ast)_Q
\wz \chi_Q\r\|_{\vlp}^{q_-}\le1.
\end{equation}

Now we show \eqref{estimate4}.
Since $\|t\|_{\beve} =1$, it follows that, for all $\wz P\in\cq$,
\begin{equation*}
\lf\|\lf\{\sum_{\gfz{Q\in\cq^\ast,Q\subset \wz P}{\ell(Q)=2^{-j}}}
[\phi(\wz P)]^{-1}\chi_{\wz P}|Q|^{-\frac{s(\cdot)}n}|t_{Q}|
\wz \chi_{Q}\r\}_{j\ge(j_{\wz P}\vee 0)}\r\|_{\ell^{q(\cdot)}(L^{p(\cdot)}(\wz P))}\le1,
\end{equation*}
which, together with \eqref{mixed-x}, and (i) and (ii) of Remark \ref{re-mixed},
implies that
\begin{equation*}
\sum_{j=(j_{\wz P}\vee0)}^\fz
\lf\|\lf[\sum_{\gfz{Q\in\cq^\ast,\,Q\subset \wz P}{\ell(Q)=2^{-j}}}
[\phi(\wz P)]^{-1}\chi_{\wz P}|Q|^{-\frac{s(\cdot)}n}|t_Q|\wz \chi_Q
\r]^{q(\cdot)}\r\|_{L^{\frac{p(\cdot)}{q(\cdot)}}(\rn)}\le1.
\end{equation*}
From this, and (iii) and (iv) of Lemma \ref{l-inf-norm}, we deduce that,
for all $\wz P\in\cq$ and $j\ge(j_{\wz P}\vee0)$,
\begin{equation}\label{estimate3}
\lf\|\sum_{\gfz{Q\in\cq^\ast,\,Q\subset \wz P}{\ell(Q)=2^{-j}}}
[\phi(\wz P)]^{-1}|Q|^{-\frac{s(\cdot)}n}|t_Q|\wz \chi_Q
\r\|_{L^{p(\cdot)}(\wz P)}\le1.
\end{equation}

For the given $P\in\cq$, $i\in\zz_+$ and $l\in\zz^n$, let
$$A(i,l,P):=\lf\{R\in\cq^\ast:\ \ell(R)=2^{-i}\ell(P),\
R\subset P+l\ell(P)\r\}.$$
Then we see that
\begin{eqnarray*}
\wz {\rm J}_{ P}
:=&&\sum_{j=(j_P\vee0)}^\fz\lf\|\sum_{\gfz{Q\in\cq^\ast,\,Q\subset P}
{\ell(Q)=2^{-j}}}\chi_P[\phi(P)]^{-1}|Q|^{-\frac{s(\cdot)}{n}}
(u_{r,\lz}^\ast)_Q\wz\chi_Q\r\|_{\vlp}^{q_-}\\
\ls&&\sum_{i=0}^\fz\lf\|\sum_{\gfz{Q\in\cq^\ast,\,Q\subset P}
{\ell(Q)=2^{-i}\ell(P)}}\frac{\chi_P}{\phi(P)}|Q|^{-[\frac{s(\cdot)}n+\frac12]}\r.\\
&&\hs\hs\times\lf.\lf[\sum_{l\in\zz^n,\,|l|\ge2}\sum_{R\in A(i,l,P)}
\frac{|u_R|^r}{(1+\{\ell(Q)\}^{-1}|x_R-x_Q|)^\lz}\r]^\frac1r\chi_Q\r\|_{\vlp}^{q_-}.
\end{eqnarray*}
Notice that, for all $i\in\zz_+$, $l\in\zz^n$ and $x\in Q\in\cq^\ast$
with $\ell(Q)=2^{-i}\ell(P)$,
\begin{equation*}
\sum_{R\in A(i,l,P)}\frac{|u_R|^r}{[1+\{\ell(Q)\}^{-1}|x_R-x_Q|]^m}
\sim\eta_{j_Q,m}\ast\lf(\lf[\sum_{R\in A(i,l,P)}|u_R|\chi_R\r]^r\r)(x),
\end{equation*}
where $m\in(n+C_{\log}(s),\fz)$ is chosen such that
$\lz>m+n+2r\log_2c_1$.
Notice that, when $|l|\ge2$, $1+\{\ell(Q)\}^{-1}|x_R-x_Q|\sim2^i|l|$.
Thus, by Lemma \ref{l-conv-ineq}, we know that
\begin{eqnarray*}
\wz{\rm J}_{ P}
&&\ls\sum_{i=0}^\fz\lf\|\sum_{\gfz{l\in\zz^n}{|l|\ge2}}
\frac{(2^i|l|)^{m-\lz}}{[\phi(P)]^{r}}\eta_{i+j_P,m-C_{\log}(s)}
\ast\lf(\lf[\sum_{R\in A(i,l,P)}
|R|^{-\frac{s(\cdot)}nr}|u_R|\wz\chi_R\r]^r\r)\r\|
_{L^{\frac{p(\cdot)}{r}}(\rn)}^{\frac{q_-}r}\\
&&\ls\sum_{i=0}^\fz\lf\{\sum_{\gfz{l\in\zz^n}{|l|\ge2}}(2^i|l|)^{m-\lz}
\lf[\frac{\phi(P+l\ell(P))}{\phi(P)}\r]^r\lf\|\sum_{R\in A(i,l,P)}\frac{\chi_{P+l\ell(P)}}
{\phi(P+l\ell(P))}|R|^{-\frac{s(\cdot)}{n}}|t_R|\wz \chi_R\r\|_{\vlp}^r
\r\}^{\frac{q_-}{r}},
\end{eqnarray*}
which, combined with \eqref{estimate3} and Lemma \ref{l-esti-cube},
implies that
\begin{eqnarray*}
\wz{\rm J}_{ P}
\ls\sum_{i=0}^\fz\lf\{\sum_{l\in\zz^n,\,|l|\ge2}
2^{i(m-\lz)}|l|^{m+2r\log_2c_1-\lz}\r\}^{\frac{q_-}r}\sim1.
\end{eqnarray*}
Therefore, there exists a positive constant $C_0$ large enough such that
\eqref{estimate4} holds true for all $P\in\cq$ and hence
\begin{equation}\label{estimate1}
{\rm I}_{P,2}\ls\|t\|_{\beve}.
\end{equation}

Combining \eqref{estimate6}, \eqref{estimate7} and \eqref{estimate1},
we conclude that
$$\|t_{r,\lz}^\ast\|_{\beve}\le\sup_{P\in\cq}\lf({\rm I}_{P,1}+{\rm I}_{P,2}\r)
\ls\|t\|_{\beve},$$
which completes the proof of Lemma \ref{l-estimate1}.
\end{proof}

Now we come to prove the main result of this section.
\begin{proof}[Proof of Theorem \ref{t-transform}]
We first show that $S_\vz$ is bounded from $\bbeve$ to $\beve$.
Let $f\in\bbeve$, $r\in(0,\frac12\min\{p_-,q_-,2\})$ and
$m\in(n+C_{\log}(s)+C_{\log}(r/q)+\log_2c_1,\fz)$.
Then, by Lemma \ref{l-r-trick}, we see that,
for all $Q_{jk}\in\cq^\ast$ and $x\in Q_{jk}$,
\begin{equation*}
|\vz_j\ast f(x_{Q_{jk}})|^r
\ls2^{jn}\sum_{l\in\zz^n}\int_{Q_{j(k+l)}}
\frac{|\vz_j\ast f(y)|^r}{(1+2^j|x-y|)^{4m}}\,dy,
\end{equation*}
which, together with the fact that $1+2^j|x-y|\sim1+|l|$ when $x\in Q_{jk}$ and
 $y\in Q_{j(k+l)}$, implies that
\begin{equation*}
|\vz_j\ast f(x_{Q_{jk}})|
\ls\lf[\sum_{l\in\zz^n}(1+|l|)^{-m}\eta_{j,3m}\ast|(\vz_j\ast
f)\chi_{Q_{j(k+l)}}|^r(x)\r]^\frac1r.
\end{equation*}
From this, Lemma \ref{l-eta} and Remark \ref{re-mixed}(iv), we deduce that
\begin{eqnarray*}
&&\|S_\vz(f)\|_{\beve}\\
&&\hs\ls\sup_{P\in\cq}\frac1{\phi(P)}
\lf\|\lf\{\sum_{k\in\zz^n}
\lf[\sum_{l\in\zz^n}\dfrac{2^{jrs(\cdot)}}{(1+|l|)^m}\eta_{j,3m}\ast|(\vz_j\ast
f)\chi_{3n|l|P}|^r\r]^\frac1r\chi_{Q_{jk}}\r\}_{j\ge(j_P\vee0)}
\r\|_{\ell^{q(\cdot)}(L^{p(\cdot)}(P))}\\
&&\hs\ls\sup_{P\in\cq}\frac1{\phi(P)}
\lf\|\lf\{\sum_{l\in\zz^n}\frac{2^{jrs(\cdot)}}{(1+|l|)^{m}}
\eta_{j,3m}\ast|(\vz_j\ast f)\chi_{3n|l|P}|^r\r\}_{j\ge(j_P\vee0)}
\r\|_{\ell^{\frac{q(\cdot)}r}(L^{\frac{p(\cdot)}r}(P))}^{\frac1r}\\
&&\hs\ls\lf[\sum_{l\in\zz^n}(1+|l|)^{-m}\sup_{P\in\cq}\frac1{\{\phi(P)\}^r}\lf\|\lf\{\eta_{j,2m}\ast|2^{js(\cdot)}
(\vz_j\ast f)\chi_{3n|l|P}|^r
\r\}_{j\ge(j_P\vee0)}\r\|_{\ell^{\frac{q(\cdot)}r}(L^{\frac{p(\cdot)}r}(P))}
\r]^\frac1r,
\end{eqnarray*}
which, combined with Lemmas \ref{l-conv-ineq} and \ref{l-esti-cube}, implies that
\begin{eqnarray*}
\|S_\vz(f)\|_{\beve}&&\ls\lf[\sum_{l\in\zz^n}(1+|l|)^{-m}\sup_{P\in\cq}\frac1{\phi(P)^r}
\lf\|\lf\{2^{js(\cdot)}|\vz_j\ast f|\r\}_{j\ge(j_P\vee0)}
\r\|_{\ell^{q(\cdot)}(L^{p(\cdot)}
({3n|l|P}))}^r\r]^\frac1r\\
&&\ls\|f\|_{\bbeve}\lf\{\sum_{l\in\zz^n}(1+|l|)^{-m}
(1+|l|)^{r\log_2c_1}\r\}^\frac1r
\ls\|f\|_{\bbeve}.
\end{eqnarray*}
Therefore, $S_\vz$ is bounded from $\bbeve$ to $\beve$.

The boundedness of $T_\psi$ from $\beve$ to $\bbeve$
is deduced from an argument similar to that used in the proof of \cite[Theorem 2.1]{ysiy}.
Indeed, by repeating the argument used in the
proof of \cite[Theorem 2.1]{ysiy},
with \cite[Lemmas 2.7 and 2.8]{ysiy} therein replaced by
Lemmas \ref{l-welld} and \ref{l-estimate1}, we conclude that
$T_\psi$ is bounded from $\beve$ to $\bbeve$, the details being omitted.
Finally, by the Calder\'on reproducing formula
(see, for example, \cite[Lemma 2.3]{ysiy}),
we know that $T_\psi\circ S_\vz$ is the
identity on $\bbeve$, which completes the proof of Theorem \ref{t-transform}.
\end{proof}

\section{Embeddings\label{s4}}

\hskip\parindent
In this section, we prove some basic properties and embeddings between $\bbeve$
and $\btlve$. Recall that the \emph{Triebel-Lizorkin-type space with variable exponents},
$\btlve$, is defined to be the set of all
$f\in\cs'(\rn)$ such that
$$\|f\|_{F_{p(\cdot),q(\cdot)}^{s(\cdot),\phi}(\rn)}
:=\sup_{P\in\cq}\frac1{\phi(P)}
\lf\|\lf\{\sum_{j=\max\{j_P,0\}}^\fz\lf[2^{js(\cdot)}|\vz_j\ast f(\cdot)|\r]
^{q(\cdot)}\r\}^{\frac1{q(\cdot)}}\r\|_{L^{p(\cdot)}(P)}<\fz,$$
where $\vz_0$ is replaced by $\Phi$ and the supremum is taken over all dyadic
cubes $P$ in $\rn$, which was introduced in \cite{yyz14}.

\begin{proposition}\label{p-embed1}
Let $\phi\in\cg(\urn)$,
$s,\ s_0,\ s_1\in C_{\rm loc}^{\log}(\rn)\cap L^\fz(\rn)$ and
$p$, $q,\ q_0,\ q_1\in C^{\log}(\rn)$.

{\rm(i)} If $q_0\le q_1$, then
$B_{p(\cdot),q_0(\cdot)}^{s(\cdot),\phi}(\rn)
\hookrightarrow B_{p(\cdot),q_1(\cdot)}^{s(\cdot),\phi}(\rn).$

{\rm(ii)} If $(s_0-s_1)_->0$, then
$B_{p(\cdot),q_0(\cdot)}^{s_0(\cdot),\phi}(\rn)
\hookrightarrow B_{p(\cdot),q_1(\cdot)}^{s_1(\cdot),\phi}(\rn).$

{\rm(iii)} If $p_+,\ q_+\in(0,\fz)$, then
$$B_{p(\cdot),\min\{p(\cdot),\,q(\cdot)\}}^{s(\cdot),\phi}(\rn)
\hookrightarrow F_{p(\cdot),q(\cdot)}^{s(\cdot),\phi}(\rn)
\hookrightarrow
B_{p(\cdot),\max\{p(\cdot),\,q(\cdot)\}}^{s(\cdot),\phi}(\rn).$$
In particular, if $p_+\in(0,\fz)$, then
$B_{p(\cdot),p(\cdot)}^{s(\cdot),\phi}(\rn)
=F_{p(\cdot),p(\cdot)}^{s(\cdot),\phi}(\rn).$
\end{proposition}
\begin{proof}
The proof of this proposition is similar to that of
\cite[Theorem 6.1]{ah10} and we only give the proof of (iii).
Let $f_j(x):=2^{js(x)}|\vz_j\ast f(x)|$ for all $x\in\rn$,
$f\in\cs'(\rn)$ and $j\in\zz_+$. To prove the first embedding of (iii),
we let $r(\cdot):=\min\{p(\cdot),\,q(\cdot)\}$ and
$f\in B_{p(\cdot),r(\cdot)}^{s(\cdot),\phi}(\rn)$. Without loss of generality,
we may assume that $\|f\|_{B_{p(\cdot),r(\cdot)}^{s(\cdot),\phi}(\rn)}=1$
and prove that
$\|f\|_{F_{p(\cdot),q(\cdot)}^{s(\cdot),\phi}(\rn)}\ls1$.
Obviously, for all $P\in\cq$,
\begin{equation*}
\lf\|\lf\{[\phi(P)]^{-1}\chi_Pf_j\r\}_{j\ge(j_P\vee0)}
\r\|_{\ell^{r(\cdot)}(\vlp)}\le1,
\end{equation*}
which, together with \eqref{mixed-x}, Remarks \ref{re-vlp}(i) and \ref{re-mixed}(i),
implies that
\begin{equation*}
\lf\|\sum_{j=(j_P\vee0)}^\fz\lf[[\phi(P)]^{-1}\chi_Pf_j\r]^{r(\cdot)}
\r\|_{L^{\frac{p(\cdot)}{r(\cdot)}}(\rn)}
\le\sum_{j=(j_P\vee0)}^\fz\lf\|\lf\{[\phi(P)]^{-1}\chi_Pf_j\r\}^{r(\cdot)}
\r\|_{L^{\frac{p(\cdot)}{r(\cdot)}}(\rn)}\le1.
\end{equation*}
Then, by Remark \ref{r-modular} and the fact that, for all $d\in(0,1]$ and
$\{a_j\}_{j\in\nn}\subset\cc$,
\begin{equation}\label{simple-ineq}
\lf(\sum_{j\in\nn}|a_j|\r)^d\le\sum_{j\in\nn}|a_j|^d,
\end{equation}
we find that, for all $P\in\cq$,
\begin{eqnarray*}
&&\varrho_{p(\cdot)}\lf(\lf[\sum_{j=(j_P\vee0)}^\fz\lf\{[\phi(P)]^{-1}
\chi_Pf_j\r\}^{q(\cdot)}\r]^{\frac1{q(\cdot)}}\r)
\le\varrho_{\frac{p(\cdot)}{r(\cdot)}}
\lf(\sum_{j=(j_P\vee0)}^\fz\lf\{[\phi(P)]^{-1}
\chi_Pf_j\r\}^{r(\cdot)}\r)\le1,
\end{eqnarray*}
which implies that
\begin{eqnarray*}
\frac1{\phi(P)}\lf\|\lf\{\sum_{j=(j_P\vee0)}^\fz
\lf[2^{js(\cdot)}|\vz_j\ast f|\r]^{q(\cdot)}\r\}^{\frac1{q(\cdot)}}
\r\|_{L^{p(\cdot)}(P)}\le1.
\end{eqnarray*}
 Therefore, $\|f\|_{\btlve}\le1$, which completes the proof
of the first embedding of (iii).

For the second embedding of (iii), let $f\in \btlve$ and
$\alpha(\cdot):=\max\{p(\cdot),\,q(\cdot)\}$.
Without loss of generality, we may assume that $\|f\|_{\btlve}=1$ and show that
$\|f\|_{B_{p(\cdot),\alpha(\cdot)}^{s(\cdot),\phi}(\rn)}\ls1$.
Since $\|f\|_{\btlve}=1$, we know that, for all $P\in\cq$,
\begin{equation*}
\lf\|\lf\{\sum_{j=(j_P\vee0)}^\fz\lf([\phi(P)]^{-1}\chi_Pf_j\r)^{q(\cdot)}
\r\}^{\frac1{q(\cdot)}}\r\|_{\vlp}\le1,
\end{equation*}
which, combined with \eqref{simple-ineq} and Remark \ref{r-modular},
 implies that, for all $P\in\cq$,
\begin{eqnarray*}
&&\varrho_{\frac{p(\cdot)}{\alpha(\cdot)}}
\lf(\sum_{j=(j_P\vee0)}^\fz\lf\{[\phi(P)]^{-1}
\chi_Pf_j\r\}^{\alpha(\cdot)}\r)
\le\varrho_{p(\cdot)}\lf(\lf[\sum_{j=(j_P\vee0)}^\fz\lf\{[\phi(P)]^{-1}
\chi_Pf_j\r\}^{q(\cdot)}\r]^{\frac1{q(\cdot)}}\r)\le1.
\end{eqnarray*}
From this, Remark \ref{re-mixed}(ii) and Remark \ref{re-vlp}(iv), we deduce that
\begin{eqnarray*}
\varrho_{\ell^{\alpha(\cdot)}(L^{p(\cdot)})}
\lf(\lf\{[\phi(P)]^{-1}\chi_Pf_j\r\}_{j\ge(j_P\vee 0)}\r)
&&=\sum_{j=(j_P\vee0)}^\fz\lf\|\lf([\phi(P)]^{-1}\chi_Pf_j\r)^{\alpha(\cdot)}
\r\|_{L^{\frac{p(\cdot)}{\alpha(\cdot)}}(\rn)}\\
&&\le\lf\|\sum_{j=(j_P\vee0)}^\fz\lf([\phi(P)]^{-1}\chi_Pf_j\r)^{\alpha(\cdot)}
\r\|_{L^{\frac{p(\cdot)}{\alpha(\cdot)}}(\rn)}\le1,
\end{eqnarray*}
which implies that $\|f\|_{B_{p(\cdot),\alpha(\cdot)}^{s(\cdot),\phi}(\rn)}\le1$
and hence completes the proof of Lemma \ref{p-embed1}.
\end{proof}

The Sobolev-type embedding of $B_{p(\cdot),q(\cdot)}^{s(\cdot)}(\rn)$
(see \cite[Theorem 6.4]{ah10}) shows that it is reasonable and necessary to
consider the Besov spaces with variable smoothness and integrability.
For $\bbeve$, we also have the following Sobolev-type embeddings.

\begin{proposition}\label{p-se-embed}
Let $\phi\in\cg(\urn)$,
$s_0,\ s_1\in  C_{\rm loc}^{\log}(\rn)\cap L^\fz(\rn)$,
$p_0,\ p_1\in C^{\log}(\rn)$ satisfy that, for all $x\in\rn$,
$s_1(x)\le s_0(x)$ and
$s_0(x)-\frac n{p_0(x)}=s_1(x)-\frac n{p_1(x)}.$
Then
\begin{equation}\label{se-embed}
b_{p_0(\cdot),\fz}^{s_0(\cdot),\phi}(\rn)\hookrightarrow
b_{p_1(\cdot),\fz}^{s_1(\cdot),\phi}(\rn);
\end{equation}
 moreover,
$B_{p_0(\cdot),\fz}^{s_0(\cdot),\phi}(\rn)\hookrightarrow
B_{p_1(\cdot),\fz}^{s_1(\cdot),\phi}(\rn)$.
\end{proposition}
\begin{proof}
To prove this proposition, we only need to show \eqref{se-embed},
since the embedding $B_{p_0(\cdot),\fz}^{s_0(\cdot),\phi}(\rn)\hookrightarrow
B_{p_1(\cdot),\fz}^{s_1(\cdot),\phi}(\rn)$ is a consequence of \eqref{se-embed}
and Theorem \ref{t-transform}.
To prove \eqref{se-embed}, let  $t:=\{t_Q\}_{Q\in\cq^\ast}\in
b_{p_0(\cdot),\fz}^{s_0(\cdot),\phi}(\rn)$ and
 $P\in\cq$ be any given dyadic cube. For all $Q\in\cq^\ast$,
let $u_Q:=t_Q$ when $Q\subset P$ and $u_Q=0$ otherwise. Then,
by the Sobolev-type embedding of
$b_{p(\cdot),\fz}^{s(\cdot)}(\rn)=b_{p(\cdot),\fz}^{s(\cdot),1}(\rn)$
(\cite[Proposition 3.9]{kempka10}), namely,
$b_{p_0(\cdot),\fz}^{s_0(\cdot),1}(\rn)\hookrightarrow
b_{p_1(\cdot),\fz}^{s_1(\cdot),1}(\rn)$, we conclude that
\begin{eqnarray*}
&&\sup_{j\ge(j_P\vee0)}\lf\|\sum_{\gfz{Q\in\cq^\ast,\,Q\subset P}{\ell(Q)=2^{-j}}}
|Q|^{-\frac{s_1(\cdot)}n}|t_Q|\wz\chi_Q\r\|_{L^{p_1(\cdot)}(P)}\\
&&\hs=\sup_{j\ge0}\lf\|\sum_{Q\in\cq^\ast,\,\ell(Q)=2^{-j}}
|Q|^{-\frac{s_1(\cdot)}n}|u_Q|\wz\chi_Q\r\|_{L^{p_1(\cdot)}(\rn)}
= \|u\|_{b_{p_1(\cdot),\fz}^{s_1(\cdot),1}(\rn)}\\
&&\hs\ls\|u\|_{b_{p_0(\cdot),\fz}^{s_0(\cdot),1}(\rn)}
\sim\sup_{j\ge0}\lf\|\sum_{Q\in\cq^\ast,\,\ell(Q)=2^{-j}}
|Q|^{-\frac{s_0(\cdot)}n}|u_Q|\wz\chi_Q\r\|_{L^{p_0(\cdot)}(\rn)}\\
&&\hs\sim\sup_{j\ge(j_P\vee0)}
\lf\|\sum_{\gfz{Q\in\cq^\ast,\,Q\subset P}{\ell(Q)=2^{-j}}}
|Q|^{-\frac{s_0(\cdot)}n}|t_Q|\wz\chi_Q\r\|_{L^{p_0(\cdot)}(P)}.
\end{eqnarray*}
From this, we further deduce that
\begin{eqnarray*}
\|t\|_{b_{p_1(\cdot),\fz}^{s_1(\cdot),\phi}(\rn)}
&&=\sup_{P\in\cq}\frac1{\phi(P)}
\sup_{j\ge(j_P\vee0)}\lf\|\sum_{\gfz{Q\in\cq^\ast,\,Q\subset P}{\ell(Q)=2^{-j}}}
|Q|^{-\frac{s_1(\cdot)}n}|t_Q|\wz\chi_Q\r\|_{L^{p_1(\cdot)}(P)}\\
&&\ls\sup_{P\in\cq}\frac1{\phi(P)}
\sup_{j\ge(j_P\vee0)}
\lf\|\sum_{\gfz{Q\in\cq^\ast,\,Q\subset P}{\ell(Q)=2^{-j}}}
|Q|^{-\frac{s_0(\cdot)}n}|t_Q|\wz\chi_Q\r\|_{L^{p_0(\cdot)}(P)}
\sim\|t\|_{b_{p_0(\cdot),\fz}^{s_0(\cdot),\phi}(\rn)},
\end{eqnarray*}
which implies that \eqref{se-embed} holds true and hence completes the proof of
Proposition \ref{p-se-embed}.
\end{proof}

\begin{theorem}\label{t-sobolev}
Let $\phi\in \cg(\urn)$, $s_0,\ s_1\in C_{\loc}^{\log}(\rn)\cap L^{\fz}(\rn)$ and
$p_0,\ p_1,\ q\in C^{\log}(\rn)$.
Assume that, for all $x\in\rn$, $s_1(x)\le s_0(x)$ and
\begin{equation}\label{sobolev3}
s_0(x)-\frac n{p_0(x)}=s_1(x)-\frac n{p_1(x)}.
\end{equation}
Then
$B_{p_0(\cdot),q(\cdot)}^{s_0(\cdot),\phi}(\rn)
\hookrightarrow B_{p_1(\cdot),q(\cdot)}^{s_1(\cdot),\phi}(\rn).$
\end{theorem}
\begin{proof}
We only give the proof of the case that $q_+\in(0,\fz)$, since the case that $q_+=\fz$
was proved in Proposition \ref{p-se-embed}.
Let $f\in B_{p_0(\cdot),q(\cdot)}^{s_0(\cdot),\phi}(\rn)$
and, for all $j\in\zz_+$ and $x\in\rn$, $g_j(x):=\vz_j\ast f(x)$. Without loss of
generality, we may assume that
$\|f\|_{B_{p_0(\cdot),q(\cdot)}^{s_0(\cdot),\phi}(\rn)}=1$. Next, we
show that $\|f\|_{B_{p_1(\cdot),q(\cdot)}^{s_1(\cdot),\phi}(\rn)}\ls1$.
Obviously, by Remark \ref{r-lattice}, \eqref{mixed-x},
and (i) and (ii) of Remark \ref{re-mixed}, we find that, for all $R\in\cd_0(\rn)$,
\begin{equation}\label{sobolev1}
\sum_{j=(j_R\vee0)}^\fz
\lf\|\lf[\frac{\chi_{R}}{\phi(R)}2^{js_0(\cdot)}|g_j|\r]^{q(\cdot)}
\r\|_{L^{\frac{p_0(\cdot)}{q(\cdot)}}(\rn)}\ls1.
\end{equation}

Let $P\in\cq$ be a given dyadic cube. We claim that
there exists $c\in(0,1)$, independent of $P$, such that,
for all $j\ge [j_P\vee0,\fz)$,
\begin{eqnarray*}
\lf\|\lf[\frac{c\chi_P}{\phi(P)}2^{js_1(\cdot)}|g_j|\r]^{q(\cdot)}
\r\|_{L^{\frac{p_1(\cdot)}{q(\cdot)}}(\rn)}
\le\sum_{i=1}^\fz2^{-i\xi}
\lf\|\lf[\frac{\chi_{P_i}}{\phi(P_i)}2^{js_0(\cdot)}|g_j|\r]^{q(\cdot)}
\r\|_{L^{\frac{p_0(\cdot)}{q(\cdot)}}(\rn)}+2^{-j}=:\delta_j,
\end{eqnarray*}
where $P_i:=2^{i+1+n}P$ and $\xi\in(0,\fz)$.
From this claim and \eqref{sobolev1}, we deduce that
\begin{eqnarray*}
\sum_{j=(j_P\vee0)}^\fz\lf\|\lf[\frac{c\chi_P}{\phi(P)}2^{js_1(\cdot)}|g_j|
\r]^{q(\cdot)}\r\|_{L^{\frac{p_1(\cdot)}{q(\cdot)}}(\rn)}\ls1,
\end{eqnarray*}
which, together with \eqref{mixed-x},
and (i) and (ii) of Remark \ref{re-mixed}, implies that
$$\|f\|_{B_{p_1(\cdot),q(\cdot)}^{s_1(\cdot),\phi}(\rn)}\ls1\sim
\|f\|_{B_{p_0(\cdot),q(\cdot)}^{s_0(\cdot),\phi}(\rn)}.$$

Therefore, it remains to prove the above claim.
Observe that, for all $j\ge[j_P\vee0,\fz)$, $\delta_j\in[2^{-j},2^{-j}+\theta]$
with $\theta\in[0,\fz)$. Then, by Lemma \ref{l-r-trick} and Remark \ref{r-3.10x},
we conclude that, for all $x\in\rn$, $r\in(0,p_-)$ and $m\in(0,\fz)$ large enough,
\begin{eqnarray}\label{sobolev2}
\frac{2^{jr[s_1(x)-\frac n{p_1(x)}]}}{[\phi(P)]^r\delta_j^{r/{q(x)}}}|g_j(x)|^r
&&\ls\eta_{j,2m}\ast\lf(\lf\{\frac{2^{j[s_1(\cdot)-\frac n{p_1(\cdot)}]}}
{\phi(P)\delta_j^{1/{q(\cdot)}}}|g_j|\r\}^r\r)(x)\\
&&\ls\sum_{i=1}^\fz\int_{D_{i,P}}\frac{2^{jn}
(2^{j[s_1(y)-\frac n{p_1(y)}]}|g_j(y)|)^r}
{[\phi(P)]^{r}\delta_j^{\frac r{q(y)}}(1+2^j|x-y|)^{2m}}\,dy
=:\sum_{i=1}^\fz {\rm A}_{j,i}(x),\noz
\end{eqnarray}
where $D_{1,P}:=4\sqrt nP$ and, for all $i\in[2,\fz)$,
$D_{i,P}:=(2^{i+1}\sqrt nP)\backslash (2^i\sqrt nP)$.
For ${\rm A}_{j,1}$, by the H\"older inequality in Remark \ref{re-vlp}(ii),
\eqref{sobolev3}, \eqref{phi-1} and Lemma \ref{l-inf-norm}, we see that
\begin{eqnarray*}
{\rm A}_{j,1}
&&\ls\lf\|\lf[\frac{\chi_{4\sqrt nP}}{\phi(P)\delta_j^{1/q(\cdot)}}
2^{js_0(\cdot)}|g_j|\r]^r\r\|_{L^{\frac{p_0(\cdot)}{r}}(\rn)}
\lf\|\frac{2^{jn}2^{-jnr/p_0(\cdot)}}{(1+2^j|x-\cdot|)^{2m}}
\r\|_{L^{(\frac{p_0(\cdot)}r)^\ast}(\rn)}\\
&&\ls\lf[\frac{\phi(P_1)}{\phi(P)}\r]^r
\lf\|\frac{\chi_{4\sqrt nP}}{\phi(P_1)\delta_j^{1/q(\cdot)}}
2^{js_0(\cdot)}|g_j|\r\|_{L^{p_0(\cdot)}(\rn)}^r\ls1,
\end{eqnarray*}
where the last inequality follows from the definition of $\delta_j$.
Similarly, observe that, for all $x\in P$ and $y\in D_{i,P}$ with $i\ge2$,
$|x-y|\gs 2^{i-j_P}$, then the fact that $j\ge j_P$ further implies that
\begin{eqnarray*}
{\rm A}_{j,i}
&&\ls\frac{2^{(j_P-j)m}}{2^{i(m-\xi/q_-)}}\lf[\frac{\phi(P_i)}{\phi(P)}\r]^r
\lf\|\frac{\chi_{P_i}2^{js_0(\cdot)}|g_j|}{\phi(P_i)\{2^{i\xi}\delta_j\}^{1/q(\cdot)}}
\r\|_{L^{p_0(\cdot)}(\rn)}^r
\lf\|\frac{2^{jn}2^{-jnr/p_0(\cdot)}}{(1+2^j|x-\cdot|)^{m}}
\r\|_{L^{(\frac{p_0(\cdot)}r)^\ast}(\rn)}\\
&&\ls2^{(j_P-j)m}2^{-i(m-r\log_2c_1)}\ls2^{-i(m-\xi/q_--r\log_2c_1)}.
\end{eqnarray*}
Thus, by \eqref{sobolev2}, we conclude that, for all $x\in\rn$,
$$\chi_P(x)[\phi(P)]^{-1}\delta_j^{-1/{q(x)}}
2^{j[s_1(x)-\frac n{p_1(x)}]}|g_j(x)|\ls1.$$
From this, \eqref{sobolev3} and an appropriate choice of
$c\in(0,1)$, we deduce that
\begin{eqnarray*}
&&\lf[\frac{c\chi_P(x)2^{js_1(x)}}{\phi(P)\delta_j^{1/q(x)}}|g_j(x)|\r]^{p_1(x)}\\
&&\hs=c^{p_0(x)}\lf[\frac{\chi_P(x)2^{js_0(x)}}{\phi(P)\delta_j^{1/q(x)}}|g_j(x)|
\r]^{p_0(x)}
\lf[\frac{c\chi_P(x)2^{j[s_1(x)-\frac n{p_1(x)}]}}
{\phi(P)\delta_j^{1/q(x)}}|g_j(x)|\r]^{p_1(x)-p_0(x)}\\
&&\hs\le c^{p_0(x)}\lf[\frac{\chi_P(x)2^{js_0(x)}}{\phi(P)\delta_j^{1/q(x)}}|g_j(x)|
\r]^{p_0(x)}
\le\lf[\frac{\chi_{P_1}(x)2^{js_0(x)}}{\phi(P_1)\{2^\xi\delta_j\}^{1/q(x)}}|g_j(x)|
\r]^{p_0(x)},
\end{eqnarray*}
which, together with the definition of $\delta_j$ and Remark \ref{re-mixed}(ii),
 implies that
the previous claim holds true and hence completes the proof of Theorem \ref{t-sobolev}.
\end{proof}

\begin{remark}
When $\phi\equiv1$, Theorem \ref{t-sobolev} just becomes \cite[Theorem 6.4]{ah10},
which is called the Sobolev inequality therein.
\end{remark}

\section{Equivalent quasi-norms\label{s-equi}}

\hskip\parindent
In this section, we are aimed to characterize $\bbeve$ in terms of the
Peetre maximal functions and establish their atomic characterization
 via Sobolev embeddings.
Following \cite[p.\,19]{dd12}, for all
$f\in\cs'(\rn)$, $a\in(0,\fz)$ and $s:\ \rn\to\rr$,
the \emph{Peetre maximal function} of $f$ is defined by setting, for all $j\in\zz_+$,
\begin{equation*}
\vz_j^{\ast,a}(2^{js(\cdot)}f)(x):=\sup_{y\in\rn}\frac{2^{js(y)}|\vz_j\ast f(y)|}
{(1+2^j|x-y|)^a}.
\end{equation*}

The following Theorem \ref{t-equivalent} is the first main result of this section.
\begin{theorem}\label{t-equivalent}
Let $p,\ q,\ s$, $\phi$ be as in Definition \ref{def-b}
and
\begin{equation}\label{equivalent1-xx}
a\in([n+\log_2c_1]/p_-,\fz).
\end{equation}
Then $f\in\bbeve$ if and only if $f\in\cs'(\rn)$ and
$\|f\|_{\bbeve}^\ast<\fz$,
where
$$\|f\|_{\bbeve}^\ast
:=\sup_{P\in\cq}\frac1{\phi(P)}\lf\|\lf\{
\vz_j^{\ast,a}(2^{js(\cdot)}f)\r\}_{j\ge(j_P\vee0)}\r\|
_{\ell^{q(\cdot)}(L^{p(\cdot)}(P))}.$$
Moreover, for all $f\in\bbeve$,
$\|f\|_{\bbeve}\sim\|f\|_{\bbeve}^\ast$
with equivalent positive constants independent of $f$.
\end{theorem}
\begin{remark}
Theorem \ref{t-equivalent} goes back to \cite[Theorem 1]{dd12} when $\phi\equiv1$.
\end{remark}
To prove Theorem \ref{t-equivalent}, we need some technical lemmas.
For all $r\in(0,\fz)$, denote by $L^r_{\rm loc}(\rn)$ the
\emph{set of all $r$-locally integrable functions} on $\rn$.
Recall that the \emph{Hardy-Littlewood maximal operator} $\cm$ is
defined by setting, for all $f\in L^1_{\rm loc}(\rn)$ and $x\in\rn$,
$$\cm(f)(x):=\sup_{B\ni x}\frac1{|B|}\int_B|f(y)|\,dy,$$
where the supremum is taken over all balls $B$ of $\rn$ containing $x$.

The following Lemma \ref{l-hlo} is just \cite[Theorem 4.3.8]{dhr11}.

\begin{lemma}\label{l-hlo}
Let $p\in  C^{\log}(\rn)$ with $p_-\in(1,\fz]$. Then there exists a positive
constant $C$, independent of $f$, such that, for all $f\in\vlp$,
$\|\cm(f)\|_{\vlp}\le C\|f\|_{\vlp}$.
\end{lemma}

The following technical lemma plays a key role in
the proof of Theorem \ref{t-equivalent}.

\begin{lemma}\label{l-equi}
Let $p,\ q,\ s$, $\phi$ be as in Definition \ref{def-b} and
$a\in(n+\log_2c_1+\vez/q_-,\fz)$ with $\vez\in(0,\fz)$.
Assume that $p_-\in(1,\fz)$, $q_+\in(0,\fz)$ and $f\in \bbeve$ with norm 1.
Then there exists a positive constant $c$ such that,
for all $P\in\cq$ and $j\in\zz_+$ with $j\ge(j_P\vee0)$,
\begin{eqnarray}\label{equi-x}
&&\inf\lf\{\lz_j\in(0,\fz):\
\varrho_{p(\cdot)}\lf(\frac{c\chi_P\vz_j^{\ast,a}(2^{js(\cdot)}f)}
{\phi(P)\lz_j^{1/q(\cdot)}}\r)\le1\r\}\\
&&\hs\le\sum_{k=1}^\fz
\inf\lf\{\eta_j\in(0,\fz):\
\varrho_{p(\cdot)}\lf(\frac{\chi_{P_k^n}2^{js(\cdot)}|\vz_j\ast f|}
{2^{k\vez /q(\cdot)}\phi(P_k^n)\eta_j^{1/q(\cdot)}}\r)\le1\r\}
+2^{-\sigma[j-(j_P\vee0)]},\noz
\end{eqnarray}
where, for all $k\in\nn$, $P_k^n:=2^{k+1+n}P$ and
$\sigma\in(0,\frac{a-n}{4(1/q_--1/q_+)})$.
\end{lemma}
\begin{proof}
Let $\delta_j^P$ be the right hand side term of \eqref{equi-x}.
Then, by Remark \ref{r-lattice}, we easily see that
\begin{eqnarray*}
\delta_j^P&&\le\sum_{k=1}^\fz2^{-k\vez}\frac1{\phi(P_k^n)}
\lf\|\lf\{2^{js(\cdot)}|\vz_j\ast f|\r\}_{j\ge (j_{P_k^n}\vee0)}
\r\|_{\ell^{q(\cdot)}(L^{p(\cdot)}(P_k^n))}+2^{-\sigma[j-(j_P\vee0)]}\\
&&\le\sum_{k=1}^\fz2^{-k\vez}\|f\|_{\bbeve}+2^{-\sigma[j-(j_P\vee0)]}
=1/(2^\vez-1)+2^{-\sigma[j-(j_P\vee0)]},
\end{eqnarray*}
which implies that
\begin{equation}\label{delta}
\delta_j^P\in\lf[2^{-\sigma[j-(j_P\vee0)]},1/(2^\vez-1)+2^{-\sigma[j-(j_P\vee0)]}\r].
\end{equation}
Thus, to prove Lemma \ref{l-equi},
we only need to show that, for some positive constant $c$,
\begin{equation*}
\inf\lf\{\lz_j\in(0,\fz):\ \varrho_{p(\cdot)}
\lf(\frac{c\chi_P(\delta_j^P)^{-1/q(\cdot)}\vz_j^{\ast,a}(2^{js(\cdot)}f)}
{\phi(P)\lz_j^{1/q(\cdot)}}\r)\le1\r\}\le1,
\end{equation*}
which, via Remark \ref{re-mixed}(ii), and (i) and (ii) of Lemma \ref{l-inf-norm},
is a consequence of
\begin{equation}\label{peetre2}
{\rm H}_P:=\lf\|\frac{\chi_P(\delta_j^P)^{-1/q(\cdot)}}{\phi(P)}
\vz_j^{\ast,a}(2^{js(\cdot)}f)\r\|_{L^{p(\cdot)}(\rn)}\ls1.
\end{equation}

Next we prove \eqref{peetre2}.
By Lemma \ref{l-r-trick} and the inequality that, for all $x,\ y,\ z\in\rn$,
$$(1+2^{-j}|x-y|)^{-a}\le(1+2^{-j}|x-z|)^{-a}(1+2^{-j}|z-y|)^a,$$
 we find that, for all $x\in\rn$,
\begin{eqnarray*}
\vz_j^{\ast,a}(2^{js(\cdot)}f)(x)
&&\ls\sup_{y\in\rn}
\int_\rn\frac{2^{jn}2^{js(z)}|\vz_j\ast f(z)|}{(1+2^j|y-z|)^{2a}}\,dz
\frac1{(1+2^j|x-y|)^a}\\
&&\ls\int_\rn\frac{2^{jn}2^{js(z)}|\vz_j\ast f(z)|}{(1+2^j|x-z|)^{a}}\,dz\\
&&\sim\int_{4\sqrt nP}\frac{2^{jn}2^{js(z)}
|\vz_j\ast f(z)|}{(1+2^j|x-z|)^{a}}\,dz
+\sum_{k=2}^\fz\int_{D_{k,P}}\cdots\,dz
=:{\rm A}_{j}(x)+\sum_{k=2}^\fz{\rm A}_{j}^{k}(x),
\end{eqnarray*}
where, for all $k\in\nn\cap[2,\fz)$,
$D_{k,P}:=(2^{k+1}\sqrt nP)\backslash (2^k\sqrt nP)$.
Thus, we obtain
\begin{eqnarray}\label{5.4-x}
{\rm H}_P
&&\le\lf\|\frac{\chi_P{\rm A}_{j}(\cdot)}{[\delta_j^P]^{1/q(\cdot)}\phi(P)}
\r\|_{L^{p(\cdot)}(\rn)}
+\lf\|\frac{\chi_P}{[\delta_j^P]^{1/q(\cdot)}\phi(P)}
\sum_{k=2}^\fz{\rm A}_{j}^{k}(\cdot)\r\|_{L^{p(\cdot)}(\rn)}
=:{\rm H}_{P,1}+{\rm H}_{P,2}.
\end{eqnarray}

We first estimate ${\rm H}_{P,1}$. For all $x\in P$,
we write
\begin{eqnarray}\label{peetre3}
\quad{\rm A}_{j}(x)
\sim\lf\{\int_{B_{-1}^j(x)}+
\sum_{i=0}^\fz\int_{B_i^j(x)}\r\}\frac{2^{jn}2^{js(z)}
|\vz_j\ast f(z)|\chi_{4\sqrt nP}(z)}{(1+2^j|x-z|)^{a}}\,dz
=:{\rm A}_{j,1}(x)+{\rm A}_{j,2}(x),
\end{eqnarray}
where, for all $x\in\rn$,
$B_{-1}^j(x):=B(x,2^{-[j-(j_P\vee0)]/2})$ and, for all $i\in\zz_+$,
$$B_i^j(x):=B\lf(x,2^{-[j-(j_P\vee0)]/2+i+1}\r)\backslash
B\lf(x,2^{-[j-(j_P\vee0)]/2+i}\r).$$
From \eqref{delta}, $q\in  C^{\log}(\rn)$ and Remark \ref{re-conv}(ii),
we deduce that, for all $x\in\rn$ and
$z\in B_{-1}^j(x)$,
\begin{eqnarray*}
(\delta_j^P)^{\frac{1}{q(z)}-\frac1{q(x)}}
&&\le \lf\{2^{\sigma[j-(j_P\vee0)]}\delta_j^P\r\}^{|\frac1{q(z)}-\frac1{q(x)}|}
\lf\{2^{\sigma[j-(j_P\vee0)]}\r\}^{|\frac1{q(z)}-\frac1{q(x)}|}\\
&&\ls 2^{2\sigma[j-(j_P\vee0)]C_{\log}(1/q)/\log(e+1/|x-z|)}\ls1.
\end{eqnarray*}
By this, $a\in(n,\fz)$ and \cite[p.\,59, (3.9)]{stein71}, we conclude that, for all $x\in P$,
\begin{eqnarray}\label{peetre4}
\frac{(\delta_j^P)^{-1/q(x)}}{\phi(P)}{\rm A}_{j,1}(x)
&&\ls\frac1{\phi(P)}\int_{B_{-1}^j(x)}
\frac{2^{jn}2^{js(z)}|\vz_j\ast f(z)|\chi_{4\sqrt nP}(z)}
{[\delta_j^P]^{1/q(z)}(1+2^j|x-z|)^{a}}\,dz\\
&&\ls\cm\lf(\frac{2^{js(\cdot)}|\vz_j\ast f|\chi_{4\sqrt nP}}
{[\delta_j^P]^{1/q(\cdot)}\phi(P)}\r)(x).\noz
\end{eqnarray}
On the other hand, by \eqref{delta}, we see that,
for all $x\in P$ and $z\in \rn$ with $i\in\zz_+$,
\begin{equation}\label{5.8x}
(\delta_j^P)^{[\frac1{q(z)}-\frac1{q(x)}]}\ls 2^{2\sigma[j-(j_P\vee0)]
(\frac1{q_-}-\frac1{q_+})}
\end{equation}
and
$1+2^j|x-z|\ge 1+2^j2^{-\frac{j-(j_P\vee0)}{2}+i}$.
Thus, by $
\sigma\in(0,\frac{a-n}{4(1/q_--1/q_+)}),
$
we conclude that, for all $x\in P$,
\begin{eqnarray*}
\frac{(\delta_j^P)^{-\frac 1{q(x)}}}{\phi(P)}{\rm A}_{j,2}(x)
&&\ls\sum_{i=0}^\fz\frac{2^{2\sigma[j-(j_P\vee0)](\frac1{q_-}-\frac1{q_+})}}
{\phi(P)2^{[\frac{j+(j_P\vee0)}{2}+i]a}}
\int_{B_i^j(x)}\frac{2^{jn}2^{js(z)}}{(\delta_j^P)^{1/q(z)}}|\vz_j\ast f(z)|
\chi_{4\sqrt nP}(z)\,dz\\
&&\ls2^{j[2\sigma (\frac1{q_-}-\frac1{q_+})+\frac n2-\frac{a}2]}
2^{(j_P\vee0)[-2\sigma (\frac1{q_-}-\frac1{q_+})-\frac{a}2+\frac n2]}\\
&&\hs\hs\times\sum_{i=0}^\fz2^{i(n-a)}
\cm\lf(\frac{\chi_{4\sqrt nP}}{[\delta_j^P]^{1/q(\cdot)}\phi(P)}2^{js(\cdot)}
|\vz_j\ast f|\r)(x)\\
&&\ls\cm\lf(\frac{\chi_{4\sqrt nP}}{[\delta_j^P]^{1/q(\cdot)}\phi(P)}2^{js(\cdot)}
|\vz_j\ast f|\r)(x),
\end{eqnarray*}
which, together with \eqref{peetre3} and \eqref{peetre4}, implies that,
for all $x\in P$,
\begin{equation*}
\frac{(\delta_j^P)^{-1/q(x)}}{\phi(P)}{\rm A}_{j}(x)
\ls\cm\lf(\frac{\chi_{4\sqrt nP}}{[\delta_j^P]^{1/q(\cdot)}\phi(P)}2^{js(\cdot)}
|\vz_j\ast f|\r)(x).
\end{equation*}
By this, Lemma \ref{l-hlo} and \eqref{phi-1},
 we further know that
\begin{eqnarray}\label{peetre5}
{\rm H}_{P,1}
&&\ls\lf\|\frac{\chi_{2^{n+2}P}}{[\delta_j^P]^{1/q(\cdot)}
\phi(2^{n+2}P)}2^{js(\cdot)}
|\vz_j\ast f|\r\|_{L^{p(\cdot)}(\rn)}\\
&&\ls2^{\vez/q_-}\lf\|\frac{\chi_{2^{n+2}P}2^{-\vez/q(\cdot)}}{[\delta_j^P]^{1/q(\cdot)}
\phi(2^{n+2}P)}2^{js(\cdot)}
|\vz_j\ast f|\r\|_{L^{p(\cdot)}(\rn)}
\ls1,\noz
\end{eqnarray}
where the last inequality comes from the definition of $\delta_j^P$.

We now estimate ${\rm H}_{P,2}$.
Notice that, when $x\in P$ and $z\in D_{k,P}$ with $k\in\nn\cap[2,\fz)$,
$1+2^j|x-z|\gs 2^k2^{j-j_P}$.
Then, by \eqref{5.8x} and \eqref{equi-x}, we see that, for all $x\in P$,
\begin{eqnarray*}
(\delta_j^P)^{-1/q(x)} {\rm A}_{j}^{k}(x)
&&\ls2^{2\sigma [j-(j_P\vee0)](\frac1{q_-}-\frac1{q_+})}
2^{-(k+j-j_P)a}2^{jn}2^{k\vez /q_-}\int_{D_{k,P}}\frac{2^{-k\vez /q(z)}}{[\delta_j^P]^{1/q(z)}}
2^{js(z)}|\vz_j\ast f(z)|\,dz\\
&&\ls 2^{-(j-j_P)[a-n-2\sigma (\frac1{q_-}-\frac1{q_+})]}
2^{-k(a-n-\frac{\vez}{q_-})}\cm\lf(\frac{\chi_{P_k^n}2^{-k\vez /q(\cdot)}}
{[\delta_j^P]^{1/q(\cdot)}}2^{js(\cdot)t}|\vz_j\ast f|\r)(x)\\
&&\ls2^{-k(a-n-\frac{\vez}{q_-})}
\cm\lf(\frac{\chi_{P_k^n}2^{-k\vez/q(\cdot)}}
{[\delta_j^P]^{1/q(\cdot)}}2^{js(\cdot)}|\vz_j\ast f|\r)(x),
\end{eqnarray*}
which, combined with Lemma \ref{l-hlo}, \eqref{phi-1},
 the definition of $\delta_j^P$ and $a\in(n+\log_2c_1+\vez/q_-,\fz)$, implies that
\begin{eqnarray}\label{peetre6}
{\rm H}_{P,2}
&&\ls\sum_{k=2}^\fz2^{-k(a-n-\frac{\vez}{q_-}-\log_2 c_1)}
\lf\|\frac{\chi_{P_k^n}2^{-k\vez/q(\cdot)}}
{\phi(P_k^n)[\delta_j^P]^{1/q(\cdot)}}2^{js(\cdot)}
|\vz_j\ast f|\r\|_{L^{p(\cdot)}(\rn)}\ls1.
\end{eqnarray}

Combining \eqref{5.4-x}, \eqref{peetre5} and \eqref{peetre6}, we conclude that
\eqref{peetre2} holds true and then complete the proof of Lemma \ref{l-equi}.
\end{proof}
\begin{proof}[Proof of Theorem \ref{t-equivalent}]
Let $f\in\cs'(\rn)$ and $\|f\|_{\bbeve}^\ast<\fz$. Then,
by the obvious fact that,
for all $j\in\zz_+$ and $x\in\rn$,
$$2^{js(x)}|\vz_j\ast f(x)|\le \vz_{j}^{\ast,a}(2^{js(\cdot)}f)(x),$$
we find that $\|f\|_{\bbeve}\le\|f\|_{\bbeve}^\ast$ and hence $f\in\bbeve$.
Thus, to complete the proof of this theorem, we only need to show that,
for all $f\in\bbeve$,
\begin{equation}\label{equivalent1}
\|f\|_{\bbeve}^\ast\ls\|f\|_{\bbeve}.
\end{equation}

Without loss of generality, to prove \eqref{equivalent1},
we may assume that $\|f\|_{\bbeve}=1$ and show that
$\|f\|_{\bbeve}^\ast\ls1$. By \eqref{equivalent1-xx},
we find that there exist $t\in(0,p_-)$ and $\vez\in(0,\fz)$ such that
\begin{equation}\label{equivalent1-z}
at\in\lf(n+\log_2 c_1+{\vez}/{q_-},\fz\r).
\end{equation}
Let $P\subset\rn$ be a given dyadic cube. Next we show that
\begin{equation}\label{peetre1-x}
\frac1{[\phi(P)]^t}\lf\|\lf\{[\vz_j^{\ast,a}(2^{js(\cdot)}f)]^t\r\}_{j\ge (j_P\vee0)}
\r\|_{\ell^{q(\cdot)/t}(L^{p(\cdot)/t}(P))}\ls1
\end{equation}
with implicit positive constant independent of $P$,
which, by Lemma \ref{l-modular} and Remark \ref{re-mixed}(i),
is equivalent to prove that
$\sum_{j=(j_P\vee0)}^\fz {\rm I}_{P,j}\ls1$, where
\begin{equation*}
{\rm I}_{P,j}:=\inf\lf\{\lz_j\in(0,\fz):\
\varrho_{\frac{p(\cdot)}t}\lf(\frac{c\chi_P[\vz_j^{\ast,a}(2^{js(\cdot)}f)]^t}
{[\phi(P)]^t\lz_j^{t/q(\cdot)}}\r)\le1\r\}
\end{equation*}
with $c$ being a positive constant sufficiently small.
Since
\begin{eqnarray*}
\lf[\vz_j^{\ast,a}(2^{js(\cdot)}f)(x)\r]^t=
\sup_{y\in\rn}\frac{2^{js(y)t}|\vz_j\ast f(y)|^t}{(1+2^j|x-y|)^{at}},
\end{eqnarray*}
it follows, from Lemma \ref{l-equi}, that, for all $j\in\zz_+\cap[(j_P\vee0),\fz)$,
\begin{eqnarray*}
{\rm I}_{P,j}
&&\le\sum_{k=1}^\fz
\inf\lf\{\eta_j\in(0,\fz):\
\varrho_{\frac{p(\cdot)}t}\lf(\frac{\chi_{P_k^n}2^{js(\cdot)t}|\vz_j\ast f|^t}
{2^{k\vez t/q(\cdot)}[\phi(P_k^n)]^t\eta_j^{t/q(\cdot)}}\r)\le1\r\}
+2^{-\wz\sigma[j-(j_P\vee0)]}\\
&&=\sum_{k=1}^\fz2^{-k\vez}
\inf\lf\{\eta_j\in(0,\fz):\
\varrho_{\frac{p(\cdot)}t}\lf(\frac{\chi_{P_k^n}
2^{js(\cdot)t}|\vz_j\ast f|^t}
{[\phi(P_k^n)]^t\eta_j^{t/q(\cdot)}}\r)\le1\r\}+2^{-\wz\sigma[j-(j_P\vee0)]}
=:\wz{\delta_j^P},
\end{eqnarray*}
where $P_k^n:=2^{k+1+n}P$ and $\wz\sigma\in(0,\frac{at-n}{4(1/q_--1/q_+)})$.
From this, we further deduce that
\begin{eqnarray*}
\sum_{j=(j_P\vee0)}^\fz {\rm I}_{P,j}
&&\ls\sum_{k=1}^\fz
\frac{2^{-k\vez}}{[\phi(P_k^n)]^t}
\lf\|\lf\{2^{js(\cdot)}|\vz_j\ast f|\r\}_{j\ge(j_P\vee0)}\r\|
_{\ell^{q(\cdot)}(L^{p(\cdot)}(P_k^n))}^t+1\\
&&\ls\sum_{k=1}^\fz2^{-k\vez}\|f\|_{\bbeve}^t+1\ls1,
\end{eqnarray*}
which implies that \eqref{peetre1-x} holds true.
This finishes the proof of Theorem \ref{t-equivalent}.
\end{proof}

As applications of Theorem \ref{t-equivalent}, we obtain more equivalent
quasi-norms of Besov-type spaces with variable smoothness and integrability.
To this end, for all $f\in\cs'(\rn)$, let
$$\lf\|f|\bbeve\r\|_1:=\sup_{P\in\cq}\frac1{\phi(P)}
\lf\|\lf\{2^{js(\cdot)}|\vz_j\ast f|\r\}_{j\ge0}
\r\|_{\ell^{q(\cdot)}(L^{p(\cdot)}(P))}$$
and
$$\lf\|f|\bbeve\r\|_2:=\sup_{Q\in\cq^\ast}\sup_{x\in Q}
[\phi(Q)]^{-1}|Q|^{-s(x)/n}
\|\chi_Q\|_{\vlp}|\vz_{j_Q}\ast f(x)|.$$

\begin{theorem}\label{t-equivalent-x}
Let $p,\ q,\ s,\ \phi$ be as in Definition \ref{def-b}.

{\rm (i)} Assume that $p_+\in(0,\fz)$ and $c_1\in(0,2^{n/p_+})$.
Then $f\in\bbeve$ if and only if $f\in\cs'(\rn)$ and
$\|f|\bbeve\|_1<\fz$; moreover, there exists a positive constant $C$,
independent of $f$, such that
\begin{equation}\label{sum-x}
\|f\|_{\bbeve}\le\lf\|f|\bbeve\r\|_1\le C\|f\|_{\bbeve}.
\end{equation}

{\rm (ii)} Assume that $p_-\in(0,\fz)$ and $c_1\in(0,2^{-n/p_-})$.
Then $f\in\bbeve$ if and only if $f\in\cs'(\rn)$ and
$\|f|\bbeve\|_2<\fz$; moreover, there exists a positive constant $C$,
independent of $f$, such that
\begin{equation}\label{sum-x1}
C^{-1}\|f\|_{\bbeve}\le\lf\|f|\bbeve\r\|_2\le C\|f\|_{\bbeve}.
\end{equation}
\end{theorem}
\begin{proof}
Let $P\subset\rn$ be a given dyadic cube and, for all $j\in\zz_+$ and $x\in\rn$,
$f_j(x):=2^{js(x)}|\vz_j\ast f(x)|$.

We first prove (i). Let $f\in\cs'(\rn)$ and
$\|f|\bbeve\|_1<\fz$. Then, by definitions, we easily find that
$\|f\|_{\bbeve}\le\|f|\bbeve\|_1$ and hence $f\in \bbeve$.
Conversely, let $f\in\bbeve$. Then $f\in\cs'(\rn)$.
To complete the proof of (i), it suffices to show the second inequality
of \eqref{sum-x}.

When $q_+\in(0,\fz)$, by Remark \ref{re-mixed}(iv), we have
\begin{eqnarray}\label{sum-y}
\qquad&&\frac1{\phi(P)}\lf\|\lf\{f_j\r\}_{j\ge0}
\r\|_{\ell^{q(\cdot)}(L^{p(\cdot)}(P))}\\
&&\hs\ls\frac1{\phi(P)}\lf\|\lf\{f_j\r\}_{j=0}^{(j_P\vee0)-1}
\r\|_{\ell^{q(\cdot)}(L^{p(\cdot)}(P))}
+\frac1{\phi(P)}\lf\|\lf\{f_j\r\}_{j\ge(j_P\vee0)}
\r\|_{\ell^{q(\cdot)}(L^{p(\cdot)}(P))}=:{\rm I}_{P,1}+{\rm I}_{P,2},\noz
\end{eqnarray}
where ${\rm I}_{P,1}=0$ if $j_P\le0$. Obviously,
${\rm I}_{P,2}\ls\|f\|_{\bbeve}$.
To estimate ${\rm I}_{P,1}$,
without loss of generality, we may assume that
 $\|f\|_{\bbeve}=1$ and show that ${\rm I}_{P,1}\ls1$
in the case that $j_P>0$.
Observe that, for all $j\in\zz_+$ with $j\le j_P-1$, there
exists a unique dyadic cube $P_j$ such
that $P\subset P_j$ and $\ell(P_j)=2^{-j}$. It follows that, for all $x\in P$,
\begin{equation}\label{5.15x}
f_j(x):=2^{js(x)}|\vz_j\ast f(x)|\ls\inf_{y\in P_j}\vz_j^{\ast,a}(2^{js(\cdot)}f)(y)
\end{equation}
and, moreover,
\begin{eqnarray}\label{sum-z}
\lf\|[\phi(P)]^{-1}\chi_Pf_j\r\|_{\vlp}
&&\ls\lf\|[\phi(P)]^{-1}\chi_P\inf_{y\in P_j}\vz_j^{\ast,a}(2^{js(\cdot)}f)(y)
\r\|_{\vlp}\\
&&\ls\frac{\|\chi_P\|_{\vlp}}{\|\chi_{P_j}\|_{\vlp}}[\phi(P)]^{-1}
\|\vz_j^{\ast,a}(2^{js(\cdot)}f)\|_{L^{p(\cdot)}(P_j)}\noz\\
&&\ls\|f\|_{\bbeve}\frac{\|\chi_P\|_{\vlp}}{\|\chi_{P_j}\|_{\vlp}}
\frac{\phi(P_j)}{\phi(P)},\noz
\end{eqnarray}
where we used Theorem \ref{t-equivalent} in the last inequality.
On the other hand, by \cite[Lemma 2.6]{zyl14}, we find that
$$\|\chi_{P_j}\|_{\vlp}\gs 2^{-j\frac n{p_+}}2^{j_P\frac n{p_+}}
\|\chi_P\|_{\vlp}$$
and, by \eqref{phi-1} and \eqref{phi-2}, we see that
$\phi(P)\gs 2^{j\log_2c_1}2^{-j_P\log_2c_1}\phi(c_{P_j},2^{-j})$. Thus, by
\eqref{sum-z}, we further conclude that
\begin{eqnarray}\label{sum-w}
\lf\|[\phi(P)]^{-1}\chi_Pf_j\r\|_{\vlp}
\ls\|f\|_{\bbeve}2^{(j-j_P)(\frac n{p_+}-\log_2c_1)},
\end{eqnarray}
which, together with (i) and (ii) of Lemma \ref{l-inf-norm}, implies that
\begin{eqnarray*}
&&\inf\lf\{\lz_j:\ \varrho_{p(\cdot)}
\lf([\phi(P)]^{-1}\lz_j^{-1/q(\cdot)}\chi_Pf_j\r)\le1\r\}\\
&&\hs\ls\lf\|[\phi(P)]^{-1}\chi_Pf_j\r\|_{\vlp}^{q_-}
+\lf\|[\phi(P)]^{-1}\chi_Pf_j\r\|_{\vlp}^{q_+}\\
&&\hs\ls2^{(j-j_P)(\frac n{p_+}-\log_2c_1)q_-}+2^{(j-j_P)(\frac n{p_+}-\log_2c_1)q_+}.
\end{eqnarray*}
From this and $c_1\in(0,2^{n/p_+})$, we deduce that there exists a
positive constant $C_0$ such that
$$\sum_{j=0}^{j_P-1}\inf\lf\{\lz_j:\ \varrho_{p(\cdot)}
\lf(\frac{\chi_Pf_j}{C_0\phi(P)\lz_j^{1/q(\cdot)}}\r)\le1\r\}\le1,$$
namely,
$\varrho_{\ell^{q(\cdot)}(L^{p(\cdot)})}
(\lf\{[C_0\phi(P)]^{-1}\chi_Pf_j\r\}_{j=0}^{j_P-1})\le1$,
which, combined with Remark \ref{re-mixed}(i),
implies that ${\rm I}_{P,1}\ls1$. Therefore, by \eqref{sum-y}, we find that
$$\lf\|f|\bbeve\r\|_1\ls\sup_{P\in\cq}({\rm I}_{P,1}+{\rm I}_{P,2})
\ls\|f\|_{\bbeve},$$
which completes the proof of the second inequality of \eqref{sum-x}
in the case $q_+\in(0,\fz)$.

We now consider the case that $q_+=\fz$. In this case,
$q\equiv\fz$ by Remark \ref{re-conv}(iii).
From \eqref{sum-w} and $c_1\in(0,2^{n/p_+})$, we deduce that, for $j_P\in\nn$,
\begin{eqnarray*}
\sup_{j\in\zz_+, j\le j_P}[\phi(P)]^{-1}\lf\|f_j\r\|_{L^{p(\cdot)}(P)}
\ls\|f\|_{\bbeve}\sup_{j\in\zz_+, j\le j_P}
2^{(j-j_P)(\frac n{p_+}-\log_2c_1)}
\ls\|f\|_{\bbeve}.\noz
\end{eqnarray*}
By this, we know that
\begin{eqnarray*}
\lf\|f|B_{p(\cdot),\fz}^{s(\cdot),\phi}(\rn)\r\|_1
&&\ls\sup_{P\in\cq}\lf\{\sup_{j\in\zz_+,j\le(j_P\vee0)}
\frac{\|f_j\|_{L^{p(\cdot)}(P)}}{\phi(P)}
+\sup_{j\in\zz_+,j\ge (j_P\vee0)}
\frac{\|f_j\|_{L^{p(\cdot)}(P)}}{\phi(P)}\r\}\\
&&\ls\|f\|_{B_{p(\cdot),\fz}^{s(\cdot),\phi}(\rn)},
\end{eqnarray*}
which completes the proof of the second inequality of \eqref{sum-x}
in the case that $q_+=\fz$ and hence (i) of Theorem \ref{t-equivalent-x}.

Next, we show (ii). Let $f\in\bbeve$. Then $f\in\cs'(\rn)$.
On the other hand, for all $Q\in\cq^\ast$ and $x\in Q$,
by Theorem \ref{t-equivalent} and \eqref{5.15x}, we easily see that
\begin{eqnarray*}
\frac{\|\chi_Q\|_{\vlp}}{\phi(Q)}f_{j_Q}(x)
&&\ls\frac{\|\chi_Q\|_{\vlp}}{\phi(Q)}
\inf_{y\in Q}\vz_{j_Q}^{\ast,a}(2^{j_Qs(\cdot)}f)(y)\\
&&\ls[\phi(Q)]^{-1}\|\vz_{j_Q}^{\ast,a}(2^{j_Qs(\cdot)}f)\|_{L^{p(\cdot)}(Q)}
\ls\|f\|_{\bbeve}.\noz
\end{eqnarray*}
This implies that $\|f|\bbeve\|_2\ls\|f\|_{\bbeve}<\fz$.

Conversely, let $f\in\cs'(\rn)$ and $\|f|\bbeve\|_2<\fz$.
We need to show the first inequality of \eqref{sum-x1}. To this end,
for all $j\ge (j_P\vee0)$ and $x\in\rn$,
$\cq_{P,j}^\ast:=\{Q\in\cq^\ast:\ Q\subset P,\ \ell(Q)=2^{-j}\}$
and, for all $Q\in \cq_{P,j}^\ast$, let
 $$g(Q,P)(x):=[\phi(P)]^{-1}\phi(Q)\|\chi_Q\|_{\vlp}^{-1}
\chi_{ Q}(x).$$
 When $q_+\in(0,\fz)$, by \cite[Lemma 2.6]{zyl14},
\eqref{phi-1} and \eqref{phi-2}, we find that
\begin{eqnarray*}
&&\lf\|\sum_{Q\in\cq_{P,j}^\ast}g(Q,P)\r\|_{\vlp}
\ls2^{(j-j_P)(\log_2c_1+\frac n{p_-})},
\end{eqnarray*}
which, combined with (i) and (ii) of Lemma \ref{l-inf-norm}, implies that
\begin{eqnarray*}
&&\varrho_{\ell^{q(\cdot)}(L^{p(\cdot)})}
\lf(\lf\{\sum_{Q\in\cq_{P,j}^\ast}g(Q,P)\r\}\r)\\
&&\hs=\sum_{j=(j_P\vee0)}^\fz\inf\lf\{\lz_j\in(0,\fz):\ \varrho_{p(\cdot)}
\lf(\sum_{Q\in\cq_{P,j}^\ast}\frac{g(Q,P)}
{\lz_j^{1/q(\cdot)}}\r)\le1\r\}\\
&&\hs\le\sum_{j=(j_P\vee0)}^\fz\lf[\lf\|\sum_{Q\in\cq_{P,j}^\ast}g(Q,P)\r\|_{\vlp}^{q_-}+
\lf\|\sum_{Q\in\cq_{P,j}^\ast}g(Q,P)\r\|_{\vlp}^{q_+}\r]\\
&&\hs\ls\sum_{j=(j_P\vee0)}^\fz\lf[2^{(j-j_P)(\log_2c_1+\frac n{p_-})q_-}
+2^{(j-j_P)(\log_2c_1+\frac n{p_-})q_+}\r]\ls1.
\end{eqnarray*}
By this and Remark \ref{re-mixed}(i), we conclude that
\begin{eqnarray*}
\lf\|\lf\{\sum_{Q\in\cq_{P,j}^\ast}g(Q,P)\r\}_{j\ge(j_P\vee0)}\r\|
_{\ell^{q(\cdot)}(L^{p(\cdot)}(P))}\ls1.
\end{eqnarray*}
Therefore,
\begin{eqnarray*}
&&\lf\|\lf\{[\phi(P)]^{-1}\chi_Pf_j\r\}_{j\ge(j_P\vee0)}
\r\|_{\ell^{q(\cdot)}(L^{p(\cdot)}(P))}\\
&&\hs\ls\lf\|\lf\{\frac{1}{\phi(P)}\sum_{Q\in\cq_{P,j}^\ast}
\chi_Qf_j\r\}_{j\ge(j_P\vee0)}
\r\|_{\ell^{q(\cdot)}(L^{p(\cdot)}(P))}\\
&&\hs\ls\lf\|f|\bbeve\r\|_2\lf\|\lf\{\sum_{Q\in\cq_{P,j}^\ast}g(Q,P)
\r\}_{j\ge(j_P\vee0)}\r\|
_{\ell^{q(\cdot)}(L^{p(\cdot)}(P))}\ls\lf\|f|\bbeve\r\|_2,
\end{eqnarray*}
which implies that the first inequality of \eqref{sum-x1} holds true
in the case that $q_+\in(0,\fz)$.
The proof of the case that $q_+=\fz$ is similar
and more simple,
the details being omitted. This finishes the proof of (ii) and hence Theorem \ref{t-equivalent-x}.
\end{proof}
As another application of Theorem \ref{t-equivalent}, we obtain the following
conclusion.
\begin{proposition}\label{p-sbs}
Let $p,\ q,\ s$ and $\phi$ be as in Definition \ref{def-b}. Then
\begin{equation}\label{embed1}
\cs(\rn)\hookrightarrow\bbeve\hookrightarrow\cs'(\rn).
\end{equation}
\end{proposition}
\begin{proof}
By Proposition \ref{p-embed1}, we see that
$B_{p(\cdot),q_-}^{s(\cdot),\phi}(\rn)
\hookrightarrow\bbeve\hookrightarrow B_{p(\cdot),\fz}^{s(\cdot),\phi}(\rn).$
Thus, to prove \eqref{embed1}, it suffices to show that
\begin{equation}\label{embed2}
\cs(\rn)\hookrightarrow B_{p(\cdot),q_-}^{s(\cdot),\phi}(\rn)\quad
{\rm and}\quad
B_{p(\cdot),\fz}^{s(\cdot),\phi}(\rn)\hookrightarrow\cs'(\rn).
\end{equation}
The first embedding of \eqref{embed2} can be obtained by an argument similar to
that used in the proof of \cite[Proposition 3.20]{yyz14}, the details being omitted.
Next we give the proof of the second one. To this end,
we only need to show that there exists an $M\in\nn$ such that, for
all $f\in B_{p(\cdot),\fz}^{s(\cdot),\phi}(\rn)$ and $h\in\cs(\rn)$,
$|\langle f,h\rangle|
\ls\|h\|_{\cs_{M+1}(\rn)}\|f\|_{B_{p(\cdot),\fz}^{s(\cdot),\phi}(\rn)}.$

Let $\vz$, $\psi$, $\Phi$ and $\Psi$ be as in \eqref{cz1}. Then, by the Calder\'on
reproducing formula in \cite[Lemma 2.3]{ysiy}, together with
\cite[Lemma 2.4]{ysiy}, we find that
\begin{eqnarray}\label{embed3}
|\langle f,h\rangle|
&&\le\int_\rn|\Phi\ast f(x)||\Psi\ast h(x)|\,dx
+\sum_{j=1}^\fz\int_\rn|\vz_j\ast f(x)||\psi_j\ast h(x)|\,dx\\
&&\ls\|h\|_{\cs_{M+1}(\rn)}\sum_{j=0}^\fz2^{-jM}\sum_{k\in\zz^n}
\int_{Q_{0k}}|\vz_j\ast f(x)|(1+|x|)^{-(n+M)}\,dx,\noz
\end{eqnarray}
where we used $\vz_0$ to replace $\Phi$. Notice that, for any $j\in\zz_+$,
$k\in\zz^n$, $a\in(0,\fz)$ and $y\in Q_{jk}$,
\begin{eqnarray*}
\int_{Q_{0k}}|\vz_j\ast f(x)|\,dx
&&\ls\vz_j^{\ast,a}(2^{js(\cdot)}f)(y)\int_{Q_{0k}}2^{-js(x)}(1+2^j|x|+2^j|y|)^a\,dx\\
&&\ls2^{-js_-}\vz_j^{\ast,a}(2^{js(\cdot)}f)(y)2^{ja}(1+|k|)^a.
\end{eqnarray*}
It follows that
$$\int_{Q_{0k}}|\vz_j\ast f(x)|\,dx
\ls2^{j(a-s_-)}(1+|k|)^a\inf_{y\in Q_{jk}}\vz_j^{\ast,a}(2^{js(\cdot)}f)(y),$$
which, combined with \eqref{embed3}, Theorem \ref{t-equivalent},
Lemmas \ref{l-esti-cube} and \ref{l-cube-1}, implies that
\begin{eqnarray*}
|\langle f,h\rangle|
&&\ls\|h\|_{\cs_{M+1}(\rn)}\sum_{j=0}^\fz2^{-j(M+s_--a)}\sum_{k\in\zz^n}
(1+|k|)^{a-n-M}\frac{\|\vz_j^{\ast,a}(2^{js(\cdot)}f)\|_{L^{p(\cdot)}(Q_{jk})}}
{\|\chi_{Q_{jk}}\|_{\vlp}}\\
&&\ls\|h\|_{\cs_{M+1}(\rn)}\|f\|_{B_{p(\cdot),\fz}^{s(\cdot),\phi}(\rn)}
\sum_{j=0}^\fz\sum_{k\in\zz^n}\frac{2^{-j(M+s_--a)}}{(1+|k|)^{M+n-a}}
\frac{\phi(Q_{jk})}{\|\chi_{Q_{jk}}\|_{\vlp}}\\
&&\ls\|h\|_{\cs_{M+1}(\rn)}\|f\|_{B_{p(\cdot),\fz}^{s(\cdot),\phi}(\rn)},
\end{eqnarray*}
where $a$ is as in Theorem \ref{t-equivalent} and $M$ is large enough.
This finishes the proof of Proposition \ref{p-sbs}.
\end{proof}
\begin{remark}
(i) When $\phi\equiv1$, Proposition \ref{p-sbs} was proved in
 \cite[Theorem 6.10]{ah10}.

(ii) When $p,\ q,\ s$ and $\phi$ are as in Remark \ref{re-conv}(ii), Proposition
\ref{p-sbs} was obtained in \cite[Proposition 2.3]{ysiy}.
\end{remark}
Next we establish the atomic characterization of $\bbeve$.

\begin{definition}
Let $k\in\zz_+$ and $L\in\zz$. A measurable function $a_Q$ on $\rn$
is called a ($K,L$)-\emph{smooth atom} supported near $Q:=Q_{jk}\in\cq$ if it
satisfies the following conditions:

(A1) (\emph{support condition}) supp\,$a_Q\subset 3Q$;

(A2) (\emph{vanishing moment}) when $j\in\nn$, $\int_\rn x^\gamma a_Q(x)\,dx=0$
for all $\gamma\in\zz_+^n$ with $|\gamma|<L$;

(A3) (\emph{smoothness condition}) for all multi-indices $\alpha\in\zz_+^n$
with $|\alpha|\le K$,
$|D^\alpha a_Q(x)|\le 2^{(|\alpha|+n/2)j}$.

A collection $\{a_Q\}_{Q\in\cq^\ast}$ is called \emph{a family of ($K,\,L$) smoothness
atoms}, if each $a_Q$ is a ($K,\,L$)-smooth atom supported near $Q$.
\end{definition}

We point out that, if $L\le0$, then the vanishing moment condition (A2) is avoid.

\begin{theorem}\label{t-atomd}
Let $p,\ q,\ s$ and $\phi$ be as in Definition \ref{def-b}.

{\rm(i)} Let $K\in( s_++\log_2 c_1,\fz)$ and
\begin{equation}\label{atomd1}
 L\in\lf( n/{\min\{1,p_-\}}-n-s_-,\fz\r).
\end{equation}
 Suppose that $\{a_Q\}_{Q\in\cq^\ast}$
is a family of $(K,\,L)$-smooth atoms and
$t:=\{t_Q\}_{Q\in\cq^\ast}\in\beve.$
Then $f:=\sum_{Q\in\cq^\ast}t_Qa_Q$ converges in $\cs'(\rn)$ and
$\|f\|_{\bbeve}\le C\|t\|_{\beve}$
with $C$ being a positive constant independent of $t$.

{\rm (ii)} Conversely, if $f\in \bbeve$, then, for any given $K,\ L\in\zz_+$,
there exist sequences $t:=\{t_Q\}_{Q\in\cq^\ast}\subset \cc$ and
$\{a_Q\}_{Q\in\cq^\ast}$ of $(K,\,L)$-smooth atoms such that
$f=\sum_{Q\in\cq^\ast}t_Qa_Q$ in $\cs'(\rn)$ and
$\|t\|_{\beve}\le C\|f\|_{\bbeve}$
with $C$ being a positive constant independent of $f$.
\end{theorem}
\begin{remark}
(i) Even when $\phi\equiv1$, conclusions of Theorem \ref{t-atomd} cover
\cite[Theorem 3]{dd12}, in which the case that $q_+=\fz$ is not included.

(ii) In the case that $p,\ q,\ s$ and $\phi$ are as in Remark \ref{r-defi}(ii),
Theorem \ref{t-atomd} was proved in \cite[Theorem 3]{dd12} and partly obtained
in \cite[Theorem 3.3]{ysiy}.

(iii) A sequence $\{a_Q\}_{Q\in\cq^\ast}$ is called
\emph{a family of smooth atoms of} $\bbeve$ if, for each $Q\in\cq^\ast$,
$a_Q$ is a $(K,\,L)$-smooth atom with $K$ and $L$ as in Theorem \ref{t-atomd}(i).
\end{remark}

To prove Theorem \ref{t-atomd}, we need the following two technical lemmas.
The first one was proved in \cite[Lemma 3.3]{fj85} and the second one is a Hardy-type
inequality which is just \cite[Lemma 3.11]{d13}.

\begin{lemma}\label{l-atom}
Let $\{\vz_j\}_{j\in\zz_+}$ be as in Definition \ref{def-b}
and $a_{Q_{vk}}$ with $v\in\zz_+$ and $k\in\zz^n$ be a $(K,\,L)$-smooth atom.
Then, for all $M\in (0,\fz)$,
 there exist positive constants $C_1$ and $C_2$ such that, for all $x\in\rn$,
when $j\le v$,
$$|\vz_j\ast a_{Q_{vk}}(x)|\le C_12^{vn/2}2^{-(v-j)(L+n)}(1+2^j|x-x_{Q_{vk}}|)^{-M}$$
and, when $j>v$,
$$|\vz_j\ast a_{Q_{vk}}(x)|\le C_22^{vn/2}2^{-(j-v)K}(1+2^v|x-x_{Q_{vk}}|)^{-M}.$$
\end{lemma}
\begin{lemma}\label{l-hardy}
Let $a\in(0,1)$, $J\in\zz$, $q\in(0,\fz]$ and $\{\vez_k\}_{k\in\zz_+}$ be a sequence
of positive real numbers. For all $k\in[J\vee0,\fz)$, let
$\delta_k:=\sum_{j=(J\vee0)}^ka^{k-j}\vez_j$ and
$\eta_k:=\sum_{j=k}^\fz a^{j-k}\vez_j.$
Then there exists a positive constant $C$, depending only on $a$ and $q$, such that
$$\lf(\sum_{k=(J\vee0)}^\fz\delta_k^q\r)^{1/q}
+\lf(\sum_{k=(J\vee0)}^\fz\eta_k^q\r)^{1/q}\le C
\lf(\sum_{k=(J\vee0)}^\fz\vez_k^q\r)^{1/q}.$$
\end{lemma}
\begin{proof}[Proof of Theorem \ref{t-atomd}]
The proof of (ii) is similar to that of \cite[Theorem 3.3]{ysiy} (see also
\cite[Theorem 4.1]{fj90}). Indeed, by repeating the argument that
used in the proof of \cite[Theorem 3.3]{ysiy}, with \cite[Lemma 2.8]{ysiy}
therein replaced by Lemma \ref{l-estimate1}, we can prove (ii), the
details being omitted.

Next we prove (i) by two steps.
First, we show that $f:=\sum_{Q\in\cq^\ast}t_Qa_Q$ converges in $\cs'(\rn)$.
To this end, it suffices to prove that
\begin{equation}\label{converge}
\lim_{N\to\fz,\Lambda\to\fz}\sum_{j=0}^N
\sum_{k\in\zz^n,|k|\le\Lambda}t_{Q_{jk}}a_{Q_{jk}}
\end{equation}
exists in $\cs'(\rn)$.
By \eqref{atomd1}, we find that there exists $r\in(0,\min\{1,p_-\})$
such that $s_-+\frac n{p_-}(r-1)>-L$.
Let, for all $x\in\rn$, $\wz p(x):=p(x)/r$ and $\wz s$ be a measurable function on $\rn$
such that $s(x)-\frac n{p(x)}=\wz s(x)-\frac n{\wz p(x)}$.
Then
$
\wz s_-\ge s_-+\frac n{p_-}(r-1)>-L.
$
Therefore, by Proposition \ref{p-se-embed} and
an argument similar to that used in the proof of
\cite[Theorem 3.8]{yyz14},
we conclude that there exist
$\delta_0\in(\log_2c_1,\fz)$, $a\in(n,\fz)$, $c_0\in\nn$ and
$R\in(0,\fz)$ being large enough such that, for all $h\in\cs(\rn)$ and $j\in\zz_+$,
\begin{eqnarray*}
&&\lf|\int_\rn\sum_{k\in\zz^n,|k|\le\Lambda}t_{Q_{jk}}a_{Q_{jk}}(y)h(y)\,dy\r|\\
&&\hs\ls2^{-j(L+\wz s_-)}\sum_{v=0}^\fz2^{-v\delta_0}\sum_{i=0}^\fz2^{-i(R-L-a)}
\lf\|\sum_{k\in\zz^n}2^{j\wz s(\cdot)}|t_{Q_{jk}}|\wz\chi_{Q_{jk}}
\r\|_{L^{\wz p(\cdot)}(Q(0,2^{i+v+c_0}))}\\
&&\hs\ls2^{-j(L+\wz s_-)}\sum_{v=0}^\fz2^{-v\delta_0}\sum_{i=0}^\fz2^{-i(R-L-a)}
\phi(Q(0,2^{i+v+c_0}))\|t\|_{b_{\wz p(\cdot),\fz}^{\wz s(\cdot),\phi}(\rn)}\\
&&\hs\ls2^{-j(L+\wz s_-)}\sum_{v=0}^\fz2^{-v(\delta_0-\log_2 c_1)}
\sum_{i=0}^\fz2^{-i(R-L-a-\log_2 c_1)}
\|t\|_{b_{p(\cdot),\fz}^{s(\cdot),\phi}(\rn)}
\ls2^{-j(L+\wz s_-)}\|t\|_{b_{p(\cdot),q(\cdot)}^{s(\cdot),\phi}(\rn)}.
\end{eqnarray*}
By this and the fact that $L>-\wz s_-$, we find that
the limit of \eqref{converge} exists in $\cs'(\rn)$.

Second, we prove that
\begin{equation}\label{atomd2}
\|f\|_{\bbeve}\ls\|t\|_{\beve}.
\end{equation}
Without loss of generality, we may assume that
$\|t\|_{\beve}=1$ and show $\|f\|_{\bbeve}\ls1$.

\textbf{Case I)} $q_+\in(0,\fz)$.
By Remark \ref{r-lattice}, we see that,
for all $R\in \cd_0(\rn)$,
\begin{eqnarray}\label{atomd4}
\frac1{\phi(R)}\lf\|\lf\{\sum_{Q\in\cq^\ast,\,\ell(Q)=2^{-v}}|Q|^{-\frac{s(\cdot)}n}
|t_Q|\wz\chi_Q\r\}_{v\ge(j_R\vee0)}\r\|_{\ell^{q(\cdot)}(L^{p(\cdot)}(R))}\ls1
\end{eqnarray}
with implicit positive constant independent of $R$.
Since $f=\sum_{Q\in\cq^\ast}t_Qa_Q$ in $\cs'(\rn)$, it follows that,
for all $P\in\cq$,
\begin{eqnarray*}
\vz_j\ast f=\lf\{\sum_{v=0}^{(j_P\vee0)-1}+\sum_{v=(j_P\vee0)}^j+
\sum_{v=j+1}^\fz\r\}\sum_{\ell(Q)=2^{-v}}
t_Q\vz_j\ast a_Q=:{\rm S}_{j,1}+{\rm S}_{j,2}+{\rm S}_{j,3},
\end{eqnarray*}
where $\sum_{v=0}^{(j_P\vee0)-1}\cdots=0$ if $j_P\le0$.
Thus, by Remark \ref{re-mixed}(iv), we find that
\begin{eqnarray*}
{\rm I}_P:=&&\frac1{\phi(P)}\lf\|\lf\{2^{js(\cdot)}\vz_j\ast f\r\}_{j\ge(j_P\vee 0)}
\r\|_{\ell^{q(\cdot)}(L^{p(\cdot)}(P))}\\
\ls&&\sum_{i=1}^3\frac1{\phi(P)}\lf\|\lf\{2^{js(\cdot)}{\rm S}_{j,i}\r\}_{j\ge(j_P\vee0)}
\r\|_{\ell^{q(\cdot)}(L^{p(\cdot)}(P))}
=:{\rm I}_{P,1}+{\rm I}_{P,2}+{\rm I}_{P,3}.
\end{eqnarray*}

In what follows, let $r\in(0,\min\{1,p_-\})$ satisfy $L+n-n/r+s_->0$.

We show that ${\rm I}_{P,1}\ls1$.
To this end, it suffices to
consider the case that $j_P>0$ and prove that
there exists a positive constant $C$ such that
\begin{eqnarray*}
\varrho_{\ell^{q(\cdot)}(L^{p(\cdot)})}
\lf(\lf\{\frac{\chi_P2^{js(\cdot)}}{C\phi(P)}\sum_{v=0}^{j_P-1}
\sum_{\ell(Q)=2^{-v}}|t_Q||\vz_j\ast a_Q|\r\}_{j\ge j_P}\r)\le1,
\end{eqnarray*}
which, by Remark \ref{re-mixed}(ii), is equivalent to show that
\begin{eqnarray*}
{\rm J}_{P,1}:=\sum_{j=j_P}^\fz\lf\|\lf[\frac{\chi_P}{\phi(P)}2^{js(\cdot)}
\sum_{v=0}^{j_P-1}\sum_{\ell(Q)=2^{-v}}
|t_Q||\vz_j\ast a_Q|\r]^{q(\cdot)}
\r\|_{L^{\frac{p(\cdot)}{q(\cdot)}}(\rn)}\ls1.
\end{eqnarray*}
By Lemma \ref{l-atom} and \eqref{simple-ineq}, we find that
\begin{eqnarray}\label{atomd3}
{\rm J}_{P,1}
\ls&&\sum_{j=j_P}^\fz
\lf\|\lf[\frac{\chi_P}{\{\phi(P)\}^r}2^{js(\cdot)r}\sum_{v=0}^{j_P-1}
\sum_{k\in\zz^n}|t_{Q_{vk}}|^r|Q_{vk}|^{-r/2}\r.\r.\\
&&\hs\hs\times\lf.2^{(v-j)Kr}(1+2^v|\cdot-x_{Q_{vk}}|)^{-Mr}\Bigg]^{q(\cdot)}\r\|
_{L^{\frac{p(\cdot)}{rq(\cdot)}}(\rn)}^{\frac1r},\noz
\end{eqnarray}
where $M\in(0,\fz)$ is large enough.
On the other hand, by the proof of \cite[Theorem 3.8(i)]{yyz14},
we know that, for all $v,\ j\in\zz_+$ with $v\le j$ and $x\in P$,
\begin{eqnarray*}
&&2^{js(x)r}\sum_{k\in\zz^n}|t_{Q_{vk}}|^r|Q_{vk}|^{-\frac r2}2^{(v-j)Kr}
(1+2^v|x-x_{Q_{vk}}|)^{-Mr}\\
&&\hs\ls2^{(v-j)(K-s_+)r}\sum_{i=0}^\fz2^{-i(M-a-\frac{C_{\log}(s)}{r})r}
\eta_{v,ar}\ast\lf(\lf[\sum_{k\in \zz^n}|t_{Q_{vk}}|2^{vs(\cdot)}
\wz\chi_{Q_{vk}}\chi_{Q(c_P,2^{i-v+c_0})}\r]^r\r)(x),
\end{eqnarray*}
where $a\in(n/r,\fz)$, $c_P$ is the center of $P$ and
$c_0\in\nn$ independent of $x,\,P,\,i,\,v$ and $k$.
From this, \eqref{atomd3}, and (i) and (ii) of Lemma \ref{l-inf-norm}, we deduce that
${\rm J}_{P,1}\ls\sum_{j=j_P}^\fz[({\rm J}_{P,1}^j)^{q_-}
+({\rm J}_{P,1}^j)^{q_+}],$
where
\begin{eqnarray*}
{\rm J}_{P,1}^j
&&:=\lf\|\frac{\chi_P}{[\phi(P)]^r}
\sum_{v=0}^{j_P-1}2^{(v-j)(K-s_+)r}\sum_{i=0}^\fz2^{-i(M-a-C_{\log}(s)/r)r}\r.\\
&&\hs\hs\times\lf.\eta_{v,ar}\ast\lf(\sum_{k\in\zz^n}
|t_{Q_{vk}}|^r2^{vs(\cdot)r}
|Q_{vk}|^{-\frac r2}\chi_{Q_{vk}}\chi_{Q(c_P,2^{i-v+c_0})}\r)
\r\|_{L^{\frac{p(\cdot)}r}(\rn)}^{\frac{1}r}.
\end{eqnarray*}
By Remark
\ref{re-vlp}(i), \eqref{atomd4}, Remarks \ref{re-conv}(i) and \ref{r-lattice},
we find that
\begin{eqnarray*}
{\rm J}_{P,1}^j
&&\ls\Bigg\{\frac1{[\phi(P)]^r}\sum_{v=0}^{j_P-1}2^{(v-j)(K-s_+)r}
\sum_{i=0}^\fz2^{-i(M-a-C_{\log}(s)/r)r}\\
&&\hs\hs\times\lf.\lf\|\sum_{k\in\zz^n}
|t_{Q_{vk}}|^r2^{vs(\cdot)r}
|Q_{vk}|^{-\frac r2}\chi_{Q_{vk}}\chi_{Q(c_P,2^{i-v+c_0})}
\r\|_{L^{\frac{p(\cdot)}r}(\rn)}\r\}^{\frac{1}r}\\
&&\ls\lf\{\sum_{v=0}^{j_P-1}2^{(v-j)(K-s_+)r}\sum_{i=0}^\fz2^{-i(M-a-C_{\log}(s)/r)r}
\frac{[\phi(Q(c_P,2^{i-v+c_0}))]^r}{[\phi(P)]^r}\r\}^{\frac{1}r}.
\end{eqnarray*}
By this, \eqref{phi-1} and the fact that
$K\in(s_++\log_2c_1,\fz)$, we know that
\begin{eqnarray*}
\sum_{j=j_P}^\fz({\rm J}_{P,1}^j)^{q_-}
&&\ls2^{j_Pq_-\log_2c_1}\sum_{j=j_P}^\fz\lf\{2^{j(s_+-K)}
\sum_{v=0}^{j_P-1}2^{(K-s_+-\log_2c_1)vr}\r\}^{\frac{q_-}r}\ls1
\end{eqnarray*}
and $\sum_{j=j_P}^\fz({\rm J}_{P,1}^j)^{q_+}\ls1$,
where $M$ is chosen large enough such that $M>a+C_{\log}(s)/r+\log_2c_1$,
which implies ${\rm I}_{P,1}\ls1$. This is a desired estimate.

We now estimate ${\rm I}_{P,2}$.
By Lemma \ref{l-atom}, we see that, for all $M\in(0,\fz)$ and $x\in\rn$,
$$\sum_{k\in\zz^n}2^{js(x)}|t_{Q_{vk}}||\vz_j\ast a_{Q_{vk}}(x)|
\ls2^{(v-j)(K-s_+)}\sum_{k\in\zz^n}|Q_{vk}|^{-\frac{s(x)}n+\frac12}
|t_{Q_{vk}}|(1+2^v|x-x_{Q_{vk}}|)^{-M}$$
and hence, for all $r\in(0,\min\{1/q_+,p_-/q_+\})$,
\begin{eqnarray}\label{atom2-z}
&&\lf\|\lf[
\frac{\chi_P}{\phi(P)}2^{js(\cdot)}{\rm S}_{j,2}\r]^{q(\cdot)}
\r\|_{L^{\frac{p(\cdot)}{q(\cdot)}}(\rn)}\\
&&\hs\ls\lf\{\sum_{v=(j_P\vee0)}^j2^{(v-j)q_-(K-s_+)}
\lf\|\lf[\frac{\chi_P}{\phi(P)}
\sum_{\ell(Q)=2^{-v}}\frac{|Q|^{-\frac{s(\cdot)}n+\frac12}
|t_{Q}|}{(1+2^v|\cdot-x_{Q}|)^{M}}\r]^{rq(\cdot)}\r\|
_{{L^{\frac{p(\cdot)}{rq(\cdot)}}(\rn)}}\r\}^{\frac1r}.\noz
\end{eqnarray}

We claim that there exists a positive constant $c$ such that
\begin{eqnarray}\label{atom2-x}
&&\lf\|\lf[\frac{c\chi_P}{\phi(P)}
\sum_{\ell(Q)=2^{-v}}\frac{|Q|^{-\frac{s(\cdot)}n+\frac12}
|t_{Q}|}{(1+2^v|\cdot-x_{Q}|)^{M}}\r]^{rq(\cdot)}\r\|
_{L^{\frac{p(\cdot)}{rq(\cdot)}}(\rn)}\\
&&\hs\le\sum_{i=0}^\fz2^{-i\tau}
\lf\|\lf[\frac{\chi_{Q_i^0}}{\phi({Q_i^0})}
\sum_{\ell(Q)=2^{-v}}|Q|^{-\frac{s(\cdot)}n}|t_Q|\wz\chi_Q
\r]^{rq(\cdot)}\r\|_{L^{\frac{p(\cdot)}{rq(\cdot)}}(\rn)}
+2^{-v}=:\delta_v^P,\noz
\end{eqnarray}
where $Q_i^0:=Q(c_P,2^{i-j_P+c_0})$ with some $c_0\in\nn$
and $\tau\in(0,\fz)$.

From the above claim, \eqref{atom2-z}, Lemma \ref{l-hardy},
the Minkowski inequality and \eqref{atomd4}, we deduce that
\begin{eqnarray*}
&&\sum_{j=(j_P\vee0)}^\fz\lf\|\lf[
\frac{\chi_P}{\phi(P)}2^{js(\cdot)}{\rm S}_{j,2}\r]^{q(\cdot)}
\r\|_{L^{\frac{p(\cdot)}{q(\cdot)}}(\rn)}\\
&&\hs\hs\ls\sum_{j=(j_P\vee0)}^\fz\lf\{\sum_{i=0}^\fz2^{-i\tau}
\lf\|\lf[\frac{\chi_{Q_i^0}}{\phi({Q_i^0})}
\sum_{\ell(Q)=2^{-j}}|Q|^{-\frac{s(\cdot)}n}|t_Q|\wz\chi_Q
\r]^{rq(\cdot)}\r\|_{L^{\frac{p(\cdot)}{rq(\cdot)}}(\rn)}\r\}^{\frac1r}
+\sum_{j=0}^\fz2^{-j/r}\\
&&\hs\hs\ls\lf\{\sum_{i=0}^\fz2^{-i\tau}\lf(\sum_{j=(j_P\vee0)}^\fz
\lf\|\lf[\frac{\chi_{Q_i^0}}{\phi({Q_i^0})}
\sum_{\ell(Q)=2^{-j}}|Q|^{-\frac{s(\cdot)}n}|t_Q|\wz\chi_Q
\r]^{q(\cdot)}\r\|_{L^{\frac{p(\cdot)}{q(\cdot)}}(\rn)}\r)^r\r\}^{\frac1r}+1\\
&&\hs\hs\ls\lf\{\sum_{i=0}^\fz2^{-i\tau}\r\}^{1/r}+1\ls1,
\end{eqnarray*}
which, together with Lemma \ref{l-modular} and Remark \ref{re-mixed}(i)
again, implies that ${\rm I}_{P,2}\ls1$.

Let us prove \eqref{atom2-x} now. Obviously, it suffices to show that
$$\lf\|[\delta_v^P]^{-1}\lf[\frac{c\chi_P}{\phi(P)}
\sum_{\ell(Q)=2^{-v}}\frac{|Q|^{-\frac{s(\cdot)}n+\frac12}
|t_{Q}|}{(1+2^v|\cdot-x_{Q}|)^{M}}\r]^{rq(\cdot)}\r\|
_{L^{\frac{p(\cdot)}{rq(\cdot)}}(\rn)}\le1,$$
which, via Lemma \ref{l-inf-norm}, is a consequence of
$$\mathcal{A}:=\lf\|[\delta_v^P]^{-\frac1{rq(\cdot)}}\frac{\chi_P}{\phi(P)}
\sum_{\ell(Q)=2^{-v}}\frac{|Q|^{-\frac{s(\cdot)}n+\frac12}
|t_{Q}|}{(1+2^v|\cdot-x_{Q}|)^{M}}\r\|
_{\vlp}\ls1.$$
Taking $t\in(0,\min\{1,p_-\})$ and using some arguments
similar to those used in \cite[pp.\,29-30]{d13}, we conclude that, for all $x\in\rn$,
\begin{eqnarray}\label{atom2-y}
&&\sum_{\ell(Q)=2^{-v}}[\delta_v^P]^{-\frac1{rq(x)}}\frac{\chi_P(x)}{\phi(P)}
\sum_{\ell(Q)=2^{-v}}\frac{|Q|^{-\frac{s(x)}n+\frac12}
|t_{Q}|}{(1+2^v|x-x_{Q}|)^{M}}\\
&&\hs\hs\ls\sum_{i=0}^{(v-c_0)\vee0}
2^{i\zeta}\lf\{\cm\lf(\lf[\frac{\chi_{Q_i^0}}{\phi(P)}
\sum_{\ell(Q)=2^{-v}}\frac{2^{vs(\cdot)}}{[\delta_v^P]^{\frac 1{rq(\cdot)}}}
|t_Q|\wz\chi_Q\r]^t\r)\r\}^{1/t}
+\sum_{i=(v-c_0)\vee0}^{\fz}
2^{i\vartheta}\cdots,\noz
\end{eqnarray}
where $\zeta:=-M+\frac nt+\frac2r C_{\log}(q)+C_{\log}(s)$ and
$\vartheta:=-M+\frac nt+\frac2r(\frac1{q_-}-\frac1{q_+})+s_+-s_-$.
Taking $M$ large enough such that
$$M>\max\lf\{\frac nt+\frac2rC_{\log}(q)+C_{\log}(s)+\log_2c_1,\,
\frac2r\lf(\frac1{q_-}-\frac1{q_+}\r)+s_+-s_-+\log_2c_1\r\}+\tau,$$
then, by \eqref{phi-1}, Lemmas \ref{l-hlo} and \ref{l-inf-norm},
we know that
\begin{eqnarray*}
&&\sum_{i=0}^{(v-c_0)\vee0}2^{i\varsigma t}\lf\|
\cm\lf(\lf[\frac{\chi_{Q_i^0}}{\phi(P)}
\sum_{\ell(Q)=2^{-v}}\frac{2^{vs(\cdot)}}{[\delta_v^P]^{\frac 1{rq(\cdot)}}}
|t_Q|\wz\chi_Q\r]^t\r)\r\|_{L^{\frac{p(\cdot)}t}(\rn)}\\
&&\hs\ls\sum_{i=0}^{(v-c_0)\vee0}2^{it(\varsigma+\log_2c_1)}
\lf\|\frac{\chi_{Q_i^0}}{\phi(Q_i^0)}
\sum_{\ell(Q)=2^{-v}}\frac{2^{vs(\cdot)}}{[\delta_v^P]^{\frac 1{rq(\cdot)}}}
|t_Q|\wz\chi_Q)\r\|_{L^{p(\cdot)}(\rn)}^t\\
&&\hs\ls\!\sum_{i=0}^{(v-c_0)\vee0}\!\!\!2^{it(\varsigma+\log_2c_1)}
\!\!\lf\|\frac1{\delta_v^P}\lf[\frac{\chi_{Q_i^0}}{\phi(Q_i^0)}
\sum_{\ell(Q)=2^{-v}}2^{vs(\cdot)}
|t_Q|\wz\chi_Q\r]^{rq(\cdot)}\r\|_{L^{\frac{p(\cdot)}{rq(\cdot)}}(\rn)}^{\frac t
{rq_+}}\ls\sum_{i=0}^{\fz}2^{it(\varsigma+\log_2c_1+\tau)}\ls1,
\end{eqnarray*}
where we used the definition of $\delta_v^P$ in the penultimate inequality
and, similarly,
\begin{equation*}
\sum_{i=(v-c_0)\vee0}^{\fz}2^{i\vartheta t}\lf\|
\cm\lf(\lf[\frac{\chi_{Q_i^0}}{\phi(P)}
\sum_{\ell(Q)=2^{-v}}\frac{2^{vs(\cdot)}}{[\delta_v^P]^{\frac 1{rq(\cdot)}}}
|t_Q|\wz\chi_Q\r]^t\r)\r\|_{L^{\frac{p(\cdot)}t}(\rn)}\ls1.
\end{equation*}
From this and \eqref{atom2-y}, we deduce that $\ca\ls1$, which implies that
\eqref{atom2-x} holds true and then completes the proof
that I$_{P,2}\ls1$.

We next prove that ${\rm I}_{P,3}\ls1$.
To this end, it suffices to show that
$$\varrho_{\ell^{\frac{q(\cdot)}{r}}(L^{\frac{p(\cdot)}r})}
\lf(\lf\{\frac{\chi_P}{\wz C\phi(P)}2^{js(\cdot)}\sum_{v=(j_P\vee0)}^\fz
\sum_{\ell(Q)=2^{-v}}|t_Q||\vz_j\ast a_Q|\r\}_{j\ge(j_P\vee0)}^r\r)\le1$$
for some positive constant $\wz C$ large enough independent of $P$,
which, by Definition \ref{def2}, is equivalent to show that
$\sum_{j=(j_P\vee0)}^\fz {\rm Y}_j^P\ls1$,
where, for all $j\in\zz_+\cap [j_P\vee0,\fz)$,
$${\rm Y}_j^P:=\inf\lf\{\lz_j\in(0,\fz):\ \varrho_{\frac{p(\cdot)}{r}}
\lf(\frac{\chi_P[2^{js(\cdot)}\sum_{v=j}^\fz\sum_{\ell(Q)=2^{-v}}
|t_Q||\vz_j\ast a_Q|]^r}{\wz C[\phi(P)\lz_j^{1/q(\cdot)}]^r}\r)\le1\r\}.$$

We claim that, for all $P\in\cq$ and $j\in\zz_+\cap[j_P\vee0,\fz)$,
\begin{eqnarray}\label{5.14x}
{\rm Y}_j^P
&&\le2^{-j}+\sum_{v=j}^\fz2^{(j-v)d}
\sum_{i=0}^\fz2^{-i\wz d}\\
&&\hs\hs\times\inf\lf\{\xi_v\in(0,\fz):\ \varrho_{\frac{p(\cdot)}r}
\lf(\frac{\chi_{P_{i}}[\sum_{\ell(Q)=2^{-v}}|t_Q|2^{-vs(\cdot)}\chi_Q]^r}
{[\phi(P_{i})\xi_v^{1/q(\cdot)}]^r}\r)\le1\r\}\noz\\
&&=:2^{-j}+\sum_{v=j}^\fz2^{(j-v) d}
\sum_{i=0}^\fz2^{-i\wz  d}
{\rm Y}_{v,2}^P=:\delta_j^P,\noz
\end{eqnarray}
where $P_i:=Q(c_P,2^{i-j_P+c_0})$ with $c_0\in\nn$,
$ d$ is chosen such that $L+n-\frac nr+s_--\frac d{q_+}>0$ and $\wz d\in(0,\fz)$.

From the above claim, \eqref{atomd4}, \eqref{mixed-x} and Remark \ref{re-mixed}(i),
we deduce that
\begin{eqnarray*}
\sum_{j=(j_P\vee0)}^\fz{\rm Y}_j^P
&&\ls1+\sum_{v=(j_P\vee0)}^\fz\sum_{j=(j_P\vee0)}^{v}
2^{(j-v) d}\sum_{i=0}^\fz2^{-i\wz  d}
{\rm Y}_{v,2}^P
\ls1+\sum_{i=0}^\fz2^{-i\wz  d}\sum_{v=(j_P\vee0)}^\fz{\rm Y}_{v,2}^P
\ls1,
\end{eqnarray*}
which implies that ${\rm I}_{P,3}\ls1$ and
$\delta_j^P\in[2^{-j},2^{-j}+\theta]$ for some
$\theta\in[0,\fz)$

Therefore, to complete the estimate for I$_{P,3}$,
it remains to prove the above claim \eqref{5.14x}.
To this end, it suffices to show that, for all $j\in\zz_+\cap[j_P\vee0,\fz)$,
\begin{eqnarray*}
\inf\lf\{\wz\lz_j\in(0,\fz):\ \varrho_{\frac{p(\cdot)}{r}}
\lf(\frac{\chi_P[2^{js(\cdot)}\sum_{v=j}^\fz\sum_{\ell(Q)=2^{-v}}
|t_Q||\vz_j\ast a_Q|]^r}{[\phi(P)(\delta_j^P\wz\lz_j)^{1/q(\cdot)}]^r}\r)
\le1\r\}\ls1,
\end{eqnarray*}
which follows from the following estimate
\begin{equation}\label{atomd5}
{\rm H}_j^P:=\lf\|\frac{\chi_P2^{js(\cdot)r}}{\{\phi(P)[\delta_j^P]^{1/q(\cdot)}\}^r}
\sum_{v=j}^\fz\sum_{\ell(Q)=2^{-v}}|t_Q|^r|\vz_j\ast a_Q|^r
\r\|_{L^{\frac{p(\cdot)}r}}\ls1.
\end{equation}

Next we show \eqref{atomd5}.
By Lemma \ref{l-atom} and Remark \ref{re-vlp}, we find that
\begin{eqnarray}\label{atomd1-y}
{\rm H}_{j}^P\ls\sum_{v=j}^\fz2^{-(v-j)(L+n)r}
\lf\|\frac{\chi_P2^{js(\cdot)r}}{[\{\phi(P)[\delta_j^P]^{1/q(\cdot)}\}^r}
\sum_{k\in\zz^n}\frac{|t_{Q_{vk}}|^r|Q_{vk}|^{-\frac r2}}
{(1+2^j|\cdot-2^{-v}k|)^{Rr}}\r\|_{L^{\frac{p(\cdot)}{r}}(\rn)},
\end{eqnarray}
where $R$ can be large enough.
For all $x\in P$ and $v\in\zz_+$ with $v\ge j$, let
 $$\Omega_{0,j}^{x,v}:=\{k\in\zz^n:\ 2^j|x-2^{-v}k|\le1\}$$ and, for all $i\in\nn$,
 $\Omega_{i,j}^{x,v}:=\{k\in\zz^n:\ 2^{i-1}<2^j|x-2^{-v}k|\le2^i\}$.
 Then, we see that, for all $x\in P$,
 \begin{eqnarray}\label{atomd1-x}
{\rm J}(v,j,x,P)
:=&&\frac{2^{js(x)r}}{(\delta_j^P)^{\frac r{q(x)}}}
\sum_{k\in\zz^n}\frac{|t_{Q_{vk}}|^r|Q_{vk}|^{-\frac r2}}
{(1+2^j|x-2^{-v}k|)^{Rr}}\\
\sim&&\frac{2^{js(x)r}}{(\delta_j^P)^{\frac r{q(x)}}}
\sum_{i=0}^\fz\sum_{k\in\Omega_{i,j}^{x,v}}
|t_{Q_{vk}}|^r|Q_{vk}|^{-\frac r2}2^{-iRr}\noz\\
\sim&&\frac{2^{js(x)r}}{(\delta_j^P)^{\frac r{q(x)}}}\sum_{i=0}^\fz2^{-iRr}2^{vn}
\int_{\cup_{\wz k\in\Omega_{i,j}^{x,v}}Q_{v\wz k}}
\lf[\sum_{k\in\Omega_{i,j}^{x,v}}|t_{Q_{vk}}|\wz \chi_{Q_{vk}}(y)\r]^r\,dy.\noz
\end{eqnarray}
Since, for all $i\in\zz_+$, $v\in\zz_+$ with $v\ge j$,
$x\in P$ and $y\in\cup_{\wz k\in\Omega_{i,j}^{x,v}}Q_{v\wz k}$,
there exists $\wz k_y\in\Omega_{i,j}^{x,v}$ such that $y\in Q_{v\wz k_y}$,
it follows that
\begin{eqnarray}\label{atomd9}
1+2^j|x-y|&&\le 1+2^j|x-x_{Q_{v\wz k_y}}|+2^j|y-x_{Q_{v\wz k_y}}|
\ls 2^i+2^{j-v}\ls2^i
\end{eqnarray}
and hence
\begin{eqnarray}\label{atomd8}
|y-c_P|\le |y-x_{Q_{v\wz k_y}}|+|x-x_{Q_{v\wz k_y}}|+|x-c_P|
\ls2^{-v}+2^{i-j}+2^{-j_P}\ls2^{i-j_P}.
\end{eqnarray}
By \eqref{atomd8}, we see that, for all $i\in\zz_+$, $v\in\zz_+$ with $v\ge j$
and $x\in P$,
$$\bigcup_{\wz k\in\Omega_{i,j}^{x,v}}Q_{v\wz k}\subset Q(c_P,2^{i-j_P+c_0})=:Q_i^0$$
for some constant $c_0\in\nn$,
which, combined with \eqref{atomd1-x}, \eqref{atomd9} and Remark \ref{r-3.10x},
implies that
\begin{eqnarray}\label{ad-y}
{\rm J}(v,j,x,P)&&\ls(\delta_j^P)^{-\frac r{q(x)}}2^{js(x)r}\sum_{i=0}^\fz
2^{-iRr}2^{(v-j)n}2^{i(a+\vez)r}\\
&&\hs\hs\times\ \eta_{j,ar+\vez r}\ast
\lf(\lf[\sum_{\wz k\in\Omega_{i,j}^{x,v}}|t_{Q_{v\wz k}}|
\wz\chi_{Q_{vk}}\chi_{Q_i^0}\r]^r\r)(x)\noz\\
&&\ls2^{(v-j)(n-rs_-)}\sum_{i=0}^\fz2^{-ir(R-a-\vez)}
 \eta_{j,ar}\ast\lf(\lf[\sum_{k\in\zz^n}
(\delta_j^P)^{-\frac 1{q(\cdot)}}\r.\r.\noz\\
&&\hs\hs\times\lf.\lf.|t_{Q_{vk}}|
2^{vs(\cdot)}\wz\chi_{Q_{vk}}\chi_{Q_i^0}\r]^r\r)(x),\noz
\end{eqnarray}
where $\vez\in[C_{\log}(s)+C_{\log}(1/q),\fz)$.
From this, \eqref{atomd1-y}, Lemma \ref{l-conv-ineq} and Remark \ref{re-vlp}(ii),
we deduce that
\begin{eqnarray*}
{\rm H}_{j}^P
&&\ls\sum_{v=j}^\fz2^{-(v-j)(L+n-\frac nr+s_-)r}\sum_{i=0}^\fz2^{-i(R-a-\vez)r}
\frac{[\phi(Q_i^0)]^r}{[\phi(P)]^r}\\
&&\hs\hs\times\lf\|\frac{(\delta_j^P)^{-r/q(\cdot)}}
{\phi(Q_i^0)}\sum_{k\in\zz^n}|t_{Q_{vk}}|
2^{vs(\cdot)}\wz\chi_{Q_{vk}}\r\|_{L^{p(\cdot)}(Q_i^0)}^r\noz\\
&&\ls\sum_{v=j}^\fz2^{-(v-j)(L+n-\frac nr+s_--\frac{ d}{q_+})r}
\sum_{i=0}^\fz2^{-i(R-a-\vez-\log c_1-\wz  d/q_-)r}\noz\\
&&\hs\hs\times \lf\|\frac{2^{(v-j)
 dr/q(\cdot)}2^{-i\wz dr/q(\cdot)}}
{[\phi(Q_i^0)]^r(\delta_j^P)^{r/q(\cdot)}}\lf[\sum_{k\in\zz^n}|t_{Q_{vk}}|
2^{vs(\cdot)}\wz\chi_{Q_{vk}}\r]^r
\r\|_{L^{\frac{p(\cdot)}r}(Q_i^0)}\noz\\
&&\ls\sum_{v=j}^\fz2^{-(v-j)(L+n-\frac nr+s_--\frac{ d}{q_+})r}
\sum_{i=0}^\fz2^{-i(R-a-\vez-\log c_1-\frac{\wz  d}{q_-})r}\ls1,\noz
\end{eqnarray*}
where $R$ is chosen large enough such that
$R>a+\vez+\log c_1+\wz d/q_-$,
which completes the proof of that I$_{P,3}\ls1$ and hence the case I.

\textbf{Case II} $q_+=\fz$.

In this case, by Remark \ref{re-conv}(iii), we see that $q(x)=\fz$
for all $x\in\rn$. Thus, by Remark \ref{re-mixed}(v), we see that
$$\|t\|_{b_{p(\cdot),\fz}^{s(\cdot),\phi}(\rn)}
=\sup_{P\in\cq}\frac1{\phi(P)}\sup_{j\in\zz_+,j\ge(j_P\vee0)}
\lf\|\sum_{Q\in\cq^\ast,\ell(Q)=2^{-j}}|Q|^{-\frac{s(\cdot)}n}|t_Q|\wz\chi_Q
\r\|_{L^{p(\cdot)}(P)}.$$
Let $P$ be a given dyadic cube. Then, by \eqref{converge}, we find that, for
all $j\in\zz_+\cap[j_P\vee0,\fz)$,
\begin{eqnarray}\label{ad-z}
{\rm G}_P^j
:=&&\frac1{\phi(P)}\lf\|2^{js(\cdot)}|\vz_j\ast f|\r\|_{L^{p(\cdot)}(P)}\\
\ls&&\frac1{\phi(P)}\lf\|2^{js(\cdot)}\sum_{v=0}^{j-1}
\sum_{Q\in\cq^\ast,\ell(Q)=2^{-v}}|t_Q||\vz_j\ast a_Q|\r\|_{L^{p(\cdot)}(P)}\noz\\
&&\hs+\frac1{\phi(P)}\lf\|2^{js(\cdot)}\sum_{v=j}^{\fz}
\sum_{Q\in\cq^\ast,\ell(Q)=2^{-v}}|t_Q||\vz_j\ast a_Q|\r\|_{L^{p(\cdot)}(P)}
=:{\rm G}_{P,1}^j+{\rm G}_{P,2}^j.\noz
\end{eqnarray}

To estimate ${\rm G}_{P,1}^j$ and ${\rm G}_{P,2}^j$, we let $\vez\in(C_{\log}(s),\fz)$,
$r\in(0,\min\{1,p_-\})$ and $a\in(n/r,\fz)$.
For ${\rm G}_{P,1}^j$, by an argument similar to that used in the estimate
for I$_{P,1}$,
we conclude that there exists a positive constant $c_0$ such that
\begin{eqnarray*}
{\rm G}_{P,1}^j
&&\ls\Bigg\{\sum_{v=0}^j2^{(v-j)(K-s_+)r}\sum_{i=0}^\fz2^{-i(M-a-\vez/r)}\\
&&\hs\hs\times\lf.\frac1{\phi(P)}
\lf\|\sum_{Q\in\cq^\ast,\ell(Q)=2^{-v}}|t_Q|^r2^{vs(\cdot)r}|Q|^{-\frac r2}
\chi_{Q}\r\|_{L^{\frac{p(\cdot)}r}(Q(c_P,2^{i-v+c_0}))}\r\}^\frac1r,
\end{eqnarray*}
which, together with Remark \ref{r-lattice}, \eqref{phi-1}
and the facts that $c_1\in[1,\fz)$ and $j\ge j_P$, implies that
\begin{eqnarray}\label{ad-w}
{\rm G}_{P,1}^j
&&\ls\|t\|_{b_{p(\cdot),\fz}^{q(\cdot),\phi}(\rn)}
\lf\{\sum_{v=0}^j2^{(v-j)(K-s_+)r}\sum_{i=0}^\fz2^{-i(M-a-\frac{\vez}r)}
\frac{[\phi(Q(c_P,2^{i-v}))]^r}{[\phi(P)]^r}\r\}^\frac1r\\
&&\ls\|t\|_{b_{p(\cdot),\fz}^{q(\cdot),\phi}(\rn)}
\lf\{\sum_{v=0}^j2^{v(K-s_+-\log_2c_1)}\sum_{i=0}^\fz
2^{-i(M-a-\frac{\vez}r-\log_2c_1)}2^{j_P\log_2c_1}\r\}^\frac1r\noz\\
&&\ls\|t\|_{b_{p(\cdot),\fz}^{q(\cdot),\phi}(\rn)}
2^{-j\log_2c_1}2^{j_P\log_2c_1}\ls\|t\|_{b_{p(\cdot),\fz}^{q(\cdot),\phi}(\rn)}.\noz
\end{eqnarray}

For ${\rm G}_{P,2}^j$, by an argument similar to that used in the proof of
\eqref{ad-y}, we find that there exists $c_0\in\nn$ such that
\begin{eqnarray*}
{\rm G}_{P,2}^j
&&\ls\frac1{\phi(P)}\lf\{\sum_{v=j+1}^\fz2^{-(v-j)(L+n)r}\sum_{i=0}^\fz
2^{-i(R-a-\vez)r}\r.\\
&&\hs\hs\times\lf.\lf\|\eta_{j,ar}\ast\lf(\lf[\sum_{\gfz{Q\in\cq^\ast}
{\ell(Q)=2^{-v}}}|t_Q||Q|^{-\frac{s(\cdot)}{n}}\wz\chi_Q
\chi_{Q(c_P,2^{i-j_P+c_0})}\r]^r\r)\r\|
_{L^{\frac{p(\cdot)}r}(\rn)}\r\}^\frac1r,
\end{eqnarray*}
which, combined with \eqref{phi-1},
Remarks \ref{re-conv}(i) and \ref{r-lattice},
implies that
\begin{eqnarray*}
{\rm G}_{P,2}^j
&&\ls\frac1{\phi(P)}\lf\{\sum_{v=j+1}^\fz2^{-(v-j)(L+n)r}\sum_{i=0}^\fz
2^{-i(R-a-\vez)r}\r.\\
&&\hs\hs\times\lf.\lf\|\lf(\sum_{Q\in\cq^\ast,\ell(Q)=2^{-v}}
|t_Q||Q|^{-\frac{s(\cdot)}{n}}\wz\chi_Q\r)\r\|
_{L^{p(\cdot)}(Q(c_P,2^{i-j_P+c_0}))}^r\r\}^\frac1r\\
&&\ls\|t\|_{b_{p(\cdot),\fz}^{q(\cdot),\phi}(\rn)}
\lf\{\sum_{v=j+1}^\fz2^{-(v-j)(L+n)r}\sum_{i=0}^\fz
2^{-i(R-a-\vez-\log_2c_1)r}\r\}^\frac1r
\ls\|t\|_{b_{p(\cdot),\fz}^{q(\cdot),\phi}(\rn)}.
\end{eqnarray*}
By this, \eqref{ad-z} and \eqref{ad-w}, we conclude that
\begin{eqnarray*}
\|f\|_{B_{p(\cdot),\fz}^{s(\cdot),\phi}(\rn)}
&&\ls\sup_{P\in\cq}\frac1{\phi(P)}\sup_{j\in\zz_+\cap[(j_P\vee0),\fz)}
\lf\|2^{js(\cdot)}|\vz_j\ast f|\r\|_{L^{p(\cdot)}(P)}\\
&&\ls\sup_{P\in\cq}\frac1{\phi(P)}\sup_{j\in\zz_+\cap[(j_P\vee0),\fz)}
({\rm G}_{P,1}^j+{\rm G}_{P,2}^j)
\ls\|t\|_{b_{p(\cdot),\fz}^{q(\cdot),\phi}(\rn)},
\end{eqnarray*}
which completes the proof of the case II.

Combining Cases I and II, we conclude that \eqref{atomd2} holds true.
This finishes the proof of Theorem \ref{t-atomd}.
\end{proof}
\begin{remark}\label{r-key}
We point out that the method used in the proof of Lemma \ref{l-equi} plays
a very important role in the proof of Theorem \ref{t-atomd}.
Precisely, the argument used in proofs of \eqref{5.14x} and \eqref{atom2-x}
is similar to that used in the proof of Lemma \ref{l-equi}.
\end{remark}

\section{An application to trace operators\label{s6}}

\hskip\parindent
The purpose of this section is to study the trace of Besov-type spaces
with variable smoothness and integrability.

Let $f\in\bbeve$. Then, by Theorem \ref{t-atomd}, we can write
$f=\sum_{Q\in\cq^\ast}t_Qa_Q$ in $\cs'(\rn)$, where $\{a_Q\}_{Q\in\cq^\ast}$
is a family of smooth atoms of $\bbeve$ and $\{t_Q\}_{Q\in\cq^\ast}\subset\cc$
satisfies
$$\|\{t_Q\}_{Q\in\cq^\ast}\|_{\beve}\le C\|f\|_{\bbeve}$$
with $C$ being a positive constant independent of $f$.
Define the \emph{trace} of $f$ by setting, for all $\widetilde{x}\in\rr^{n-1}$,
\begin{equation}\label{trace-d}
\mathop\mathrm{Tr}(f)(\widetilde{x}):=\sum_{Q\in\cq^\ast}t_Qa_Q(\widetilde{x},0).
\end{equation}
This definition of $\mathop\mathrm{Tr}(f)$ is determined canonical for all $f\in \bbeve$,
since the actual construction of $a_Q$ in the proof of Theorem \ref{t-atomd}
implies that $t_Qa_Q$ is obtained canonical. Moreover, Lemma \ref{l-trace2} below shows that
the summation in \eqref{trace-d} converges in $\cs'(\rr^{n-1})$. Thus, the trace
operator is well defined on $\bbeve$.

To state our main result of this section, we adopt the following notation.
For $p,\ q,\ s$ and $\phi$ as in Definition
\ref{def-b}, let, for all $\widetilde{x}\in\rr^{n-1}$,
$$\widetilde p(\widetilde{x}):=p(\widetilde{x},0),\quad\widetilde q(\widetilde{x}):=q(\widetilde{x},0),
\quad\widetilde s(\widetilde{x}):=s(\widetilde{x},0)$$
and, for all cubes $\widetilde{Q}$ of $\rr^{n-1}$, $\widetilde\phi(\widetilde{Q}):=\phi(\widetilde{Q}\times[0,\ell(\widetilde{Q}))$.
In what follows, let $\rr_+^n:=\rr^{n-1}\times[0,\fz)$,
$\rr_-^n:=\rr^{n-1}\times(-\fz,0]$ and $k:=(k_1,\dots,k_n)\in\zz^n$.
Denote by $C_c^\fz(\rr)$ the \emph{set} of all continuous functions $f$ on $\rr$
with compact support satisfying that all classical derivatives of $f$
are also continuous.
\begin{theorem}\label{t-trace}
Let $n\ge2$, $\phi\in\cg(\urn)$, $p,\ q\in C^{\log}(\rn)$ and
$s\in  C_{\rm loc}^{\log}(\rn)\cap L^\fz(\rn)$ satisfy
\begin{equation}\label{condition}
s_--\frac1{p_-}-(n-1)\lf(\frac1{\min\{1,p_-\}}-1\r)>0.
\end{equation}
Then
$$\mathop\mathrm{Tr}\bbeve=B_{\widetilde p(\cdot),\widetilde q(\cdot)}^{\widetilde s(\cdot)-
\frac1{\widetilde p(\cdot)},\widetilde\phi}(\rr^{n-1}).$$
\end{theorem}

\begin{remark}
(i) Using quarkonial characterizations of both $B_{p(\cdot),q(\cdot)}^{s(\cdot)}(\rn)$
and $B_{\widetilde p(\cdot),\widetilde q(\cdot)}^{\widetilde s(\cdot)-
\frac1{\widetilde p(\cdot)}}(\rr^{n-1})$, Noi \cite[Theorem 5.1]{noi14} proved the
following conclusion: $\mathop\mathrm{Tr}B_{p(\cdot),q(\cdot)}^{s(\cdot)}(\rn)
=B_{\widetilde p(\cdot),\widetilde q(\cdot)}^{\widetilde s(\cdot)-
\frac1{\widetilde p(\cdot)}}(\rr^{n-1})$ under a weaker condition that
$$\mathop{\rm ess\,inf}\limits_{x\in \rn}
\lf\{s(x)-\frac1{p(x)}-(n-1)\lf(\frac1{\min\{1,p(x)\}}-1\r)\r\}>0,$$
but $s\in C^{\log}(\rn)$ is required, which is stronger than the corresponding
one in Theorem \ref{t-trace}.

(ii) When $p_+\in(0,\fz)$ and $q(\cdot)\equiv q\in(0,\fz)$ is a constant, the
conclusion of Theorem \ref{t-trace} was proved
by Moura et al. \cite[Theorem 3.4]{mns13} under the condition \eqref{condition}.

(iii) When $p,\ q,\ s$ and $\phi$ are as in Remark \ref{r-defi}(ii),
Theorem \ref{t-trace} coincides with \cite[Theorem 6.8]{ysiy}.
\end{remark}
The following conclusion implies that the
summation in \eqref{trace-d} converges in $\cs'(\rr^{n-1})$,
whose proof is similar to that of \cite[Lemma 4.3]{yyz14}, the details being omitted.
\begin{lemma}\label{l-trace2}
Let $n$, $p$, $q$, $s$ and $\phi$ be as in Theorem \ref{t-trace}.
Then, for all $f\in\bbeve$, $\mathop\mathrm{Tr}(f)\in\cs'(\rn)$.
\end{lemma}
By Lemma \ref{l-conv-ineq} and an argument similar to that used in the proof of
\cite[Proposition 4.6]{yyz14}, we obtain Lemma \ref{l-trace1} below, the
details being omitted. The corresponding result in the case that
 $\phi\equiv1$ was obtained in \cite[Lemma 5.3]{noi14}.
\begin{lemma}\label{l-trace1}
Let $p,\ q\in C^{\log}(\rn)$, $s\in C_{\loc}^{\log}(\rn)\cap L^\fz(\rn)$
and $\phi\in\cg(\urn)$.
Let $\delta\in(0,\fz)$ and $\{E_Q\}_{Q\in\cq^\ast}$
be a collection of sets such that, for all
$Q\in \cq^\ast$, $E_{Q}\subset 4Q$ and $|E_{Q}|\ge\delta|Q|$.
Then, for all $t:=\{t_{Q}\}_{Q\in\cq^\ast}\subset\cc$,
$t\in\beve$ if and only if
$$\|t\|_{\wz \beve}
:=\sup_{P\in\cq}\frac1{\phi(P)}\lf\|\lf\{\sum_{\gfz{Q\in\cq^\ast}{\ell(Q)=2^{-j}}}
2^{j[s(\cdot)+\frac n2]}|t_Q|\chi_{E_Q}\r\}_{j\ge(j_P\vee0)}
\r\|_{\ell^{q(\cdot)}(L^{p(\cdot)}(P))}<\fz.
 $$
\end{lemma}

In what follows, for all $j\in\zz$ and $\widetilde{k}\in\zz^{n-1}$,
let $\widetilde Q_{j\widetilde{k}}:=2^{-j}([0,1)^{n-1}+\widetilde{k})$
be the dyadic cube of $\rr^{n-1}$,
$\widetilde{\cq}$ the set of all dyadic cubes of $\rr^{n-1}$ and
$\widetilde{\cq}^\ast:=\{\widetilde{Q}\in\widetilde{\cq}:\ \ell(\widetilde{Q})\le1\}$.
For all $\widetilde{Q}\in\widetilde{\cq}^\ast$ and $i\in\zz$, let
$\wz \chi_{\widetilde{Q}}:=|\widetilde{Q}|^{-1/2}\chi_{\widetilde{Q}}$, $I_{\widetilde{Q}}^i:=[(i-1)\ell(\widetilde{Q}),i\ell(\widetilde{Q}))$ and
$(\wh {\widetilde{Q}})_i:=\widetilde{Q}\times I_{\widetilde{Q}}^i$. For all $P\in\cq$,
denote by $P_{\rr^{n-1}}^\bot$ the \emph{vertical projection} of $P$ on $\rr^{n-1}$, namely,
$$P_{\rr^{n-1}}^\bot:=\{\widetilde{x}\in \rr^{n-1}:\ \exists\ x_n\in\rr\ {\rm s.\,t.}\
(\widetilde{x},x_n)\in P\}$$
and, for all $j\in\zz_+$, let
$P_{\rr^{n-1}}^{\bot,j}
:=\{\widetilde{Q}\in\widetilde{\cq}^\ast:\ \widetilde{Q}\subset P_{\rr^{n-1}}^\bot,\ \ell(\widetilde{Q})=2^{-j}\}$.

Applying Lemma \ref{l-trace1}, we conclude that the following conclusion holds true,
which, in the case that $\phi\equiv1$, was proved in \cite[Lemma 5.4]{noi14}.
\begin{lemma}\label{l-trace3}
Let $p_1,\ p_2,\ q_1,\ q_2\in C^{\log}(\rn)$,
$s_1,\ s_2\in C_{\loc}^{\log}(\rn)\cap L^\fz(\rn)$
and $\phi\in\cg(\urn)$.
Assume that $p_1=p_2$, $q_1=q_2$ and $s_1=s_2$ on $\rr_-^n$ or $\rr_+^n$. Then,
for all $\{t_{Q}\}_{Q\in\cq^\ast}\subset\cc$ and $i\in\{0,1,2\}$,
\begin{equation*}
\lf\|\lf\{t_{(\wh {\widetilde{Q}})_i}\r\}_{\widetilde{Q}\in\widetilde{\cq}^\ast}
\r\|_{b_{p_1(\cdot),q_1(\cdot)}^{s_1(\cdot),\phi}(\rn)}
\sim\lf\|\lf\{t_{(\wh{\widetilde{Q}})_i}\r\}_{\widetilde{Q}\in\widetilde{\cq}^\ast}
\r\|_{b_{p_2(\cdot),q_2(\cdot)}^{s_2(\cdot),\phi}(\rn)},
\end{equation*}
where the implicit positive constants are independent of $\{t_{Q}\}_{Q\in\cq^\ast}$.
\end{lemma}

\begin{proof}
By similarity, we only consider the case that
$p_1=p_2$, $q_1=q_2$ and $s_1=s_2$ on $\rr_+^n$.
For all $\widetilde{Q}\in\widetilde{\cq}^\ast$ and $i\in\{0,1,2\}$, let
$$E_{(\wh{\widetilde{Q}})_i}
:=\lf\{(\widetilde{x},x_n)\in\rn:\ \widetilde{x}\in \widetilde{Q},\
\frac{i+1}2\ell(\widetilde Q)\le x_n
<\frac{3(i+1)}4\ell(\widetilde Q)\r\}.$$
Then $E_{(\wh{\widetilde{Q}})_i}\subset\rr^n_+$, $E_{(\wh{\widetilde{Q}})_i}\subset 4(\wh{\widetilde{Q}})_i$ and
$|E_{(\wh{\widetilde{Q}})_i}|\ge \frac4{i+1}|(\wh{\widetilde{Q}})_i|$.
 By this and Lemma \ref{l-trace1}, we conclude that
\begin{eqnarray*}
&&\lf\|\lf\{t_{(\wh{\widetilde{Q}})_i}\r\}_{\widetilde{Q}\in\widetilde{\cq}^\ast}
\r\|_{b_{p_1(\cdot),q_1(\cdot)}^{s_1(\cdot),\phi}(\rn)}\\
&&\hs\sim\sup_{P\in\cq}\frac 1{\phi(P)}\lf\|\lf\{
2^{js_1(\cdot)}\sum_{\widetilde{Q}\in P_{\rr^{n-1}}^{\bot,j}}
\lf|t_{(\wh{\widetilde{Q}})_i}\r|\lf|(\wh{\widetilde{Q}})_i\r|^{-\frac12}\chi_{E_{(\wh{\widetilde{Q}})_i}}
\r\}_{j\ge(j_P\vee0)}\r\|_{\ell^{q_1(\cdot)}(L^{p_1(\cdot)}(P))}\\
&&\hs\sim\sup_{P\in\cq}\frac 1{\phi(P)}\lf\|\lf\{
2^{js_2(\cdot)}\sum_{\widetilde{Q}\in P_{\rr^{n-1}}^{\bot,j}}
\lf|t_{(\wh{\widetilde{Q}})_i}\r|\lf|(\wh{\widetilde{Q}})_i\r|^{-\frac12}\chi_{E_{(\wh{\widetilde{Q}})_i}}
\r\}_{j\ge(j_P\vee0)}\r\|_{\ell^{q_2(\cdot)}(L^{p_2(\cdot)}(P))}\\
&&\hs\sim\lf\|\lf\{t_{(\wh{\widetilde{Q}})_i}\r\}_{\widetilde{Q}\in\widetilde{\cq}^\ast}
\r\|_{b_{p_2(\cdot),q_2(\cdot)}^{s_2(\cdot),\phi}(\rn)},
\end{eqnarray*}
which completes the proof of Lemma \ref{l-trace3}.
\end{proof}

Adopting an argument similar to that used in the proof \cite[Proposition 7.3]{dhr09},
we obtain the following conclusion, the details being omitted.
\begin{corollary}\label{c-trace}
Let $p_1,\ p_2,\ q_1,\ q_2\in C^{\log}(\rn)$,
$s_1,\ s_2\in C_{\loc}^{\log}(\rn)\cap L^\fz(\rn)$
and $\phi\in\cg(\urn)$.
Assume that $p_1=p_2$, $q_1=q_2$ and $s_1=s_2$ on $\rr^{n-1}\times\{0\}$. Then,
for all $t:=\{t_{(\wh{\widetilde{Q}})_i}\}_{\widetilde{Q}\in\widetilde{\cq}^\ast}\subset\cc$
 and $i\in\{0,1,2\}$,
\begin{equation*}
\lf\|\lf\{t_{(\wh{\widetilde{Q}})_i}\r\}_{\widetilde{Q}\in\widetilde{\cq}^\ast}
\r\|_{b_{p_1(\cdot),q_1(\cdot)}^{s_1(\cdot),\phi}(\rn)}
\sim\lf\|\lf\{t_{(\wh{\widetilde{Q}})_i}\r\}_{\widetilde{Q}\in\widetilde{\cq}^\ast}
\r\|_{b_{p_2(\cdot),q_2(\cdot)}^{s_2(\cdot),\phi}(\rn)},
\end{equation*}
where the implicit positive constants are independent of $t$.
\end{corollary}

For the notation simplicity, let $\widetilde \beta(\widetilde{x}):=\widetilde s(\widetilde{x})-\frac1{\widetilde p(\widetilde{x})}$ for all
$\widetilde{x}\in\rr^{n-1}$.
\begin{lemma}\label{l-trace4}
Let $p,\ q\in C^{\log}(\rn)$, $s\in C_{\loc}^{\log}(\rn)$
and $\phi\in\cg(\urn)$.
Then there exists a positive constant $C$ such that,
for all $t:=\{t_{Q}\}_{Q\in\cq^\ast}\subset\cc$ and $i\in\{0,1,2\}$,
\begin{eqnarray}\label{6.x1}
\lf\|\lf\{t_{(\wh{\widetilde{Q}})_i}\r\}_{\widetilde{Q}\in\widetilde{\cq}^\ast}
\r\|_{b_{\widetilde p(\cdot),\widetilde q(\cdot)}^{\widetilde s(\cdot)-\frac1{\widetilde p(\cdot)},
\widetilde\phi}(\rr^{n-1})}
\sim\lf\|\lf\{t_{(\wh{\widetilde{Q}})_i}[\ell(\widetilde{Q})]^{\frac12}\r\}_{\widetilde{Q}\in\widetilde{\cq}^\ast}\r\|_{\beve},
\end{eqnarray}
where the implicit positive constants are independent of $t$.
\end{lemma}
\begin{proof}
By similarity, we only give the proof of ``$\ls$" in \eqref{6.x1}.
By Corollary \ref{c-trace}, we may assume that $p$, $q$ and $s$ are independent
of the $n$-th coordinate $x_n$ with $|x_n|\le2$.
For all $\widetilde{P}\in\widetilde{\cq}$, $j\in\zz_+$, $\widetilde{x}\in\rr^{n-1}$ and $x\in\rn$, let
$\wh {\widetilde{P}}:=\widetilde{P}\times[0,\ell(\widetilde{P}))$,
$$\Gamma_{\widetilde{P}}^j:=\lf\{Q\in\widetilde{\cq}^\ast:\ \widetilde{Q}\subset \widetilde{P},\ \ell(\widetilde{Q})=2^{-j}\r\},\qquad H_{\widetilde{P}}^j(\widetilde{x}):=\sum_{\widetilde{Q}\in\Gamma_{\widetilde{P}}^j}
\lf|t_{(\wh{\widetilde{Q}})_i}\r|\wz\chi_{\widetilde{Q}}(\widetilde{x})$$
and
$$G_{{\widetilde{P}}}^j(x):=\sum_{\widetilde{Q}\in\Gamma_{\widetilde{P}}^j}
\lf|t_{(\wh{\widetilde{Q}})_i}\r|[\ell(\widetilde{Q})]^{\frac12}
\wz\chi_{(\wh{\widetilde{Q}})_i}(x).$$

Let $\widetilde{P}\in\widetilde{\cq}$ be a given dyadic cube.
Then, by \eqref{phi-1}, we find that,
for all $j\in\zz_+\cap[(j_{\widetilde{P}}\vee0),\fz)$ and $\lz,\ \mu\in(0,\fz)$,
\begin{eqnarray*}
&&\int_{\widetilde{P}}\lf[\frac1{\mu}\lf\{[\lz\widetilde\phi(\widetilde{P})]^{-1}2^{j\widetilde\beta(\widetilde{x})}
H_{\widetilde{P}}^j(\widetilde{x})\r\}^{\widetilde q(\widetilde{x})}
\r]^{\frac{\widetilde p(\widetilde{x})}{\widetilde q(\widetilde{x})}}\,d\widetilde{x}\\
&&\hs=\int_{\widetilde{P}}\lf[\int_{(i-1)2^{-j}}^{i2^{-j}}
\frac{2^{j\widetilde s(\widetilde{x})\widetilde p(\widetilde{x})}}{\mu^{\widetilde p(\widetilde{x})/\widetilde q(\widetilde{x})}}
\lf\{[\lz\widetilde\phi(\widetilde{P})]^{-1}H_{\widetilde{P}}^j(\widetilde{x})
\r\}^{\widetilde p(\widetilde{x})}\chi_{I_{\widetilde{Q}}^i}(x_n)\,dx_n\r]\,d\widetilde{x}\\
&&\hs=\int_{\widetilde{P}}\int_{(i-1)2^{-j}}^{i2^{-j}}
\frac{2^{j s(\widetilde{x},x_n) p(\widetilde{x},x_n)}}{\mu^{p(\widetilde{x},x_n)/q(\widetilde{x},x_n)}}
\lf\{[\lz\widetilde\phi(\widetilde{P})]^{-1}H_{\widetilde{P}}^j(\widetilde{x})\chi_{I_{\widetilde{Q}}^i}(x_n)
\r\}^{ p(\widetilde{x},x_n)}\,dx_nd\widetilde{x}\\
&&\hs\ls\int_{4\wh {\widetilde{P}}}
\lf[\frac1{\mu}\lf\{\lf[\lz\phi(4\wh{\widetilde{P}})\r]^{-1}2^{js(x)}
G_{{\widetilde{P}}}^j(x)\r\}^{q(x)}
\r]^{\frac{p(x)}{q(x)}}\,dx,
\end{eqnarray*}
which implies that
\begin{eqnarray*}
&&\lf\|\lf\{[\lz\phi(\widetilde{P})]^{-1}2^{j\widetilde \beta(\cdot)}
H_{\widetilde{P}}^j\r\}^{\widetilde q(\cdot)}
\r\|_{L^{\frac{\widetilde p(\cdot)}{\widetilde q(\cdot)}}(\widetilde{P})}
\ls\lf\|\lf\{\lf[\lz\phi(4\wh{\widetilde{P}})\r]^{-1}2^{js(\cdot)}
G_{\widetilde{P}}^j\r\}^{q(\cdot)}
\r\|_{L^{\frac{p(\cdot)}{q(\cdot)}}(4\wh {\widetilde{P}})}.
\end{eqnarray*}
From this and Remark \ref{r-lattice}, we deduce that
\begin{eqnarray*}
&&\frac1{\widetilde\phi(\widetilde{P})}\lf\|\lf\{2^{j\widetilde\beta(\cdot)}H_{\widetilde{P}}^j\r\}_{j\ge(j_{\widetilde{P}}\vee0)}
\r\|_{\ell^{\widetilde q(\cdot)}(L^{\widetilde p(\cdot)}(\widetilde{P}))}\\
&&\hs=\inf\lf\{\lz\in(0,\fz):\ \sum_{j=(j_{\widetilde{P}}\vee0)}^\fz
\lf\|\lf\{[\lz\widetilde\phi(\widetilde{P})]^{-1}2^{j\widetilde \beta(\cdot)}
H_{\widetilde{P}}^j\r\}^{\widetilde q(\cdot)}
\r\|_{L^{\frac{\widetilde p(\cdot)}{\widetilde q(\cdot)}}(\widetilde{P})}\le1\r\}\\
&&\hs\ls\frac1{\phi(4\wh{\widetilde{P}})}\lf\|\lf\{2^{js(\cdot)}G_{\widetilde{P}}^j\r\}_{j\ge(j_{\widetilde{P}}\vee0)}
\r\|_{\ell^{q(\cdot)}(L^{p(\cdot)}(4\wh {\widetilde{P}}))}
\ls\lf\|\lf\{t_{(\wh{\widetilde{Q}})_i}[\ell(\widetilde{Q})]^{-\frac12}
\r\}_{\widetilde{Q}\in\widetilde{\cq}^\ast}\r\|_{\beve},
\end{eqnarray*}
which, combined with the arbitrariness of $\widetilde{P}$, further implies that
\begin{eqnarray*}
\lf\|\lf\{t_{(\wh{\widetilde{Q}})_i}\r\}_{\widetilde{Q}\in\widetilde{\cq}^\ast}
\r\|_{b_{\widetilde p(\cdot),\widetilde q(\cdot)}^{\widetilde s(\cdot)-\frac1{\widetilde p(\cdot)},
\widetilde\phi}(\rr^{n-1})}
\ls \lf\|\lf\{t_{(\wh{\widetilde{Q}})_i}[\ell(\widetilde{Q})]^{\frac12}\r\}_{\widetilde{Q}\in\widetilde{\cq}^\ast}\r\|_{\beve}.
\end{eqnarray*}
This finishes the proof of Lemma \ref{l-trace4}.
\end{proof}
\begin{proof}[Proof of Theorem \ref{t-trace}]
 Let $f\in\bbeve$. Then, by Theorem \ref{t-atomd}, we have an atomic decomposition
$f=\sum_{Q\in\cq^\ast}t_Qa_Q$ in $\cs'(\rn)$,
where $\{a_Q\}_{Q\in\cq^\ast}$ is a family of smooth atoms of
$B_{p(\cdot),p(\cdot)}^{s(\cdot),\phi}(\rn)$
 and
$t:=\{t_Q\}_{Q\in\cq^\ast}\subset \cc$ satisfies
\begin{equation}\label{trace-y}
\|t\|_{f_{p(\cdot),p(\cdot)}^{s(\cdot),\phi}(\rn)}
\ls\|f\|_{F_{p(\cdot),p(\cdot)}^{s(\cdot),\phi}(\rn)}.
\end{equation}
Since ${\rm supp}\,a_Q\subset 3Q$ for each $Q\in\cq^\ast$, it follows that,
if $i\notin\{0,1,2\}$, then, for each $\widetilde{Q}\in\widetilde{\cq}^\ast$,
 $a_{(\wh{\widetilde{Q}})_i}(\cdot,0)=0$,
which implies that $\mathop\mathrm{Tr}(f)$ can be rewritten as, for all $\widetilde x\in\rr^{n-1}$,
\begin{equation}\label{trace-1}
\mathop\mathrm{Tr}(f)(\widetilde x,0)=\sum_{i=0}^2\sum_{\widetilde{Q}\in\widetilde{\cq}^\ast}
t_{{(\wh{\widetilde{Q}})_i}}a_{(\wh {\widetilde{Q}})_i}(\widetilde x,0)
=:\sum_{i=0}^2\sum_{\widetilde{Q}\in\widetilde{\cq}^\ast}
\lz_{(\wh{\widetilde{Q}})_i}b_{(\wh{\widetilde{Q}})_i}(\widetilde x),
\end{equation}
where, for each $\widetilde{Q}\in \widetilde{\cq}^\ast$ and
$\widetilde x\in\rr^{n-1}$,
$b_{(\wh{\widetilde{Q}})_i}(\widetilde x)
:=[\ell(\widetilde{Q})]^\frac12a_{(\wh{\widetilde{Q}})_i}(\widetilde x,0)$ and
$\lz_{(\wh{\widetilde{Q}})_i}:=[\ell(\widetilde{Q})]^{-\frac12}t_{(\wh{\widetilde{Q}})_i}$.
Since $a_Q$ is a smooth atom supported near $Q$ of $\bbeve$,
by \eqref{condition}, we easily find that, for each $\widetilde{Q}\in \widetilde{\cq}^\ast$,
$b_{(\wh{\widetilde{Q}})_i}$ is also a smooth atom of
$B_{\widetilde p(\cdot),\widetilde q(\cdot)}^{\widetilde \beta(\cdot),\widetilde\phi}(\rr^{n-1})$ supported near $\widetilde{Q}$.
On the other hand, by Lemma \ref{l-trace4} and \eqref{trace-y}, we find that
\begin{eqnarray*}
\lf\|\lf\{\lz_{(\wh{\widetilde{Q}})_i}\r\}_{\widetilde{Q}\in\widetilde{\cq}^\ast}\r\|_
{b_{\widetilde p(\cdot),\widetilde q(\cdot)}^{\widetilde \beta(\cdot),\widetilde\phi}(\rr^{n-1})}
&&\ls\lf\|\lf\{\lz_{(\wh{\widetilde{Q}})_i}[\ell(\widetilde{Q})]^{\frac12}
\r\}_{\widetilde{Q}\in\widetilde{\cq}^\ast}\r\|_{\beve}\\
&&\sim\lf\|\lf\{t_{(\wh{\widetilde{Q}})_i}\r\}_{\widetilde{Q}\in\widetilde{\cq}^\ast}\r\|_{\beve}\ls\|f\|_{\bbeve}.
\end{eqnarray*}
Therefore, by Theorem \ref{t-atomd} and \eqref{trace-1}, we conclude that
\begin{eqnarray*}
\|\mathop\mathrm{Tr}(f)\|_{B_{\widetilde p(\cdot),\widetilde q(\cdot)}^{\widetilde \beta(\cdot),\widetilde\phi}(\rr^{n-1})}
\ls\sum_{i=0}^2\lf\|\lf\{\lz_{(\wh{\widetilde{Q}})_i}\r\}_{\widetilde{Q}\in\widetilde{\cq}^\ast}\r\|_
{b_{\widetilde p(\cdot),\widetilde q(\cdot)}^{\widetilde \beta(\cdot),\widetilde\phi}(\rr^{n-1})}
\ls\|f\|_{\bbeve}.
\end{eqnarray*}

Conversely, we prove that the operator Tr is surjective.
Let $f\in B_{\widetilde p(\cdot),\widetilde q(\cdot)}^{\widetilde \beta(\cdot),\widetilde\phi}(\rr^{n-1})$. Then,
by Theorem \ref{t-atomd}, we find that there exist a sequence
$\{\lz_{\widetilde{Q}}\}_{\widetilde{Q}\in \widetilde{\cq}^\ast}\subset\cc$ and
a family $\{a_{\widetilde{Q}}\}_{\widetilde{Q}\in \widetilde{\cq}^\ast}$ of smooth atoms
of $B_{\widetilde p(\cdot),\widetilde q(\cdot)}^
{\widetilde \beta(\cdot),\widetilde\phi}(\rr^{n-1})$ such that
$f=\sum_{\widetilde{Q}\in\widetilde{\cq}^\ast}\lz_{\widetilde{Q}}
a_{\widetilde{Q}}$ in $\cs'(\rr^{n-1})$ and
\begin{equation}\label{trace-2}
\|\{\lz_{\widetilde{Q}}\}_{\widetilde{Q}\in\widetilde{\cq}^\ast}
\|_{b_{\widetilde p(\cdot),\widetilde q(\cdot)}^{\widetilde \beta(\cdot),
\widetilde\phi}(\rr^{n-1})}
\ls \|f\|_{B_{\widetilde p(\cdot),\widetilde q(\cdot)}^{\widetilde
\beta(\cdot),\widetilde\phi}(\rr^{n-1})}.
\end{equation}

Similar to the proof of \cite[Theorem 6.8]{ysiy}, we choose a function
$\eta\in C_c^\fz(\rn)$ satisfying $\supp\eta\subset(-1/2,1/2)$ and $\eta(0)=1$.
For all $\widetilde{Q}\in\widetilde{\cq}^\ast$ and $\xi\in\rr$, let
$\eta_{\widetilde{Q}}(\xi):=\eta(2^{-\log_2\ell(\widetilde{Q})}\xi)$. Then
$\supp\eta_{\widetilde{Q}}\subset (-\ell(\widetilde{Q}),\ell(\widetilde{Q}))$. Let
\begin{equation}\label{trace-3}
g:=\sum_{\widetilde{Q}\in\widetilde{\cq}^\ast}\lz_{\widetilde{Q}}a_{\widetilde{Q}}\otimes\eta_{\widetilde{Q}}
=:\sum_{i=0}^1\sum_{\widetilde{Q}\in \widetilde{\cq}^\ast}t_{(\wh{\widetilde{Q}})_i}b_{(\wh{\widetilde{Q}})_i},
\end{equation}
where, for all $Q\in\cq^\ast$ and $(\widetilde{x},x_n)\in\rn$,
$$b_{Q}(\widetilde{x},x_n):=[\ell(\widetilde{Q})]^{-\frac12}a_{\widetilde{Q}}\otimes\eta_{\widetilde{Q}}(\widetilde{x},x_n)
=:[\ell(\widetilde{Q})]^{-\frac12}a_{\widetilde{Q}}(\widetilde{x})\eta_{\widetilde{Q}}(x_n),$$
$t_Q:=[\ell(\widetilde{Q})]^{1/2}\lz_{\widetilde{Q}}$ if $Q=(\wh {\widetilde{Q}})_i$ for some $i\in\{0,1\}$
and $t_Q:=0$ otherwise.
By the construction of $\{b_Q\}_{Q\in\cq^\ast}$, we easily find that, for each
$Q\in\cq^\ast$, $b_Q$ is a smooth atom supported near $Q$ of $\bbeve$.
On the other hand, by Lemma \ref{l-trace4} and \eqref{trace-2},
 we conclude that, for each $i\in\{0,1\}$,
\begin{eqnarray*}
\lf\|\lf\{t_{(\wh {\widetilde{Q}})_i}\r\}_{\widetilde{Q}\in \widetilde{\cq}^\ast}\r\|_{\beve}
&&\sim\lf\|\lf\{\lz_{\widetilde{Q}}[\ell(\widetilde{Q})]^{1/2}\r\}_{\widetilde{Q}\in \widetilde{\cq}^\ast}\r\|_{\beve}\\
&&\ls\lf\|\{\lz_{\widetilde{Q}}\}_{\widetilde{Q}\in\widetilde{\cq}^\ast}
\r\|_{b_{\widetilde p(\cdot),\widetilde q(\cdot)}^{\widetilde \beta(\cdot),\widetilde\phi}(\rr^{n-1})}
\ls \|f\|_{B_{\widetilde p(\cdot),\widetilde q(\cdot)}^{\widetilde \beta(\cdot),\widetilde\phi}(\rr^{n-1})},
\end{eqnarray*}
which, together with Theorem \ref{t-atomd}, implies that
the summation in \eqref{trace-3} converges in $\cs'(\rn)$, $g\in\bbeve$
and $\|g\|_{\bbeve}
\ls\|f\|_{B_{\widetilde p(\cdot),\widetilde q(\cdot)}^{\widetilde \beta(\cdot),
\widetilde\phi}(\rr^{n-1})}$;
furthermore, Tr$(g)=f$ in $\cs'(\rr^{n-1})$. Therefore, Tr is surjective.
This finishes the proof of Theorem \ref{t-trace}.
\end{proof}

\bigskip

\noindent  Dachun Yang, Ciqiang Zhuo (Corresponding author) and Wen Yuan

\medskip

\noindent  School of Mathematical Sciences, Beijing Normal University,
Laboratory of Mathematics and Complex Systems, Ministry of
Education, Beijing 100875, People's Republic of China

\smallskip

\noindent {\it E-mails}: \texttt{dcyang@bnu.edu.cn} (D. Yang)

\hspace{0.98cm}\texttt{cqzhuo@mail.bnu.edu.cn} (C. Zhuo)

\hspace{0.98cm}\texttt{wenyuan@bnu.edu.cn} (W. Yuan)

\end{document}